\documentclass{article}

\usepackage[english]{babel}

\usepackage[letterpaper,top=2cm,bottom=2cm,left=3cm,right=3cm,marginparwidth=1.75cm]{geometry}

\usepackage{amsmath, amsfonts, amsthm, bbm, dsfont, mathrsfs, subcaption}
\usepackage{graphicx}
\usepackage[colorlinks=true, allcolors=blue]{hyperref}
\usepackage{cleveref}
\newtheorem{theorem}{Theorem}[section]
\newtheorem{definition}[theorem]{Definition}
\newtheorem{lemma}[theorem]{Lemma}

\newtheorem{assumption}[theorem]{Assumption}

\newcommand{\fc}{{\mathfrak c}}

\newlength\figureheight
\newlength\figurewidth

\def\RR{{\mathbb R}}

\newcommand\prob{\mathbb{P}}
\newcommand\expect{\mathbb{E}}

\newcommand{\be}{\begin{equation}}
\newcommand{\ee}{\end{equation}}
\newcommand{\wt}{\widetilde}

\newcommand\xx{\mathbf{x}}
\newcommand\yy{\mathbf{y}}
\newcommand\uu{\mathbf{u}}
\newcommand\ut{\tilde{\mathbf{u}}}
\newcommand\uh{\hat{\mathbf{u}}}
\newcommand\vv{\mathbf{v}}
\newcommand\vt{\tilde{\mathbf{v}}}
\newcommand\vh{\hat{\mathbf{v}}}
\renewcommand\AA{\mathsf{A}}
\newcommand\BB{\mathsf{B}}
\newcommand\DD{\mathsf{D}}

\newcommand\ii{\mathrm{i}}
\newcommand\dd{\mathrm{d}}
\numberwithin{equation}{section}

\newcommand{\Sig}{\Sigma}
\newcommand{\eps}{\varepsilon}

\newcommand{\pbb}[1]{\biggl({#1}\biggr)}
\newcommand{\pBB}[1]{\Biggl({#1}\Biggr)}

\newcommand{\beq}{\begin{equation}}
\newcommand{\bEq}{\end{equation}}

\newcommand{\al}{\alpha}

\newcommand{\e}{{\varepsilon}}

\newcommand{\Ga}{{\Gamma}}

\newcommand{\si}{\sigma}
\newcommand{\tsig}{\widetilde{\sigma}}

\newcommand{\cQ}{{\cal Q}}

\usepackage{amsmath}
\usepackage{amssymb}
\usepackage{amsthm}

\def\RR{{\mathbb R}}

\newcommand{\bb}{\mathbb}

\renewcommand{\cal}{\mathcal}

\newcommand{\ul}[1]{\underline{#1} \!\,}

\newcommand{\ctQ}{\widetilde{\mathcal{Q}}}

\newcommand{\col}{\mathrel{\mathop:}}
\newcommand{\deq}{\mathrel{\mathop:}=}

\renewcommand{\epsilon}{\varepsilon}
\renewcommand{\leq}{\leqslant}
\renewcommand{\geq}{\geqslant}

\renewcommand{\le}{\leq}
\renewcommand{\ge}{\geq}

\newcommand{\E}{\mathbb{E}}
\newcommand{\R}{\mathbb{R}}
\newcommand{\C}{\mathbb{C}}
\newcommand{\N}{\mathbb{N}}
\newcommand{\Z}{\mathbb{Z}}

\newcommand{\p}[1]{({#1})}
\newcommand{\pb}[1]{\bigl({#1}\bigr)}
\newcommand{\pB}[1]{\Bigl({#1}\Bigr)}
\newcommand{\pa}[1]{\left({#1}\right)}

\newcommand{\qb}[1]{\bigl[{#1}\bigr]}
\newcommand{\qB}[1]{\Bigl[{#1}\Bigr]}
\newcommand{\qbb}[1]{\biggl[{#1}\biggr]}

\newcommand{\qa}[1]{\left[{#1}\right]}

\newcommand{\h}[1]{\{{#1}\}}

\newcommand{\abs}[1]{\lvert #1 \rvert}
\newcommand{\absb}[1]{\bigl\lvert #1 \bigr\rvert}

\newcommand{\absa}[1]{\left\lvert #1 \right\rvert}

\newcommand{\norm}[1]{\lVert #1 \rVert}

\newcommand{\normBB}[1]{\Biggl\lVert #1 \Biggr\rVert}
\newcommand{\norma}[1]{\left\lVert #1 \right\rVert}

\DeclareMathOperator{\tr}{Tr}

\DeclareMathOperator{\supp}{supp}

\DeclareMathOperator{\im}{Im}
\renewcommand{\Re}{{\mathrm{Re}}\,}
\renewcommand{\Im}{{\mathrm{Im}}\,}

\DeclareMathOperator{\OO}{O}
\DeclareMathOperator{\oo}{o}

\DeclareMathOperator{\bv}{\mathbf{v}}

\theoremstyle{plain}

\newcommand{\Tr}{\mbox{Tr\,}}
\newcommand{\f}[1]{\boldsymbol{\mathrm{#1}}}

\allowdisplaybreaks
\title{Local Laws and Outlier Eigenvalues for Spiked Separable Covariance Matrices}
\author{
  Zhili Wang\thanks{Qiuzhen College, Tsinghua University. E-mail: wang-zl21@mails.tsinghua.edu.cn}
  \and 
  Bin Qin\thanks{Department of Probability and Statistics, School of Mathematical Sciences, Peking University. E-mail: bqin25@stu.pku.edu.cn}
}

\begin{document}
\maketitle

\begin{abstract}
We prove local laws for the resolvents of separable covariance matrices of the form $\mathcal Q=A^{1/2}XBX^*A^{1/2}$, where $X=(x_{ij})$ is a $p\times n$ random matrix whose entries $x_{ij}$ are i.i.d.~random variables with mean 0 and variance $n^{-1}$, and $A,B$ are deterministic non-negative definite symmetric (or Hermitian) matrices. Following the method developed in \cite{isotropic}, we first establish a self-consistent equation for the resolvent of $\mathcal Q$ and use it to prove optimal local laws without the technical assumption $\expect[x_{ij}^{3}]=0$, which was essential in the previous derivation of the local laws in \cite{yang2019edge}. As an application of our local law, we compute the asymptotic distribution of the outlier eigenvalues for spiked separable covariance matrices, extending the corresponding result in \cite{bao2020statisticalinferenceprincipalcomponents}.
\end{abstract}

\section{Introduction}

In multivariate statistics, a fundamental objective is to infer the covariance structure of a centered random vector $\mathbf y\in \mathbb R^p$ from independent observations. Suppose that we observe $n$ independent copies $\mathbf y_1,\cdots,\mathbf y_n$, and write the (normalized) data matrix as $Y:=n^{-1/2}(\mathbf y_1,\cdots,\mathbf y_n)\in\mathbb R^{p\times n}$. A natural estimator for the population covariance matrix $\Sigma:=\mathbb E \mathbf y \mathbf y^*$ is the sample covariance matrix $\mathcal Q := YY^*$. When $p$ is fixed and $n\to\infty$, the matrix $\mathcal Q$ converges almost surely to $\Sigma$. In many modern applications, however, such as high-dimensional statistics \cite{IJ,IJ2,IJ2008,Peter2008,DT2011,ledoit2012nonlinear,Tony2012}, economics \cite{fan2008,Onatski2}, and population genetics \cite{Genetics, Bryc_2013,xu2022eigenvalueratioapproachinferring} the dimension $p$ is comparable to, or even larger than, the sample size $n$. In this regime, the classical asymptotic framework is no longer applicable. Throughout this paper, we work in the high-dimensional setting where $d_n:=p/n\in[\tau,\tau^{-1}]$ for some fixed constant $\tau>0$. Although $\mathcal Q$ is no longer a consistent estimator of $\Sigma$ in operator norm, its spectral statistics still contain useful information about the underlying covariance structure. 
It is helpful to interpret the rows of the data matrix as spatial locations and the columns as observation times. The classical correlated sample covariance model takes the form $Y=A^{1/2}X$, where $A$ is a deterministic $p\times p$ non-negative definite symmetric (or Hermitian) matrix describing spatial correlations, and $X=(x_{i\mu})$ is a $p\times n$ random matrix with $i.i.d.$ entries satisfying $\mathbb E x_{11}=0$ and $\mathbb E |x_{11}|^2=n^{-1}$. The associated sample covariance matrix is $\mathcal Q=A^{1/2}XX^*A^{1/2}$. 
The limiting empirical spectral distribution of such matrices is described by the deformed Marchenko--Pastur law \cite{MP}. Much finer properties, including local laws, eigenvalue rigidity, and Tracy--Widom fluctuations of the extreme eigenvalues, have also been established in considerable generality; see, for example, \cite{Isotropic_local,Anisotropic,principal,DY_AAP,BAO2024143}.

In many applications, the observations exhibit correlations not only across spatial locations but also across time. Such structured dependence appears naturally in environmental statistics \cite{MF2006,KJ1999,LMS2008,MG2003} and wireless communication systems \cite{985982,RMT_Wireless,5437443}. This motivates the separable covariance model
\[
Y=A^{1/2}XB^{1/2},
\]
where $A$ and $B$ are deterministic non-negative definite symmetric (or Hermitian) matrices of dimensions $p\times p$ and $n\times n$, respectively. Neither $A$ nor $B$ is assumed to be diagonal. Thus, the model allows general spatial correlations through $A$ and general temporal correlations through $B$. The term {\it separable} reflects the fact that, in the Gaussian case, the covariance matrix of the vectorized data is $A\otimes B$, so that the spatial and temporal covariance structures factorize. The corresponding separable covariance matrix is
\[
\mathcal Q:=YY^*=A^{1/2}XBX^*A^{1/2}.
\]
Separable covariance matrices have been studied in a variety of contexts, including wireless communications, where spectral quantities are closely related to channel capacity \cite{Verdu,RMT_Wireless}. Their global spectral behavior and deterministic equivalents were investigated in \cite{Hachem2007,Separable,Separable_solution}, among other works; see also \cite{PRE,Karoui2009,WANG2014,Zhang_thesis}.

The main object of the present paper is the local spectral behavior of separable covariance matrices. Local laws give optimal-scale estimates for resolvents and provide the basic input for eigenvalue rigidity, edge universality, and the analysis of spiked models. When both $A$ and $B$ are diagonal, optimal local laws for random Gram matrices were proved in \cite{Alt_Gram,AEK_Gram}. For general non-diagonal $A$ and $B$, Yang \cite{yang2019edge} established optimal local laws and edge universality for separable covariance matrices, but the proof of the optimal local laws in the non-diagonal case required the technical vanishing third moment assumption
\begin{align}
  \expect[x_{11}^3]=0. \label{eq_3rdmoment}
\end{align}
This condition was introduced in \cite{yang2019edge} as a technical assumption required for the proof of anisotropic local laws via a self-consistent Green's function comparison argument. However, it is not believed to be an intrinsic feature of the separable covariance model. Indeed, in certain special cases where one of the population covariance matrices is diagonal, the vanishing third-moment condition can be removed by adapting the argument of \cite{Anisotropic}. For general non-diagonal matrices $A$ and $B$, however, the interaction between the two deterministic eigenvector bases makes the contributions arising from nonzero third moments substantially more difficult to control.

Our first contribution is to remove this assumption \eqref{eq_3rdmoment}. More precisely, we prove the optimal anisotropic and averaged local laws for general separable covariance matrices under a bounded support condition on the entries of $X$, without imposing \eqref{eq_3rdmoment}. Instead of relying on a Green's function comparison argument that requires third-moment matching, we follow the approach of \cite{isotropic} and control the relevant error terms directly through cumulant expansions and isotropic self-consistent equations.
This strengthens the local-law input in \cite{yang2019edge} and, consequently, removes the third-moment restriction from results whose proofs relied on those local laws.

As an application, we study spiked separable covariance matrices. Spiked models are designed to capture low-rank signal structures in high-dimensional data \cite{xu2022eigenvalueratioapproachinferring} and have been extensively studied for classical sample covariance matrices; see, for example, \cite{BBPstarter,bai2006,paul2007,bai2008,bai2012,principal}.
In the separable setting, spikes may appear in either the spatial covariance matrix or the temporal covariance matrix. Both types of spikes can generate outlier eigenvalues of the separable covariance matrix, while their associated eigenvectors carry different information about the underlying spatial and temporal signal directions. Ding and Yang \cite{ding2020spikedseparablecovariancematrices} obtained precise outlier locations, convergence rates, and eigenvector behavior for spiked separable covariance matrices, using the local laws and rigidity estimates available for the non-spiked model. Their results identify the supercritical spikes, describe the deterministic locations of the associated outlier eigenvalues, and distinguish the eigenvector behavior associated with spatial and temporal spikes. The second-order fluctuation distribution of these outliers, however, was not derived there. 
Moreover, some of the results in \cite{ding2020spikedseparablecovariancematrices} are not optimal, as they are based on the local laws established in \cite{yang2019edge}. As discussed above, the optimal local laws in that work were derived under the third-moment condition \eqref{eq_3rdmoment}.

Our second contribution is to establish the limiting distribution of the supercritical outlier eigenvalues of spiked separable covariance matrices without imposing the condition \eqref{eq_3rdmoment}. To the best of our knowledge, this is the first such distributional result for the general spiked separable model. Compared with the classical spiked sample covariance model, the separable model contains additional deterministic structure from the temporal covariance matrix $B$, and the resulting fluctuation variances depend on both population covariance matrices and on the alignment of the spike eigenvectors with the deterministic eigenbases. 
The proof relies crucially on the local laws established in this paper, which allow us to reduce the fluctuations of the outlier eigenvalues to certain resolvent functionals of the non-spiked model. These functionals are then analyzed using the local laws and a further cumulant expansion argument, following the approach of \cite{bao2020statisticalinferenceprincipalcomponents,Bao2021AOS}. In particular, when $A=I_p$ and $B=I_n$, our formula recovers the corresponding outlier fluctuation result for the classical spiked covariance model obtained in \cite{bao2020statisticalinferenceprincipalcomponents}.

\medskip

The rest of the paper is organized as follows. In Section \ref{main_result}, we introduce the non-spiked and spiked separable covariance models and state the main results: the local law, Theorem \ref{local_law}, for the non-spiked model, and the asymptotic distribution of the outlier eigenvalues, Theorem \ref{distribution of outliers}, for the spiked model. The latter relies crucially on the local laws established in Theorem \ref{local_law}. 
Section \ref{sec:stoch_step} is devoted to the stochastic step in the proof of Theorem \ref{local_law}, based on the method of cumulant expansions. In Section \ref{sec:det_step}, we carry out the deterministic step and complete the proof of Theorem \ref{local_law}. 
Finally, in Section \ref{sec:pf_outlier_distribution}, we prove Theorem \ref{distribution of outliers} by applying the local laws to the spiked model. We first reduce the outlier problem to resolvent functionals and then establish their asymptotic distribution through a cumulant expansion argument.   

\subsection*{Acknowledgement}  
We would like to thank Fan Yang for suggesting the problem, providing valuable guidance, and commenting on the draft of the paper. 
AI tools were used only for polishing the language of this paper. All mathematical results, arguments, and computations were developed, written, and checked by the authors.

\section{Model and Main Results}\label{main_result}

\subsection{Spiked separable covariance matrices}
{}We consider a class of separable sample covariance matrices of the form $\mathcal Q_1:=A^{1/2}XBX^*A^{1/2}$, where $A$ and $B$ are deterministic non-negative definite real symmetric matrices. Note that $A$ and $B$ are not necessarily diagonal. We assume that $X=(x_{i\mu})_{i\in\cal I_1,\mu\in\cal I_2}$ is a $p\times n$ random matrix, where $\cal I_1:=\{1,\cdots,p\}$ and $\cal I_2:=\{1,\cdots,n\}$ are index sets, and the entries $x_{i\mu}$ are real or complex independent random variables satisfying
\begin{equation}\label{eq_124moment}
\mathbb{E} x_{i\mu} =0, \quad \mathbb{E} \abs{x_{i\mu}}^2  = n^{-1},\quad \max_{i,\mu}\mathbb{E}\abs{x_{i\mu}}^4 \leq C_4n^{-2},  
\end{equation}
for some constant $C_4>0$. For definiteness, in this paper we focus on the real case, that is, the random variables $x_{i\mu}$ are real. However, our proof can be applied to the complex case after minor modifications if we assume in addition that $\Re x_{i\mu}$ and $\Im x_{i\mu}$ are independent centered random variables with variance $(2n)^{-1}$. 
When discussing applications of our main results—namely, local laws for the resolvents of separable covariance matrices—to the distribution of outliers (Theorem \ref{distribution of outliers}), we shall additionally assume that the random variables $x_{i\mu}$ possess moments of arbitrarily high order. More precisely, for any fixed $k\in \mathbb N$, there exists a constant $C_k>0$ such that
\begin{equation}\label{eq_highmoment} 
\max_{i,\mu}\expect\abs{x_{i\mu}}^k \le C_k n^{-k/2} 
\end{equation}
for all $n$. This is a convenient assumption that is widely used in applications. However, for the sake of generality, we impose only the weaker bounded support condition \eqref{eq_support} when establishing the local law (Theorem \ref{local_law}).

We will also use the $n \times n$ matrix $\mathcal Q_2:=B^{1/2}X^* A X B^{1/2}$. 
We denote the eigenvalues of $\mathcal Q_1$ and $\mathcal Q_2$ in descending order by $\lambda_1(\mathcal Q_1)\geq \cdots \geq \lambda_{p}(\mathcal Q_1)$ and $\lambda_1(\mathcal Q_2) \geq \cdots \geq \lambda_n(\mathcal Q_2)$. Since $\mathcal Q_1$ and $\mathcal Q_2$ share the same nonzero eigenvalues, we will simply write $\lambda_j$, $1\le j \le p \wedge n$, to denote the $j$-th eigenvalue of both $\mathcal Q_1$ and $\mathcal Q_2$ without causing any confusion. 
We shall consider the high-dimensional setting in this paper. More precisely, we assume that there exists a constant $0<\tau <1$ such that the aspect ratio $d_n:= p/n$ satisfies 
\begin{equation}
 \tau \le d_n \le \tau^{-1} \ \ \text{ for all } n. \label{eq_ratio}
 \end{equation}
We assume that $A$ and $B$ have eigendecompositions
\be\label{eigen}
A= V^a\Sig^a (V^a)^*, \quad  B= V^b \Sigma^b (V^b)^* ,
\ee
where
$$\Sig^a=\text{diag}(\si_1^a, \ldots, \si_p^a), \quad V^a= (\bv^a_1, \cdots, \bv^a_p),\quad \text{and}\quad  \Sig^b=\text{diag}( \si_1^b, \ldots,  \si_n^b), \quad V^b= (\bv^b_1, \cdots, \bv^b_n). $$ 
We denote the empirical spectral distributions (ESD) of $A$ and $B$ by
\begin{equation}\label{sigma_ESD}
\pi_A\equiv \pi_A^{(p)} := \frac{1}{p} \sum_{i\in\cal I_1} \delta_{\si_i^a} \ ,\quad \pi_B\equiv \pi_B^{(n)} := \frac{1}{n} \sum_{\mu\in\cal I_2} \delta_{\si_\mu^b}\ .
\end{equation}
We assume that there exists a small constant $0<\tau<1$ such that for all $n$ large enough,
\begin{equation}\label{assm3}
\max\left\{\si_i^a,  \sigma_\mu^b:i\in\cal I_1,\mu\in\cal I_2 \right\} \le \tau^{-1}, \quad \max\left\{\pi_A^{(p)}([0,\tau]), \pi_B^{(n)}([0,\tau])\right\} \le 1 - \tau .
\end{equation}
Note that the first condition means that the operator norms of $A$ and $B$ are bounded by $\tau^{-1}$, and the second condition means that the spectra of $A$ and $B$ cannot concentrate at zero.

To introduce spikes, we follow the setup in \cite{ding2017} and assume there exist fixed integers $r,s\in\N$ and constants $d_i^a$ for $i\in\cal I_1^+:=\{1,\ldots,r\}\subset\cal I_1$, and $d_\mu^b$ for $\mu\in\cal I_2^+:=\{1,\ldots,s\}\subset\cal I_2$, such that
\begin{equation}\label{eq_defnsigmaa}
\begin{split}
\wt A= V^a\wt \Sig^a (V^a)^*, \quad & \wt B= V^b \wt \Sigma^b (V^b)^* ,\\
 \wt\Sig^a=\text{diag}(\wt\si_1^a, \ldots, \wt\si_p^a), \quad & \wt\Sig^b=\text{diag}( \wt\si_1^b, \ldots,  \wt\si_n^b),
 \end{split}
\end{equation}
where 
  \begin{equation}\label{eq_defnsigmab}
\widetilde{\sigma}_i^a=
 \begin{cases}
 \sigma_i^a(1+d^a_i),   & 1 \leq i \leq r \\
 \sigma_i^a,  & \text{otherwise}
 \end{cases},\qquad \widetilde{\sigma}_\mu^b=
 \begin{cases}
 \sigma_\mu^b(1+d^b_\mu),   & 1 \leq \mu \leq s \\
 \sigma_\mu^b,  & \text{otherwise}
 \end{cases}.
 \end{equation} 
 Without loss of generality, we assume that the indices are reordered such that 
 \be\label{reorder} \wt\si_1^a \ge \wt\si_2^a \ge \ldots \ge \wt\si_p^a \ge 0 \ , \quad  \wt\si_1^b \ge \wt\si_2^b \ge \ldots \ge  \wt\si^b_n \ge 0 \ .\ee
 Moreover, we assume that 
 \begin{equation}\label{assm33}
\max\{\widetilde\si_1^a,  \widetilde\sigma_1^b \} \le \tau^{-1} .
\end{equation}
With (\ref{eq_defnsigmaa}) and (\ref{eq_defnsigmab}), we can write   
\be\label{AOBO} 
\begin{split}
& \widetilde A = A\Big(I_p+{V_o^a} {D}^{a} (V_o^a)^*\Big)=\Big(I_p+ V_o^{a} D^a(V_o^a)^*\Big)A, \\ 
& \widetilde B= B\Big(I_n+V_o^{b} {D}^{b} {(V_o^b)}^*\Big)=\Big(I_n+V_o^{b} {D}^{b} {(V_o^b)}^*\Big)B,
\end{split}
\ee
where 
$${D}^{a}=\text{diag}(d_1^{a},\cdots, d_r^a), \quad V_o^a=(\bv_1^a,\cdots, \bv^a_r),\quad \text{and}\quad {D}^{b}=\text{diag}(d_1^{b},\cdots, d_s^b), \quad V_o^b=(\bv_1^b,\cdots, \bv^b_s).$$
We define the spiked separable sample covariance matrices as
\begin{equation}\label{eq_sepamodel}
\widetilde{\mathcal{Q}}_1=\widetilde A^{1/2} X \widetilde B X^* \widetilde A^{1/2}, \quad  \widetilde{\mathcal{Q}}_2=\widetilde B^{1/2} X^* \widetilde A X \widetilde B^{1/2}.
\end{equation} 

We summarize the basic assumptions used below. To establish local laws for the resolvents of (non-spiked) separable sample covariance matrices, we make the following assumption.

\begin{assumption}\label{assm_big1}
Assume $X$ is a $p\times n$ random matrix with real i.i.d.~entries satisfying \eqref{eq_124moment}, $A$ and $B$ are deterministic non-negative definite symmetric matrices satisfying \eqref{eigen} and \eqref{assm3}, and $d_n$ satisfies \eqref{eq_ratio}.
\end{assumption}

For the spiked model, we additionally assume the high-moment condition \eqref{eq_highmoment}.

\begin{assumption}\label{assm_big1'}
Assume $X$ is a $p\times n$ random matrix with real i.i.d. entries satisfying \eqref{eq_124moment} and \eqref{eq_highmoment}. Moreover, we denote by $\kappa_3$ and $\kappa_4$ the third and fourth cumulants, respectively, of the entries of $X$.
In addition, assume that $A$ and $B$ are deterministic, non-negative definite symmetric matrices satisfying \eqref{eigen} and \eqref{assm3}, that \smash{$\widetilde A$} and \smash{$\widetilde B$} are deterministic, non-negative definite symmetric matrices satisfying \eqref{eq_defnsigmaa}, \eqref{eq_defnsigmab}, \eqref{reorder}, and \eqref{assm33}, and that $d_n$ satisfies \eqref{eq_ratio}.
\end{assumption}

\subsection{Resolvents and limiting law}

We study the eigenvalue statistics of $\mathcal Q_{1}$, $\mathcal Q_{2}$ and $\wt {\mathcal Q}_{1}$, $\wt {\mathcal Q}_{2}$ through their {\it{resolvents}} (or  {\it{Green's functions}}). 

\begin{definition}[Resolvents]\label{defn_resolvent}
For $z = E+ \ii \eta \in \mathbb C_+,$ we define the following resolvents for $\al=1,2$: 
\begin{equation}\label{def_green}
 \mathcal G_{\al}(X,z):=\left({\mathcal Q}_{\al}(X) -z\right)^{-1} , \ \ \ \wt{\mathcal G}_{\al} (X,z):=(\wt{\mathcal Q}_{\al}(X)-z)^{-1} .
\end{equation}
 We denote the ESD $\rho^{(p)}$ of ${\mathcal Q}_{1}$ and its Stieltjes transform as
\be\label{defn_m}
\rho^{(p)} := \frac{1}{p} \sum_{i\in\cal I_1} \delta_{\lambda_i({\mathcal Q}_1)},\quad m^{(p)}(z):=\int \frac{\rho^{(p)}(\dd x)}{x-z}=\frac{1}{p} \mathrm{Tr} \, \mathcal G_1(z).
\ee
We also introduce the following quantities:
\begin{equation}\label{defn_m1m2}
 m_1^{(n)}(z):= \frac{1}n\tr \left(A \mathcal G_1(z)\right) ,\quad m_2^{(n)}(z):=\frac{1}{n}\tr\left( B\mathcal G_2(z)\right). 
\end{equation}
\end{definition}

 It was shown in \cite{Separable} that if $d_n \to d \in (0,\infty)$ and $\pi_A^{(p)}$, $\pi_B^{(n)}$ converge to fixed probability distributions, then almost surely $\rho^{(p)}$ converges to a deterministic limit law $\rho_{\infty}$ as $n\to\infty$. We now give its definition. 
For any finite $n$, $p=nd_n$, and $z\in \mathbb C_+$, we define \smash{$(m^{(n)}_{1c}(z),m^{(n)}_{2c}(z))\in \mathbb C_+^2$} as the unique solution to the following system of self-consistent equations
\begin{equation}\label{separa_m12}
\begin{split}
& {m^{(n)}_{1c}(z)} = d_n \int\frac{x}{-z\left[1+xm^{(n)}_{2c}(z) \right]} \pi_A^{(p)}(\dd x) ,\\ 
& {m^{(n)}_{2c}(z)} =  \int\frac{x}{-z\left[1+xm^{(n)}_{1c}(z) \right]} \pi_B^{(n)}(\dd x) .
\end{split}
\end{equation}
Then we define
\begin{equation}\label{def_mc}
m_c(z)\equiv m_c^{(p)}(z):= \int\frac{1}{-z\left[1+xm^{(n)}_{2c}(z) \right]} \pi_A^{(p)}(\dd x).
\end{equation}
It is easy to verify that $m_c(z)\in \mathbb C_+$ for $z\in \mathbb C_+$. Letting $\eta \to 0^+$, we obtain a probability measure $\rho_{c}^{(p)}$ through the inversion formula
\begin{equation}\label{ST_inverse}
\rho_{c}^{(p)}(E) =\frac{1}{\pi} \lim_{\eta\to 0^+} \Im m^{(p)}_{c}(E+\ii \eta).
\end{equation}
If $d_n \to d \in (0,\infty)$ and $\pi_A^{(p)}$, $\pi_B^{(n)}$ converge to fixed probability distributions, {then $\rho_{c}^{(p)}$ converges weakly as $n\to \infty$, and its weak limit is $\rho_\infty$.}

The above definitions of $m_c^{(p)}$, $\rho_c^{(p)}$, and $\rho_\infty$ make sense due to the following theorem. Throughout the rest of this paper, we often omit the superscripts $(p)$ and $(n)$ from our notations for simplicity.

\begin{theorem} [Existence, uniqueness, and continuous density]
For any $z\in \mathbb C_+$, there exists a unique solution $(m_{1c},m_{2c})\in \mathbb C_+^2$ to the system of equations in (\ref{separa_m12}). The function $m_c$ in (\ref{def_mc}) is the Stieltjes transform of a probability measure $\mu_c$ supported on $\mathbb R^+$. Moreover, $\mu_c$ has a continuous derivative $\rho_c(x)$ on $(0,\infty)$.
\end{theorem}
\begin{proof}
See {\cite[Theorem 1.2.1]{Zhang_thesis}}, {\cite[Theorem 2.4]{Hachem2007}} and {\cite[Theorem 3.1]{Separable_solution}}.
\end{proof}

From (\ref{separa_m12}), it is easy to see that if we define the function
\begin{equation}\label{separable_MP}
f(z,m):=- m + \int\frac{x}{-z+xd_n \int\frac{t}{1+tm} \pi_A(\dd t)} \pi_B(\dd x) ,
\end{equation}
then $m_{2c}(z)$ can be characterized as the unique solution to the equation $f(z,m)=0$ that satisfies $\Im \, m> 0$ for $z\in \mathbb C_+$, and $m_{1c}(z)$ can be defined using the first equation in \eqref{separa_m12}.
Moreover,  {$m_{1c}(z)$ and $m_{2c}(z)$} are the Stieltjes transforms of densities {$\rho_{1c}$ and $\rho_{2c}$} defined as:
\begin{equation}\label{eq_inverse}
 \rho_{\al c}(E) = \frac{1}{\pi} \lim_{\eta\to 0^+} \Im m_{\al c}(E+\ii \eta),\quad \al=1,2.
\end{equation}
Then we have the following result.

\begin{lemma}\label{lambdar}
The densities $\rho_{c}$, $\rho_{1c}$, and $\rho_{2c}$ share the same support on $(0,\infty)$, which is a union of intervals. For $\al=1,2$,
\begin{equation}\label{support_rho1c}
{\rm{supp}} \, \rho_{c} \cap (0,\infty) ={\rm{supp}} \, \rho_{\al c} \cap (0,\infty) = \bigcup_{k=1}^{L} [e_{2k}, e_{2k-1}] \cap (0,\infty),
\end{equation}
where $L\in \mathbb N$ depends only on $\pi_{A,B}$. The points $(x,m)=(e_k, m_{2c}(e_k))$ are the real solutions to
\begin{equation}
f(x,m)=0, \quad \text{and} \quad \frac{\partial f}{\partial m}(x,m) = 0. \label{equationEm2}
\end{equation}
Finally, $e_1 = {\rm O}(1)$, $m_{1c}(e_1) \in (-(\max_{\mu}\sigma_\mu^b)^{-1}, 0)$, and $m_{2c}(e_1) \in (-(\max_i \sigma_{i}^a)^{-1}, 0)$.
\end{lemma}
\begin{proof}
See Section 3 of \cite{Separable_solution}.
\end{proof}

We shall call $e_k$ the spectral edges. In particular, we focus on the rightmost edge $\lambda_+ := e_1$. Now we make the following assumption. It guarantees a regular square-root behavior of the spectral densities {$\rho_{1c}$ and $\rho_{2c}$} near $\lambda_+$ and rules out the existence of outliers.

\begin{assumption} \label{ass:unper} 
There exists a constant $\tau>0$ such that 
\begin{equation}\label{assm_gap}
1+m_{1c}(\lambda_+) \max_\mu \sigma_\mu^b \geq \tau, \quad 1+m_{2c}(\lambda_+) \max_i \sigma_i^a \geq \tau.
\end{equation}
\end{assumption}

\subsection{Notations}

{}

The fundamental large parameter in this paper is $n$, and we always assume that $p$ is comparable to $n$. All quantities that are not explicitly constant may depend on $n$, and we usually omit $n$ from our notations. We use $C$ to denote a generic large positive constant, whose value may change from one line to the next. Similarly, we use $\epsilon$, $\tau$, $\delta$, and $c$ to denote generic small positive constants. If a constant depends on a quantity $a$, we use $C(a)$ or $C_a$ to indicate this dependence. We use $\tau>0$ in various assumptions to denote a small positive constant. All constants appearing in the statements or proofs may depend on $\tau$; we neither indicate nor track this dependence.
For two quantities $a_n$ and $b_n$ depending on $n$, the notation $a_n = \OO(b_n)$ means that $|a_n| \le C|b_n|$ for some constant $C>0$, and $a_n=\oo(b_n)$ means that $|a_n| \le c_n |b_n|$ for some positive sequence $c_n\downarrow 0$ as $n\to \infty$. We also use the notations $a_n \lesssim b_n$ if $a_n = \OO(b_n)$, and $a_n \sim b_n$ if $a_n = \OO(b_n)$ and $b_n = \OO(a_n)$. For a matrix $A$, we use $\|A\|:=\|A\|_{l^2 \to l^2}$ to denote the operator norm; for a vector $\mathbf v=(v_i)_{i=1}^n$, $\|\mathbf v\|\equiv \|\mathbf v\|_2$ stands for the Euclidean norm, while $|\mathbf v| \equiv \|\mathbf v\|_1$ stands for the $l^1$-norm. In this paper, we often write an identity matrix as $I$ or $1$ without causing any confusion. We will use $\left\langle\uu,\vv\right\rangle:=\uu^*\vv$ to denote the inner product of two vectors $\uu$ and $\vv$.  

We use $\R$, $\C$, $\Z$, and $\N$ to denote the sets of real numbers, complex numbers, integers, and natural numbers, respectively. Note that the set of natural numbers includes zero, i.e., $\mathbb{N}=\{0,1,2,3,\ldots\}$, while $\mathbb{N}_+=\{1,2,3,\ldots\}$ denotes the set of positive integers. The upper half of the complex plane and the positive, negative, and non-negative real lines are denoted by
$$\mathbb C_+:=\{z\in \mathbb C:\ \mathrm{Im}\, z>0\},\quad \mathbb R_{>0}:=(0,+\infty),\quad \mathbb R_{<0}:=(-\infty,0),\quad \mathbb R_+:=[0,\infty).$$
We will use the following notion of stochastic domination, which was first introduced in \cite{Average_fluc} and subsequently used in many works on random matrix theory, such as \cite{isotropic,principal,local_circular,Delocal,Semicircle,Anisotropic}. It simplifies the presentation of the results and proofs by systematizing statements of the form ``$\xi$ is bounded by $\zeta$ with high probability up to a small power of $n$".

\begin{definition}[Stochastic domination]\label{stoch_domination}
(i) Let
\[\xi=\left(\xi^{(n)}(u):n\in\bb N, u\in U^{(n)}\right),\quad \zeta=\left(\zeta^{(n)}(u):n\in\bb N, u\in U^{(n)}\right)\]
be two families of non-negative random variables, where $U^{(n)}$ is a parameter set depending on $n$. We say that $\xi$ is stochastically dominated by $\zeta$, uniformly in $u$, if for any fixed (small) $\epsilon>0$ and (large) $D>0$, 
\[\sup_{u\in U^{(n)}}\bb P\left[\xi^{(n)}(u)>n^\epsilon\zeta^{(n)}(u)\right]\le n^{-D}\]
holds for large enough $n\ge n_0(\epsilon, D)$, denoted by $\xi\prec\zeta$. Throughout this paper, the stochastic domination will always be uniform in all parameters that are not explicitly fixed (such as matrix indices, and $z$ that takes values in some compact set). Note that $n_0(\epsilon, D)$ may depend on quantities that are explicitly constant, such as $\tau$ in Assumption \ref{assm_big1} and \eqref{assm_gap}. If for some complex family $\xi$ we have $|\xi|\prec\zeta$, then we will also write $\xi \prec \zeta$ or $\xi=\OO_\prec(\zeta)$.
\medskip

\noindent (ii) We extend the definition of $\OO_\prec(\cdot)$ to matrices in the weak operator sense as follows. Let $A$ be a family of random matrices and $\zeta$ be a family of nonnegative random variables. Then $A=\OO_\prec(\zeta)$ means that $\left|\left\langle\mathbf v, A\mathbf w\right\rangle\right|\prec\zeta \| \mathbf v\|_2 \|\mathbf w\|_2 $ uniformly in any deterministic vectors $\mathbf v$ and $\mathbf w$. Here and throughout the following, whenever we say ``uniformly in any deterministic vectors", we mean that ``uniformly in any deterministic vectors belonging to certain fixed set of cardinality $n^{\OO(1)}$".

\medskip

\noindent (iii) We say that an event $\Xi$ holds with high probability if $\mathbbm{1}_{\Xi^c}\prec0$, or equivalently, for any constant $D>0$, $\mathbb P(\Xi)\ge 1- n^{-D}$ for large enough $n$.
\end{definition}

The following lemma collects several basic properties of stochastic domination $\prec$, which will be used tacitly in the proof.

\begin{lemma}[Lemma 3.2 in \cite{Isotropic_local}]\label{lem_stodomin}
Let $\xi$ and $\zeta$ be families of nonnegative random variables.
\begin{itemize}
  \item[(i)] Suppose that $\xi (u,v)\prec \zeta(u,v)$ uniformly in $u\in U$ and $v\in V$. If $|V|\le n^C$ for some constant $C$, then $\sum_{v\in V} \xi(u,v) \prec \sum_{v\in V} \zeta(u,v)$ uniformly in $u$.

\item[(ii)] If $\xi_1 (u)\prec \zeta_1(u)$ and $\xi_2 (u)\prec \zeta_2(u)$ uniformly in $u\in U$, then $\xi_1(u)\xi_2(u) \prec \zeta_1(u)\zeta_2(u)$ uniformly in $u$.

\item[(iii)] Suppose that $\psi(u)\ge n^{-C}$ is deterministic and $\xi(u)$ satisfies $\mathbb E[\xi(u)^2] \le n^C$ for all $u$. Then if $\xi(u)\prec \psi(u)$ uniformly in $u$, we have $\mathbb E\xi(u) \prec \psi(u)$ uniformly in $u$.
\end{itemize}
\end{lemma}

The following lemma provides a useful way to establish stochastic domination via moment bounds.

\begin{lemma}\label{momentmethod}
Let $\{\xi(u),u\in U\}$ be a family of random variables with $|U|\leq n^C$ for some constant $C$, and let $\psi$ be a non-negative deterministic parameter. If
$$\mathbb{E}\left[\abs{\xi(u)}^{2k}\right]\prec\psi^{2k}$$
uniformly in $u$ holds for all $k\in\mathbb{N}_+$, then we have $\max_{u\in U}\abs{\xi(u)}\prec\psi$.
\end{lemma}
\begin{proof}
For any $\eps,D>0$, choose $k\in\N_+$ such that $k\geq\frac{C+D+1}{2\e}$. Then, by Markov's inequality we have
$$\prob\qB{\max_{u\in U}\abs{\xi(u)}>n^{\eps}\psi}
\leq\sum_{u\in U}\prob\qb{\abs{\xi(u)}>n^{\eps}\psi}
\leq\sum_{u\in U}\frac{\expect\qb{\abs{\xi(u)}^{2k}}}{n^{2k\eps}\psi^{2k}}
\leq n^{C+1-2k\eps}\leq n^{-D}$$
holds for sufficiently large $n\geq n_0(\e,D,k)$.
\end{proof}

We now introduce the bounded support condition, which is slightly more general than the high-moment assumption 
\eqref{eq_highmoment}.

\begin{definition}[Bounded support condition] \label{defn_support}
A random matrix $X$ is said to satisfy the {\it bounded support condition} with parameter $q$ if
\begin{equation}
\max_{i,\mu}\abs{x_{i\mu}} \prec q, \label{eq_support}
\end{equation}
Here $q$ is a deterministic parameter, typically satisfying \(n^{-1/2} \le q \le n^{-c_\phi}\) for some small constant \(c_\phi>0\). Whenever \eqref{eq_support} holds we say that $X$ has support $q$. In particular, under the arbitrarily high-moment assumption (\ref{eq_highmoment}) one may take $q=n^{-1/2}$ by Lemma \ref{momentmethod}.
\end{definition}

Next, we introduce a convenient self-adjoint linearization trick, which has been proved to be useful in studying the local laws of random matrices of the Gram type \cite{Alt_Gram, AEK_Gram, Anisotropic, XYY_circular,yang2019edge}. We define the following $(p+n)\times (p+n)$ self-adjoint block matrix, which is a linear function of $X$:
 \begin{equation}\label{linearize_block}
   H \equiv H(X,z): = z^{1/2} \left( {\begin{array}{*{20}c}
   { 0 } &A^{1/2} X B^{1/2}   \\
   {B^{1/2} X^* A^{1/2} } & {0}  \\
   \end{array}} \right),  \quad z\in \C\backslash\R_{<0} .
 \end{equation}
where $z^{1/2}$ is taken to be the branch cut with positive imaginary part. Then we define its resolvent (or Green's function) as
 \begin{equation}\label{eq_gz}
 G \equiv G (X,z):= \left(H(X,z)-z\right)^{-1} .
 \end{equation}
By Schur complement formula, we can verify that (recall definition (\ref{def_green}) above)
\begin{align} 
G(z) = \left( {\begin{array}{*{20}c}
   { \mathcal G_1} & z^{-1/2}\mathcal G_1 Y \\
   {z^{-1/2}Y^* \mathcal G_1} & { \mathcal G_2 }  \\
\end{array}} \right) = \left( {\begin{array}{*{20}c}
   { \mathcal G_1} & z^{-1/2}Y \mathcal G_2   \\
   {z^{-1/2} \mathcal G_2 Y^*} & { \mathcal G_2 }  \\ 
\end{array}} \right),\label{green2} 
\end{align}
where $Y:= A^{1/2}X B^{1/2}$. 
Thus a control of $G$ yields directly a control of the resolvents $\mathcal G_{\al}$, $\al\in\{1,2\}$. 

The index set of entries of these $(p+n)\times(p+n)$ matrices is naturally given by $\cal I\equiv \cal I_1\sqcup\cal I_2:=(\cal I_1\times\{1\})\cup(\cal I_2\times\{2\})$. For simplicity of notation, we identify $\cal I$ with $\{1,\ldots,p+n\}$, with $\cal I_1=\{1,\ldots,p\}$ and $\cal I_2=\{p+1,\ldots,p+n\}$. Throughout this paper, we consistently use latin letters $i,j$ for indices in $\cal I_1$ or $\cal I_1^+= \{1,\ldots,r\}$, greek letters $\mu,\nu$ for indices in $\cal I_2$ or $\cal I_2^+=\{1,\ldots,s\}$, and $a,b$ for indices in $\cal I$. For example, $G_{i\mu}$ refers to an entry in the upper right block of $G$.
Denote the unit sphere in $\C^{p+n}$ by
$$\mathbf{S} = \{\mathbf u\in \C^{p+n}:\left \| \mathbf u \right \| =1\}.$$
Given a vector $\uu = (u_1, u_2,\cdots,u_{p+n})\in \C^{p+n}$, let
\begin{equation}\label{eq_projection}
\uh:= (u_1,u_2,\cdots,u_{p},0,\cdots,0),\quad \ut:= (0,\cdots,0,u_{p+1},u_{p+2},\cdots,u_{p+n})
\end{equation}
be its orthogonal projection onto the first $p$ coordinates and the last $n$ coordinates, respectively. For vectors $\uu,\vv\in \C^{p+n}$ and matrix $X\in \C^{(p+n)\times(p+n)}$, we abbreviate 
$$X_{\uu\vv} := \uu^{*}X\vv,\quad X_{a\vv} := e_a^{*}X\vv,\quad X_{\uu a} := \uu^{*}Xe_a,\quad\text{and}\quad\underline{X}:=\frac{1}{p+n}\Tr X,$$
where $e_a$ denotes the $a$-th standard basis vector in $\C^{p+n}$. 

Finally, it is convenient to introduce the following notion of asymptotic equality in distribution.

\begin{definition}\label{defn_asymptotic}
    Two sequences of random vectors $\mathsf{X}_N,\mathsf{Y}_N\in\R^k$ are called asymptotically equal in distribution, denoted as $\mathsf{X}_N\simeq\mathsf{Y}_N$, if they are tight and satisfy
    $$\lim_{N\to\infty}\expect\qb{f(\mathsf{X}_N)-f(\mathsf{Y}_N)}=0$$
    for any bounded continuous function $f:\R^k\to\R$.
\end{definition}

\subsection{Main result: local law}

For any constants $c_0,C_0>0$ and $\omega \le 1$, we define a domain of the spectral parameter $z$ as
\begin{equation}
S(c_0,C_0,\omega):= \left\{z=E+ \ii \eta: \lambda_+ - c_0 \leq E \leq C_0 \lambda_+, n^{-1+\omega} \leq \eta \leq C_0 \right\}. \label{SSET1}
\end{equation}
In particular, we shall denote
\begin{equation}
S(c_0,C_0,-\infty):= \left\{z=E+ \ii \eta: \lambda_+ - c_0 \leq E \leq C_0 \lambda_+, 0 \leq \eta \leq C_0 \right\}.
\end{equation}
We define the distance to the rightmost edge as
\begin{equation}
\kappa \equiv \kappa_E := \vert E -\lambda_+\vert  \quad \text{for } z= E+\ii \eta.\label{KAPPA}
\end{equation}
We have the following lemma, which summarizes some basic properties of $m_{1,2c}$ and $\rho_{1,2c}$.

\begin{lemma}\label{lem_mbehavior}
Suppose the assumptions \eqref{eq_ratio}, \eqref{assm3}, and \eqref{assm_gap} hold. Then
there exists sufficiently small constant $\wt c>0$ such that the following estimates hold:
\begin{itemize}
\item[(1)] We have
\begin{equation}
\rho_{\al c}(x) \sim \sqrt{\lambda_+-x}, \quad  \text{ for } \ \ x \in \left[\lambda_+ - 2\wt c,\lambda_+ \right],\ \al\in\{1,2\}.\label{SQUAREROOT}
\end{equation}
\item[(2)] For $z =E+\ii \eta\in S(\wt c,C_0,-\infty)$ and $\al\in\{1,2\}$, 
\begin{equation}\label{Immc}
\vert m_{\al c}(z) \vert \sim 1,  \quad  \im m_{\al c}(z) \sim \begin{cases}
    {\eta}/{\sqrt{\kappa+\eta}}, & \text{ if } E\geq \lambda_+ \\
    \sqrt{\kappa+\eta}, & \text{ if } E \le \lambda_+\\
  \end{cases}.
\end{equation}
\item[(3)] There exists a constant $\tau'>0$ such that
\begin{equation}\label{Piii}
\min_{\mu} \vert 1 + m_{1c}(z)\wt \sigma_\mu \vert \ge \tau', \quad \min_{i} \vert 1 + m_{2c}(z)\sigma_i  \vert \ge \tau',
\end{equation}
for any $z \in S(\wt c,C_0,-\infty)$.
\end{itemize}
The estimates \eqref{SQUAREROOT} and \eqref{Immc} also hold for $\rho_c$ and $m_c$. 
\end{lemma}

\begin{proof}
See Lemma 3.4 of \cite{yang2019edge}.
\end{proof}

We define the deterministic limit $\Pi$ of the resolvent $G$ in (\ref{eq_gz}) as
\begin{equation}\label{defn_pi}
\Pi (z):=\left( {\begin{array}{*{20}c}
   { \Pi_1} & 0  \\
   0 & { \Pi_2}  \\
\end{array}} \right), 
\end{equation}
where
$$ \Pi_1:  =-z^{-1}\left(1+m_{2c}(z)A \right)^{-1},\quad \Pi_2:=- z^{-1} (1+m_{1c}(z)B )^{-1}.$$
Note that by (\ref{separa_m12}) and (\ref{def_mc}), we have
\be\label{mcPi}
\frac1{p}\tr \Pi_{1} =m_c, \quad  \frac1{n}\tr \left(A\Pi_{1}\right) =m_{1c}, \quad \frac1{n}\tr \left(B\Pi_{2}\right) =m_{2c}.
\ee
Define the control parameter
\begin{equation}\label{eq_defpsi}
\Psi (z):= \sqrt {\frac{\Im m_{2c}(z)}{{n\eta }} } + \frac{1}{n\eta}.
\end{equation}
By Lemma \ref{lem_mbehavior}, we have
\begin{equation}\label{psi12}
\|\Pi\|=\OO(1), \ \ n^{-1/2}\lesssim\Psi \lesssim (n\eta)^{-1/2}, \ \ \Psi(z) \sim  \sqrt {\frac{\Im \, m_{1c}(z)}{{n\eta }} } + \frac{1}{n\eta},
\end{equation}
for $z\in S(\wt c,C_0,-\infty)$. 
We now state the first main result of this paper: a sharp local law for $G(X,z)$ without requiring a vanishing third moment assumption as in \eqref{eq_3rdmoment} for the entries of $X$. 

\begin{theorem}
\label{local_law}
Suppose Assumptions \ref{assm_big1} and \ref{ass:unper} hold. Suppose $X$ satisfies the bounded support condition (\ref{eq_support}) with $n^{-1/2}\le q\le n^{-\tau}$ for some constant $\tau>0$. 
 Fix $C_0>1$ and let $0<c_0<\wt c$ be a sufficiently small constant where $\wt c$ is defined in Lemma \ref{lem_mbehavior}.
Then the following estimates hold for any constant $\e>0$. 
\begin{itemize}
\item[(1)] {\bf Anisotropic local law}: For any $z\in  S(c_0,C_0,\epsilon)$ and deterministic unit vectors $\mathbf u, \mathbf v \in \mathbb C^{p+n}$,
\begin{equation}
\label{anisotropic}
\left| \langle \mathbf u, G(X,z) \mathbf v\rangle - \langle \mathbf u, \Pi (z)\mathbf v\rangle \right| \prec q+\Psi(z).
\end{equation}
\item[(2)] {\bf Averaged local law}: For any $z \in S(c_0, C_0, \epsilon)$,  we have
\begin{equation}
\label{average}
|m_1(z)-m_{1c}(z)|+|m_2(z)-m_{2c}(z)|+|m(z)-m_c(z)|\prec  \frac{q^2}{\sqrt{\kappa+\eta}}\wedge q +\frac{1}{n\eta}.
\end{equation}
\end{itemize}
The above estimates are uniform in $z$ and any set of deterministic unit vectors of cardinality $n^{\OO(1)}$. 
\end{theorem}

In general, proofs of local laws can be divided into two main steps, which are essentially decoupled.

\medskip
\noindent (A) \textit{A stochastic step.}
This step establishes a self-consistent equation for the resolvent up to a small random error. More precisely, we construct a map
$$M:\C^{(p+n)\times (p+n)}\to\C^{(p+n)\times (p+n)}$$
such that the deterministic limit of $G$ is uniquely characterized as the solution to the \emph{self-consistent equation} $M(\Lambda)=0$ with positive imaginary part. We then derive high-probability error bounds for the random matrix $M(G)$.

\medskip
\noindent (B) \textit{A deterministic step.}
This step analyzes the stability of the self-consistent equation. Together with the estimates from step (A), it implies that $G$ is close to its deterministic approximation $\Pi$.

\medskip

For separable covariance matrices, step (A) is implemented as follows.
We define the map $M=M(z):\C^{(p+n)\times (p+n)}\to\C^{(p+n)\times (p+n)}$ by
\begin{equation}
\label{s-ceq}
M(\Lambda)=M(z,\Lambda):=I_{p+n}+z\Lambda+\Lambda \cal S(\Lambda),
\end{equation}
where $\cal S:\C^{(p+n)\times (p+n)}\to\C^{(p+n)\times (p+n)}$ is a linear map defined by
$$\cal S(\Lambda):=\expect[H\Lambda H].$$
It is known that the equation $M(\Lambda)=0$ admits a unique solution with positive imaginary part, which coincides with $\Pi(z)$ defined in (\ref{defn_pi}).
The remaining task in step (A) is to establish high-probability bounds for $M(G)$ and for its averaged version $\underline{\Gamma M(G)}$, for any deterministic matrix $\Gamma$ with $\|\Gamma\|=\OO(1)$, as stated in the following theorem.

\begin{theorem}
\label{maintheorem}
Under the assumptions of Theorem \ref{local_law}, suppose that, at some $z\in S(c_0,C_0,\epsilon)$, $G-\Pi=\OO_{\prec}(\phi)$ (recall the notation in (ii) of Definition \ref{stoch_domination}) for some deterministic parameter $\phi\in[n^{-1},n^{\tau/10}]$, and that the constant $\tau$ satisfies $\tau <\epsilon/10$. Then, at this $z$, we have
\begin{equation}\label{eq_anisotropic_average}
  M(G)=\OO_{\prec}(\Phi),\qquad |\underline{\Gamma M(G)}|\prec\Phi^{2},\end{equation}
for any deterministic matrix $\Gamma\in\C^{(p+n)\times(p+n)}$ with $\|\Gamma\|=\OO(1)$, where the error parameter is given by
\begin{equation}\label{defn_Phi}
\Phi:=(1+\phi)^{3}(q+\Theta),\quad \text{with}\quad \Theta:=\sqrt{\frac{\norm{\Im \Pi}+\phi+\eta}{n\eta}} .\end{equation}
Note that $\Phi\ll 1$ under the conditions $q\le n^{-\tau}$ and $\tau <\epsilon/10$.
\end{theorem}

We first prove Theorem \ref{maintheorem} in Section \ref{sec:stoch_step}, completing the stochastic step; the deterministic step is accomplished in Section \ref{sec:det_step}. 
As a consequence of Theorem \ref{local_law}, we obtain the following rigidity estimates for the eigenvalues of $\cQ_1$ and $\cQ_2$.
\begin{theorem}[Rigidity of eigenvalues]
Suppose the assumptions of Theorem \ref{local_law} hold. Then, outside the spectrum, we have the following stronger averaged local law:
\begin{equation}\label{eq:averaged_outside}
|m_1(z)-m_{1c}(z)|+|m_2(z)-m_{2c}(z)|+|m(z)-m_c(z)|\prec
\frac{1}{N(\kappa+\eta)}
+\frac{1}{(N\eta)^2\sqrt{\kappa+\eta}},
\end{equation}
uniformly in
\(
z\in S(c_0,C_0,\varepsilon)
\cap \{z=E+\ii\eta:E\geq \lambda_+,\, N\eta\sqrt{\kappa+\eta}\geq N^\varepsilon\}
\)
for any fixed $\varepsilon>0$. 

The local laws \eqref{average} and \eqref{eq:averaged_outside} imply the following rigidity estimate for any constant $0<c_1<c_0$. Define the classical location $\gamma_j$ of the $j$-th eigenvalue of $\mathcal Q_1$ by
\begin{equation*}
\gamma_j:=\sup_x\left\{\int_x^{+\infty}\rho_c(x)\dd x>\frac{j-1}{p}\right\}.
\end{equation*}
In particular, $\gamma_1=\lambda_+$. Then, for any $j$ such that $\lambda_+-c_1\leq \gamma_j\leq \lambda_+$, we have
\begin{equation}\label{eq:rigidity}
|\lambda_j(\cal Q_1)-\gamma_j|\prec j^{-1/3}n^{-2/3}.
\end{equation}
\end{theorem}
\begin{proof}
This theorem is essentially Theorem 3.8 of \cite{yang2019edge} in the case where $A$ and $B$ are diagonal. For non-diagonal $A$ and $B$, \cite{yang2019edge} imposed the third moment condition (3.20) on the entries of $X$ in order to derive the local laws, which were then used to establish rigidity. Since Theorem \ref{local_law} proves the required local laws without any third moment assumption, the same argument as in \cite{yang2019edge} yields \eqref{eq:averaged_outside} and the rigidity estimates for general non-diagonal $A$ and $B$.
\end{proof}

 \subsection{Main result: outlier eigenvalues}
 
In this section, we present several applications of the local laws for (non-spiked) separable covariance matrices. Our primary focus is the asymptotic distribution of outlier eigenvalues for spiked separable covariance matrices defined in \eqref{eq_sepamodel}, which will be established in Theorem~\ref{distribution of outliers}. We also note that the local laws obtained in Theorem \ref{local_law} can be applied to improve the results in \cite{ding2020spikedseparablecovariancematrices,yang2019edge} (specifically, Theorem 2.7 of \cite{yang2019edge} and Theorems 3.7, 3.10, 3.14, 3.17, and 3.18 of \cite{ding2020spikedseparablecovariancematrices,yang2019edge}) by removing the vanishing third moment assumption for the entries of $X$, which was only used in the derivation of the local laws in \cite{yang2019edge} and is no longer needed in light of our results in Theorem \ref{local_law}. Throughout the ensuing discussion, we use the term {\it spikes} for eigenvalues of the population matrices $\wt A$ and $\wt B$, and the term {\it outlier eigenvalues} for eigenvalues of the sample separable covariance matrices $\ctQ_1$ and $\ctQ_2$.

We will see that a spike $\wt\sigma_i^a,\ i\in\cal I_1^+$, or $\wt\sigma_\mu^b,\ \mu\in\cal I_2^+$, causes an outlier eigenvalue beyond $\lambda_+$, if 
\be\label{spikes}
 \tsig_{i}^a > - m_{2c}^{-1}(\lambda_+)  \quad \text{ or } \quad  \tsig_{\mu}^b > - m_{1c}^{-1}(\lambda_+),
 \ee
 where $m_{1c}(\cdot)$ and $m_{2c}(\cdot)$ are defined in (\ref{separa_m12}). Moreover, such an outlier is around a deterministic location
\be\label{g12c} 
\theta_1(\tsig_{i}^a): = g_{2c}\left(-(\tsig_i^a)^{-1}\right) \quad \text{or} \quad \theta_2(\tsig_{\mu}^b): = g_{1c}\left(-(\tsig_\mu^b)^{-1}\right),
\ee
where $g_{1c}$ and $g_{2c}$ are the inverse functions of $m_{1c}: (\lambda_+,\infty)\to (m_{1c}(\lambda_+),0)$ and $m_{2c}: (\lambda_+,\infty)\to (m_{2c}(\lambda_+),0)$, respectively. Note that the inverse functions exist because
\begin{equation}\label{real_stiel}
 m_{\al c}(x) = \int_0^{\lambda_+} \frac{\rho_{\al c}(t)}{t-x}\dd t, \quad \al\in\{1,2\},
 \end{equation}
are monotonically increasing functions of $x$ for $x> \lambda_+$. 

\begin{assumption}\label{ass:spike} 
We assume that \eqref{spikes} holds for all $1\le i \le r$ and $1\le \mu \le s$. Otherwise, if \eqref{spikes} fails for some $\wt\sigma_i^a$ or $\wt\sigma_\mu^b$, we can simply redefine it as the unperturbed version $\sigma_i^a$ or $\sigma_\mu^b$. Moreover, we define the integers $0\le r^+ \leq r$ and $0\le s^+ \leq s$ such that
\begin{equation}\label{eq_spikeassua}
\wt{\sigma}^a_{i} \geq - m_{2c}^{-1} (\lambda_{+}) + n^{-1/3}  + q  \quad \text{if and only if} \quad  1\le  i \le r^+,
\end{equation}
and 
\begin{equation} \label{eq_spikeassub}
\wt{\sigma}^b_{\mu}\ge - m_{1c}^{-1} (\lambda_{+}) + n^{-1/3} + q  \quad \text{if and only if} \quad 1\le  \mu \le s^+ .
\end{equation}
The lower bound $n^{-1/3}+q$ is chosen for definiteness, and {it can be replaced with any $n$-dependent parameter that is of the same order}.
\end{assumption}

To state the results on outlier eigenvalues, we introduce the following labeling convention.

\begin{definition}\label{defn_relabelling}
We define the labeling functions $\al:\{1,\cdots, p\}\to \N$ and $\beta:\{1,\cdots, n\}\to \N$ as follows. For any $1\le i \le r$, we assign to it a label $\alpha(i)\in \{1,\cdots, r+s\}$ if $\theta_1(\wt{\sigma}^a_{i})$ is the $\alpha(i)$-th largest element in $\{\theta_1(\wt{\sigma}^a_i)\}_{i=1}^r \cup  \{\theta_{2}(\wt{\sigma}^b_\mu)\}_{\mu=1}^s$. We also assign to any $1\le \mu \le s$ a label $\beta(\mu)\in \{1,\cdots, r+s\}$ in a similar way. Moreover, we define $\al(i)=i+s$ if $i>r$ and $\beta(\mu)=\mu + r$ if $\mu >s$.
\end{definition}

Denote the nontrivial eigenvalues of $\widetilde{\mathcal{Q}}_{1,2}$ by $\wt\lambda_1 \geq \wt\lambda_2 \geq \cdots \geq \wt\lambda_{n \wedge p}.$ For $1\le i \le r$ and $1\le \mu \le s$, we define
\begin{align}\label{deltaimu}
& \Delta_1(\widetilde{\sigma}_i^a):= \left(\widetilde{\sigma}_i^a+m_{2c}^{-1}(\lambda_+)\right)^{1/2}, \quad \Delta_2(\widetilde{\sigma}_\mu^b):= \left(\widetilde{\sigma}_\mu^b+m_{1c}^{-1}(\lambda_+)\right)^{1/2}.
\end{align}
The following theorem from \cite{ding2020spikedseparablecovariancematrices} gives the large deviation bounds for the locations of the outliers.

\begin{theorem}[Theorem 3.6 of \cite{ding2020spikedseparablecovariancematrices}]\label{thm_outlier}
Suppose that Assumptions \ref{assm_big1'}, \ref{ass:unper}, and \ref{ass:spike} hold, except that the high-moment condition \eqref{eq_highmoment} is replaced by the slightly more general bounded support condition \eqref{eq_support}, with $ n^{-{1}/{2}} \leq q \leq n^{- c_q} $ for some constant $c_q>0$. Then
\begin{equation}\label{eq_spike}
\left|\wt\lambda_{\alpha(i)}-\theta_1(\wt\sigma_i^a)\right| \prec n^{-1/2}\Delta_1(\widetilde{\sigma}_i^a)  + q\Delta_1^2(\widetilde{\sigma}_i^a), \quad 1 \leq i \leq r,
\end{equation}
\begin{equation}\label{eq_spike2}
\left|\wt\lambda_{\beta(\mu)}-\theta_2(\wt\sigma_\mu^b)\right| \prec n^{-1/2} \Delta_2(\widetilde{\sigma}_\mu^b) + q\Delta_2^2(\widetilde{\sigma}_\mu^b), \quad 1 \leq \mu \leq s .
\end{equation}
Moreover, for any fixed integer $k>r+s$,
\begin{equation}\label{eq_nonspike}
\left|\wt\lambda_{k}-\lambda_+\right| \prec n^{-2/3} + q^2.
\end{equation}
\end{theorem}

To state the asymptotic distribution of the outliers, we assume that the spikes $\wt\sigma_i^a$ and $\wt\sigma_\mu^b$ are well separated from the phase transition points \smash{\(- m_{2c}^{-1}(\lambda_+)\)} and \smash{\(- m_{1c}^{-1}(\lambda_+)\)} (recall \eqref{spikes}). For simplicity of presentation, we then relabel $\wt\sigma_i^a$ with $r^+<i\le r$ and $\wt\sigma_\mu^b$ with $s^+<\mu\le s$ as $\sigma_i^a$ and $\sigma_\mu^b$, respectively. Under this convention, we have $r^+=r$ and $s^+=s$, and we impose the following assumption.

\begin{assumption}\label{ass:awayfrombulk}
We assume that
\be\label{eq:supercritical}
 \tsig_{i}^a > - m_{2c}^{-1}(\lambda_+)+\tau  \quad \text{ and } \quad  \tsig_{\mu}^b > - m_{1c}^{-1}(\lambda_+)+\tau,
 \ee
 holds for all $1\leq i\leq r$, $1\leq \mu\leq s$ and some constant $\tau>0$.
\end{assumption}

By Theorem \ref{thm_outlier}, this assumption ensures that the outliers remain at distances of order 1 from the bulk. We further require that the outliers be well separated from one another, which is guaranteed by the following assumption together with Theorem \ref{thm_outlier}. 

\begin{assumption}\label{ass:nonoverlapping}
Assume that the distances between any two of the $r+s$ classical locations of the outlier eigenvalues
\be\nonumber
\theta_1(\tsig_{i}^a),\ \ \ 1\leq i\leq r \quad \text{and} \quad \theta_2(\tsig_{\mu}^b), \ \ \ 1\leq \mu\leq s
\ee
are bounded below by some constant $\tau>0$.
\end{assumption}

We now state the asymptotic distribution of the outliers.
\begin{theorem}\label{distribution of outliers}
Suppose that Assumptions \ref{assm_big1'}, \ref{ass:unper}, \ref{ass:awayfrombulk} and \ref{ass:nonoverlapping} hold. Then
\begin{equation}
\label{eq:main_asymptotic_outlier}
\sqrt{n}\pb{\wt\lambda_{\alpha(i)}-\theta_i^a}=\Phi_i^a+\OO_{\prec}(n^{-\e}),\quad \sqrt{n}\pb{\wt\lambda_{\beta(\mu)}-\theta_\mu^b}=\Phi_\mu^b+\OO_{\prec}(n^{-\e}),
\end{equation}
where we abbreviate $\theta_i^a\equiv\theta_1(\wt\sigma_i^a)$ and $\theta_\mu^b\equiv\theta_2(\wt\sigma_\mu^b)$, and $\Phi_i^a$ and $\Phi_\mu^b$ are random variables that are asymptotically equal in distribution (recall Definition \ref{defn_asymptotic}) to centered Gaussian variables with variances
\begin{equation}\label{eq:asymp_var}
  \pbb{\frac{1}{m'_{2c}(\theta_i^a)}}^2\pB{2m_{2a}(\theta_i^a)+k_4b(\theta_i^a)s_4(\vv_i^a)},\quad \pbb{\frac{1}{m'_{1c}(\theta_\mu^b)}}^2\pB{2m_{1b}(\theta_\mu^b)+k_4a(\theta_\mu^b)s_4(\vv_\mu^b)},\end{equation}
respectively. Here, we denote $k_4:=n^2\kappa_4$ and define the following functions of $z\in \C$: 
\begin{equation}\label{def:alpha(z) beta(z) a(z) b(z) s4}
\begin{aligned}
m_{2a}(z):=\frac{\gamma_2(z)}{1-z^2\gamma_1(z)\gamma_2(z)},&\quad m_{1b}(z):=\frac{\gamma_1(z)}{1-z^2\gamma_1(z)\gamma_2(z)},\\
\gamma_1(z):=d_n\int\frac{x^2}{z^2(1+m_{2c}(z)x)^2}\pi_A(\dd x),&\quad \gamma_2(z):=\int\frac{x^2}{z^2(1+m_{1c}(z)x)^2}\pi_B(\dd x),\\
a(z):=\frac{1}{n}\sum_i \pB{z^{-1}A\pb{1+m_{2c}(z)A}^{-1}}_{ii}^2,&\quad b(z):=\frac{1}{n}\sum_\mu \pB{z^{-1}B\pb{1+m_{1c}(z)B}^{-1}}_{\mu\mu}^2,\\
s_4(\vv_i^a):=\sum_j \vv_i^a(j)^4,&\quad s_4(\vv_\mu^b):=\sum_\nu \vv_\mu^b(\nu)^4.
\end{aligned}
\end{equation} 
\end{theorem}

In the special case $A=I_p$ and $B=I_n$, the functions $m_{1c}$ and $m_{2c}$ admit explicit expressions, and Theorem \ref{distribution of outliers} reduces to part of Theorem 2.9 in \cite{bao2020statisticalinferenceprincipalcomponents}.
    
We conduct numerical simulations to validate Theorem \ref{distribution of outliers}; the results are displayed in Figure \ref{Simulation Results}. In our experiments, we take $p=500,\ n=1000,$ and assume that $\textup{Spec}(A),\ \textup{Spec}(B)\subset[1,2]$, with $r=s=1$. We consider the following four parameter settings:
$$ \text{(a)}:\theta_1^a=10,\ \theta_1^b=20,\quad \text{(b)}:\theta_1^a=10,\ \theta_1^b=80,\quad \text{(c)}:\theta_1^a=20,\ \theta_1^b=80,\quad \text{(d)}:\theta_1^a=40,\ \theta_1^b=80,$$
corresponding to subfigures (\ref{fig:10-20}), (\ref{fig:10-80}), (\ref{fig:20-80}), and (\ref{fig:40-80}), respectively.
We observe that the theoretical predictions agree well with the simulations in (\ref{fig:10-20}) and (\ref{fig:10-80}), exhibit mild deviations in (\ref{fig:20-80}), and break down in (\ref{fig:40-80}). In fact, the empirical distribution in (\ref{fig:40-80}) is no longer approximately Gaussian. This discrepancy is mainly due to the relatively small sample size $n$. More precisely, as will become clear in the proof of Lemma \ref{reduction to Green functions} below, a key ingredient in treating the outlier eigenvalues individually is that $\norm{\mathsf C^{-1}}$ remains bounded, where $\mathsf C$ is a matrix to be defined later in the proof. However, in the setting of (\ref{fig:40-80}), this quantity is already numerically of order $\sqrt{n}$ (approximately 45), violating the required condition. Nevertheless, we point out that, when there is only one outlier eigenvalue arising from either $A$ or $B$, the agreement between theory and simulation in setting (d) is even better than in (\ref{fig:10-20}) and (\ref{fig:10-80}).

\begin{figure}[htbp]
    \centering
    \begin{subfigure}[b]{1\textwidth}
        \centering
        \includegraphics[width=\textwidth]{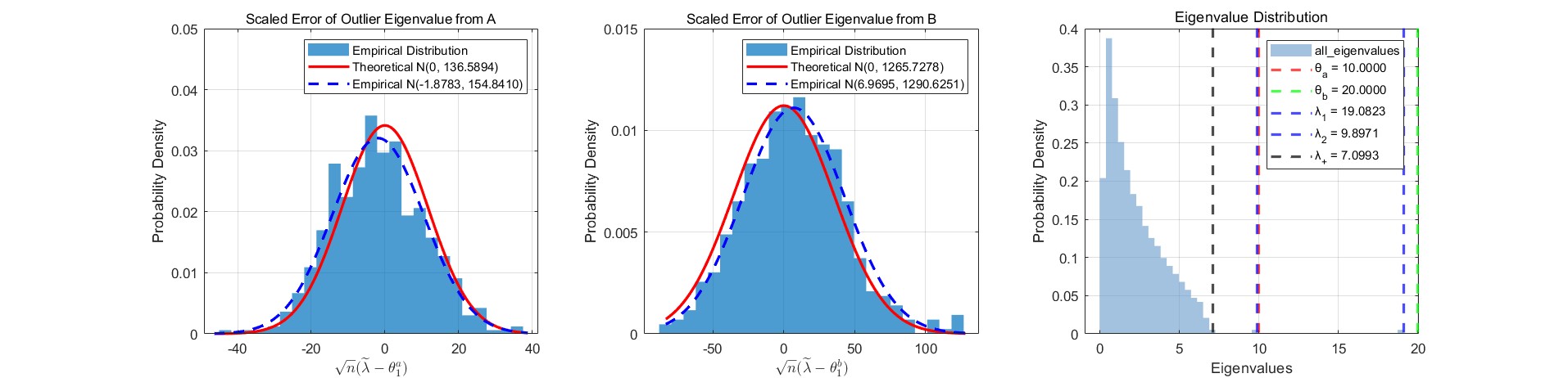}
        \caption{$\theta_1^a=10,\ \theta_1^b=20.$}
        \label{fig:10-20}
    \end{subfigure}
    \begin{subfigure}[b]{1\textwidth}
        \centering
        \includegraphics[width=\textwidth]{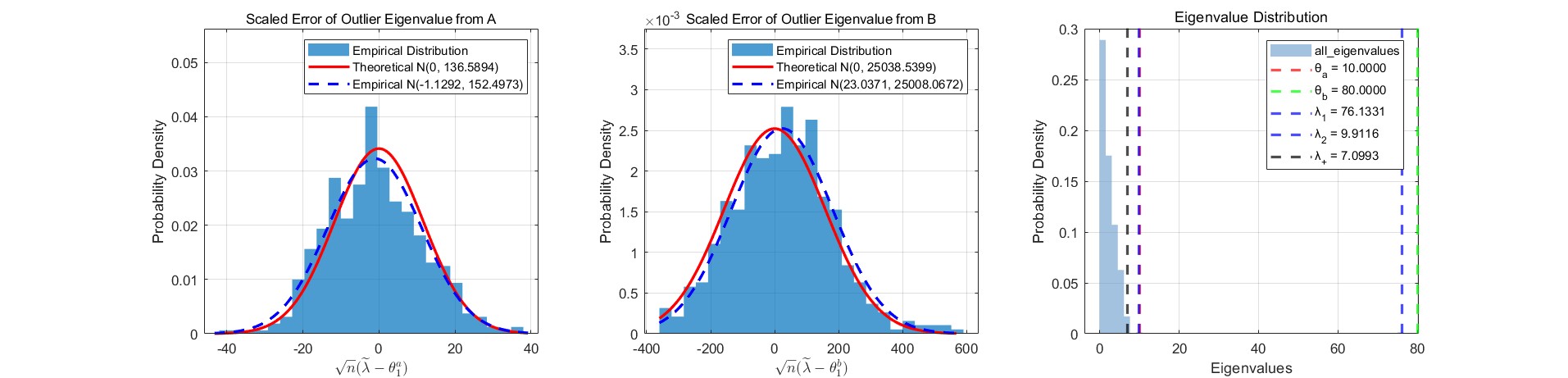}
        \caption{$\theta_1^a=10,\ \theta_1^b=80.$}
        \label{fig:10-80}
    \end{subfigure}
    \begin{subfigure}[b]{1\textwidth}
        \centering
        \includegraphics[width=\textwidth]{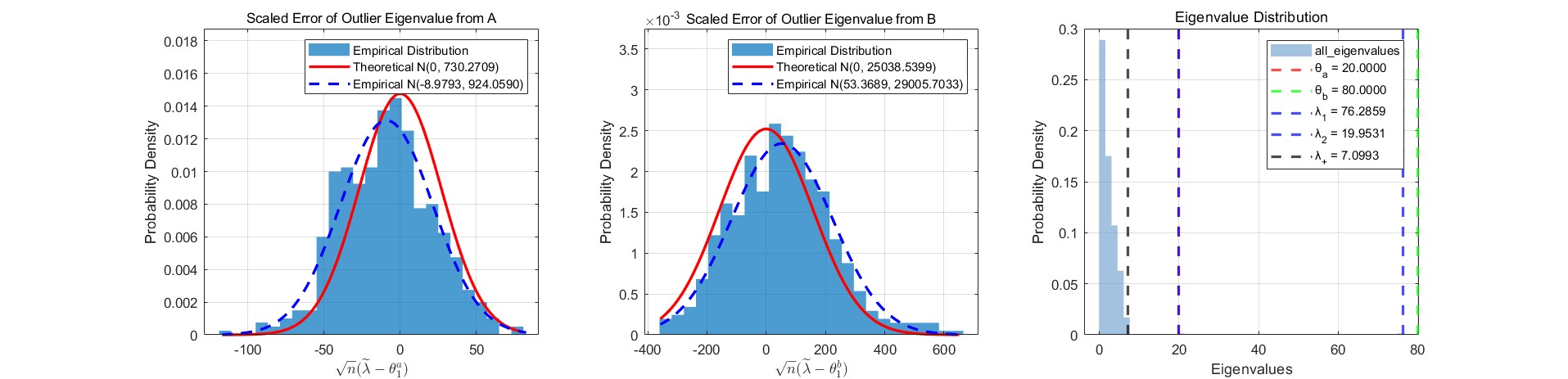}
        \caption{$\theta_1^a=20,\ \theta_1^b=80.$}
        \label{fig:20-80}
    \end{subfigure}
    \begin{subfigure}[b]{1\textwidth}
        \centering
        \includegraphics[width=\textwidth]{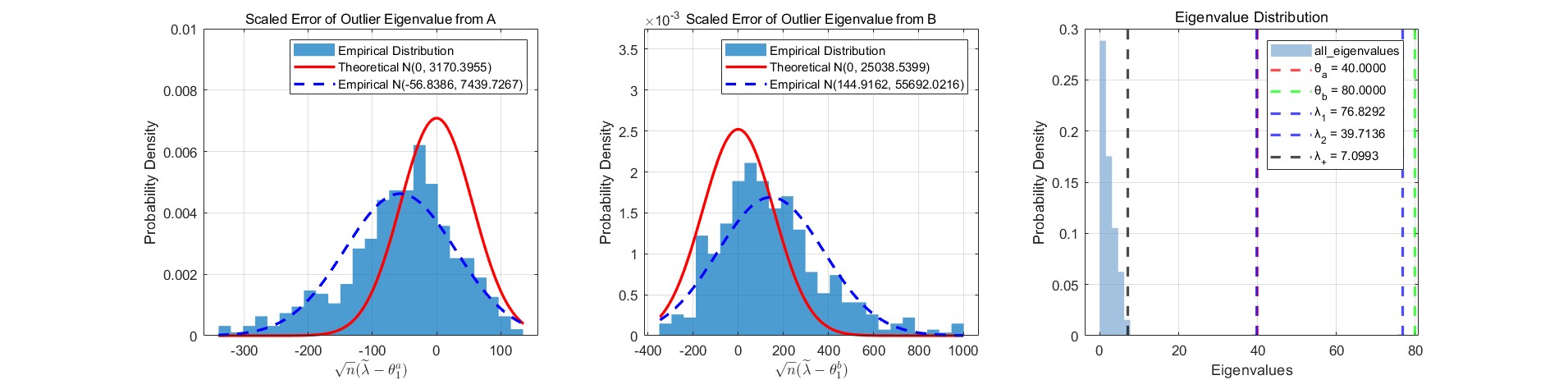}
        \caption{$\theta_1^a=40,\ \theta_1^b=80.$}
        \label{fig:40-80}
    \end{subfigure}
    \caption{Simulation results for the asymptotic distributions of outlier eigenvalues.}
    \label{Simulation Results}
\end{figure}

\section{Proof of Theorem \ref{maintheorem}}\label{sec:stoch_step}
\subsection{Main tools}

We first establish an a priori bound on the spectral norms of $\mathcal Q_{1}$ and $\mathcal Q_{2}$.

\begin{lemma}\label{priori}
Under the assumption of Theorem \ref{local_law}, we have $\norm{\cal Q_1}=\norm{\cal Q_2}\prec 1.$
\end{lemma}
\begin{proof}
With $\norm{\cal Q_1}=\norm{\cal Q_2}\leq\norm{A}\norm{B}\norm{XX^*},$ it suffices to show that $\norm{XX^*}\prec 1$. This follows from the rigidity of the eigenvalues of $XX^*$; see \cite[Lemma 3.12]{DY_AAP}.
\end{proof}

Second, we will apply the cumulant expansion formulas introduced in \cite{Cumulant1,Cumulant2,He:2017wm}, which have proved useful in various studies of random matrix theory. We adopt the formulation stated in \cite[Lemma C.11]{bao2020statisticalinferenceprincipalcomponents}. 
Recall that, for a real random variable $\xi$ whose moments are all finite, the $k$-th cumulant of $\xi$ is 
\begin{equation*}
	\kappa_k(\xi)\deq(-\mathrm{i})^k\bigg(\frac{\dd^{k}}{\dd t^k}\log\mathbb{E}\qb{e^{\mathrm{i}t\xi}}\bigg)\Bigg|_{t=0}.
\end{equation*}

\begin{lemma}\label{lem:cumulant_expansion}
(Cumulant expansion formula, Lemma C.11 of \cite{bao2020statisticalinferenceprincipalcomponents}) For a fixed $\ell\in \mathbb{N}$, let $f\in C^{\ell+1}(\mathbb{R})$. Suppose that $\xi$ is a centered random variable with finite moments up to order $\ell+2$. Then we have  
\begin{equation}\label{eq:cumulant_expansion}
\expect\qb{\xi f(\xi)}=\sum_{k=1}^{\ell}\frac{\kappa_{k+1}(\xi)}{k!}\expect\qb{f^{(k)}(\xi)}+\cal R_{\ell+1},
\end{equation}
where the error term $\cal R_{\ell+1}$ satisfies
\begin{equation}\label{R_l+1}
\abs{\cal R_{\ell+1}}\leq C_\ell\expect\qb{|\xi|^{\ell+2}}\sup_{|t|\leq s}|f^{(\ell+1)}(t)|+C_\ell\expect\qb{|\xi|^{\ell+2}\mathds{1}(|\xi|>s)}\sup_{t\in \mathbb{R}}|f^{(\ell+1)}(t)|
\end{equation}
for any $s>0$ and $C_\ell:=(C\ell)^{\ell}/\ell!$ for some constant $C>0$.
\end{lemma}

The following result gives bounds on the cumulants of the entries of $X$ under the bounded support condition \eqref{eq_support}. If the entries of $X$ satisfy \eqref{eq_124moment} and \eqref{eq_support}, then $\kappa_{1}(x_{i\mu})=0$, $\kappa_2(x_{i\mu})=n^{-1}$, $\max_{i,\mu}\abs{\kappa_3(x_{i\mu})}=\OO(n^{-3/2})$, and for any fixed $k\ge 3$,
	\begin{equation}\label{eq:kappa_k+1}
		\max_{i,\mu}\abs{\kappa_{k+1}(x_{i\mu})}\prec n^{-2}q^{k-3}.
	\end{equation}
This follows readily from the homogeneity of cumulants. Next, we record a resolvent bound that is a direct consequence of Ward's identity.   

 \begin{lemma}
 \label{wardlemma}
 Under the assumptions of Theorem \ref{maintheorem}, for any deterministic matrices $\Gamma_1,\Gamma_2\in\C^{(p+n)\times(p+n)}$ satisfying $\norm{\Gamma_1},\norm{\Gamma_2}=\OO(1)$ and any deterministic unit vector $\uu \in \mathbf{S}$, we have
 \begin{equation}
\label{ward2}
\sum_{a}\left|(\Gamma_1G\Gamma_2)_{\uu a}\right|^{2}\prec n\Theta^2.
\end{equation}
where we recall that $\Theta$ is defined in (\ref{defn_Phi}).
\end{lemma}
\begin{proof}
We first consider the special case $\Gamma_1=\Gamma_2=I_{p+n}$. By (\ref{def_green}), we have Ward's identity
\begin{equation}\label{eq_Ward_G1}
\Im \mathcal G_1 = \frac{\mathcal G_1-\mathcal G_1^{*}} {2\ii} = \eta\mathcal G_1\mathcal G_1^{*}.\end{equation}
Using the notation in \eqref{eq_projection}, we then deduce that
\begin{equation}
\nonumber
\begin{aligned} 
\sum_{a}\left|G_{\uh a}\right|^{2}
&=(GG^{*})_{\uh\uh} = (\mathcal G_1\mathcal G_1^{*}+\abs{z}^{-1}\mathcal G_1 YY^{*} \mathcal G_1^{*})_{\uh\uh}\prec (\mathcal G_1\mathcal G_1^{*})_{\uh\uh} \\
& 
= \frac{\Im (\mathcal G_{1})_{\uh\uh}}{\eta}= \frac{\Im (\Pi_1)_{\uh\uh}+\Im (\mathcal G_{1}-\Pi_1)_{\uh\uh}}{\eta}
\leq \frac{\norm{\Im \Pi}+\phi}{\eta} \leq n\Theta^2.
\end{aligned}
\end{equation}
where in the third step we used the bound from Lemma \ref{priori}, in the fourth step we used the identity \eqref{eq_Ward_G1}, and in the sixth step we used (\ref{green2}), \eqref{defn_pi}, together with the assumption $G-\Pi=\OO_\prec(\phi)$.
A similar argument gives $\sum_{a}\left|G_{\ut a}\right|^{2}\prec n\Theta^2$. Therefore, we conclude that
\begin{equation}\label{eq:L2G}\sum_{a}\left|G_{\uu a}\right|^{2}=\norm{G^*\uu}^2\leq2(\norm{G^*\uh}^2+\norm{G^*\ut}^2)\prec n\Theta^2.\end{equation}
For general $\Ga_1$ and $\Ga_2$, set $\uu'=\Ga_1^*\uu$ and $\uu'_0=\uu'/\norm{\uu'}\in\mathbf{S}$.\footnote{Here, we assume that $\Ga_1^*\uu\neq 0$; otherwise, the conclusion is trivial} Then we have 
$$\sum_{a}\left|(\Ga_1G\Ga_2)_{\uu a}\right|^{2}=\norm{\Ga_2^*G^*\Ga_1^*\uu}^2\leq\norm{\Ga_2}^2\norm{\uu'}^2
\norm{G^*\uu'_0}^2\prec n\Theta^2,$$
where we used $\norm{\Ga_1},\norm{\Ga_2}=\OO(1)$ and \eqref{eq:L2G}. 
\end{proof}
The above lemma provides the main tool for obtaining smallness in the proof, improving the trivial bound $n(1+\phi)^2$ for the left-hand side of \eqref{ward2} to $n\Theta^2$. With a slight abuse of notation, in the following proof, when we refer to ``Ward's identity", we will mean either the ``true" Ward's identity \eqref{eq_Ward_G1} or Lemma \ref{wardlemma}. We then state some useful corollaries of Lemma \ref{wardlemma}.

\begin{lemma}\label{multipleG}
 Under the assumptions of Theorem \ref{maintheorem}, for any deterministic matrices $\Gamma_1,\Gamma_2,\Ga_3,\Ga_4\in\C^{(p+n)\times(p+n)}$ with $\OO(1)$ operator norms and any deterministic unit vectors $\xx,\yy\in\mathbf{S}$, we have
 $$\abs{(\Ga_1G\Ga_2)_{\xx\yy}}\prec 1+\phi,\quad \abs{(\Ga_1G\Ga_2G\Ga_3)_{\xx\yy}}\prec n\Theta^2,\quad \abs{(\Ga_1G\Ga_2G\Ga_3G\Ga_4)_{\xx\yy}}\prec (1+\phi)n^2\Theta^2.$$
 These bounds also hold if some of the factors $G$ are replaced by $\overline{G}$.
\end{lemma}
\begin{proof}
Using $G=\Pi+(G-\Pi)=\OO_\prec(1+\phi)$, we immediately obtain $\abs{(\Ga_1G\Ga_2)_{\xx\yy}}\prec 1+\phi$. By Cauchy's inequality and Lemma \ref{wardlemma}, we have
\begin{equation}
\nonumber
\begin{aligned} 
\abs{(\Ga_1G\Ga_2G\Ga_3)_{\xx\yy}}&=\left|\sum_a(\Ga_1G\Ga_2)_{\xx a}(G\Ga_3)_{a\yy}\right| \leq\sqrt{\sum_a\abs{(\Ga_1G\Ga_2)_{\xx a}}^2\sum_a\abs{(G\Ga_3)_{a\yy}}^2}\\
&\prec\sqrt{n\Theta^2\cdot n\Theta^2}=n\Theta^2,\\
\abs{(\Ga_1G\Ga_2G\Ga_3G\Ga_4)_{\xx\yy}}=&\left|\sum_{a,b}(\Ga_1G)_{\xx a}(\Ga_2G\Ga_3)_{ab}(G\Ga_4)_{b\yy}\right| \prec(1+\phi)\sum_a\abs{(\Ga_1G)_{\xx a}}\sum_b\abs{(G\Ga_4)_{b\yy}}\\
&\leq(1+\phi)\sqrt{(p+n)\sum_a\abs{(\Ga_1G)_{\xx a}}^2}\sqrt{(p+n)\sum_b\abs{(G\Ga_4)_{b\yy}}^2}\\
&\prec(1+\phi)\sqrt{n\cdot n\Theta^2}\sqrt{n\cdot n\Theta^2}=(1+\phi)n^2\Theta^2.
\end{aligned}
\end{equation}
Since complex conjugation does not affect the use of Cauchy's inequality, the same argument applies if some of the factors $G$ are replaced by $\overline{G}$.
\end{proof}

\subsection{Anisotropic estimation}

In this subsection, we prove the first assertion of Theorem \ref{maintheorem}, namely, $M(G)_{\uu\vv}\prec\Phi$ for any deterministic vectors $\uu,\vv\in\C^{p+n}$ with $\norm{\uu},\norm{\vv}\leq 1$. By separating the real and imaginary parts, we may assume that $\uu,\vv\in\R^{p+n}$.
Since $M(G)_{\uu\vv} = M(G)_{\uu\vh} + M(G)_{\uu\vt}$, it suffices to prove $M(G)_{\uu\vh}\prec\Phi$; the proof of $M(G)_{\uu\vt}\prec\Phi$ is analogous by symmetry.
Without loss of generality, we may assume that $\norm{\uu}=\norm{\vh}=1$, since replacing $\uu$ by $\uu/\norm{\uu}$ and $\vh$ by $\vh/\norm{\vh}$ can only increase the  value of $|M(G)_{\uu\vh}|$.

We start with some basic calculations. We abbreviate $\partial_{i\mu}={\partial}/{\partial x_{i\mu}}$ and introduce the following block-diagonal matrices:
\begin{equation}
\AA := \begin{pmatrix}
  A^{1/2}& 0\\
  0&0
\end{pmatrix},\quad  \BB :=  \begin{pmatrix}
  0& 0\\
  0&B^{1/2}
\end{pmatrix}.\label{defn_AB}
\end{equation} 
Note that
$$H(z) = z^{1/2}\begin{pmatrix}
  0& A^{1/2}XB^{1/2}\\
  B^{1/2}X^*A^{1/2}&0
\end{pmatrix}=z^{1/2}\begin{pmatrix}
  A^{1/2}&0 \\
  0&B^{1/2}
\end{pmatrix}\begin{pmatrix}
  0&X \\
  X^*&0
\end{pmatrix}\begin{pmatrix}
  A^{1/2}&0 \\
  0&B^{1/2}
\end{pmatrix}.$$
Thus, we obtain
\begin{equation}\label{Gderivative}
\partial_{i\mu}G_{\xx\yy}=-\pb{G(\partial_{i\mu}H)G}_{\xx\yy}=-z^{1/2}\pb{(G\AA)_{\xx i}(\BB G)_{\mu\yy}+(G\BB)_{\xx\mu}(\AA G)_{i\yy}}.
\end{equation}
We now derive explicit formulas for $M(G)_{\uu\vh}$ and its derivative. Consider an arbitrary $(p+n)\times (p+n)$ matrix $\Lambda$ with the block structure
$$\Lambda = \begin{pmatrix}
  P&Q\\
  R&S
\end{pmatrix}.$$
A direct calculation gives
\begin{equation}\label{S(Lambda)}
\begin{aligned}
\cal S(\Lambda)&=\expect[H\Lambda H]
= \frac{z}{n}\begin{pmatrix}
  \Tr(BS)\cdot A& AR^TB\\
  BQ^TA&\Tr(AP)\cdot B
\end{pmatrix}\\
&=\frac{z}{n}\pb{\Tr(\Lambda\BB^2)\cdot\AA^2+\Tr(\Lambda\AA^2)\cdot\BB^2+\BB^2\Lambda^T\AA^2+\AA^2\Lambda^T\BB^2}.
\end{aligned}
\end{equation}
Substituting $\Lambda=G(=G^T)$ and using $\BB\vh=0$, together with the notation in \eqref{defn_m1m2}, we obtain
\begin{equation}
\label{M(G)expansion}
\begin{aligned}
M(G)_{\uu\vh}&=\left<\uu,\vh\right>+zG_{\uu\vh}+\pb{G\cal S(G)}_{\uu\vh}\\
&=\left<\uu,\vh\right>+zG_{\uu\vh}+zm_{2}(z)(G\AA^2)_{\uu\vh}+\frac{z}{n}(G\BB^2G\AA^2)_{\uu\vh}\\
&=(GH)_{\uu\vh}+zm_{2}(z)(G\AA^2)_{\uu\vh}+\frac{z}{n}(G\BB^2G\AA^2)_{\uu\vh}\\
&=z^{1/2}\sum_{i,\mu}(G\BB)_{\uu\mu}\AA_{i\vh}x_{i\mu}+zm_{2}(z)(G\AA^2)_{\uu\vh}+\frac{z}{n}(G\BB^2G\AA^2)_{\uu\vh},
\end{aligned}
\end{equation}
and
\begin{equation}
\label{M(G)'}
\begin{aligned}
\partial_{i\mu}M(G)_{\uu\vh}=&~\partial_{i\mu}\pB{\left<\uu,\vh\right>+zG_{\uu\vh}+zm_2(z)(G\AA^2)_{\uu\vh}+\frac{z}{n}(G\BB^2G\AA^2)_{\uu\vh}}\\
=&-z^{3/2}\Bigl((G\AA)_{\uu i}(\BB G)_{\mu\vh}+(G\BB)_{\uu\mu}(\AA G)_{i\vh}\\
&~+m_2(z)(G\AA)_{\uu i}(\BB G\AA^2)_{\mu\vh}+m_2(z)(G\BB)_{\uu\mu}(\AA G\AA^2)_{i\vh}\\
&~+\frac{1}{n}(\BB G\BB^2G\AA)_{\mu i}(G\AA^2)_{\uu\vh}+\frac{1}{n}(\AA G\BB^2G\BB)_{i\mu}(G\AA^2)_{\uu\vh}\\
&~+\frac{1}{n}(G\AA)_{\uu i}(\BB G\BB^2G\AA^2)_{\mu\vh}+\frac{1}{n}(G\BB)_{\uu\mu}(\AA G\BB^2G\AA^2)_{i\vh}\\
&~+\frac{1}{n}(G\BB^2 G\AA)_{\uu i}(\BB G\AA^2)_{\mu\vh}+\frac{1}{n}(G\BB^2 G\BB)_{\uu\mu}(\AA G\AA^2)_{i\vh}\Bigl).
\end{aligned}
\end{equation}
We next record a preliminary estimate on the derivatives of $M(G)_{\uu\vh}$.
\begin{lemma}\label{lem:roughbound}
For each fixed $r\in \mathbb{N}$, we have
\begin{equation}
\label{roughbound}
\begin{aligned}
\abs{\partial_{i\mu}^{r}M(G)_{\uu\vh}}\prec(1+\phi)^{r+2}.
\end{aligned}
\end{equation}
\end{lemma}
\begin{proof}
By (\ref{Gderivative}), for any deterministic matrices $\Gamma_1,\Gamma_2\in\C^{(p+n)\times(p+n)}$ with $\OO(1)$ operator norms, $\partial_{i\mu}^{r} (\Ga_1G\Ga_2)_{\uu\vh}$ is a scalar multiple of a sum of $2^{r}$ terms, each given by a product of $r+1$ entries of the form $(\Ga_1'G\Ga_2')_{\xx\yy}$, where
$$\Ga_1',\Ga_2'\in\{\Ga_1,\Ga_2,\AA,\BB\},\quad \xx,\yy\in\{\uu,\vh,i,\mu\}.$$
This immediately yields
\[\abs{\partial_{i\mu}^{r} (\Ga_1G\Ga_2)_{\uu\vh}}\prec(1+\phi)^{r+1}.\]
Consequently,
\begin{equation}
\begin{aligned}
\absa{\partial_{i\mu}^{r}m_2(z)}&=\frac{1}{n}\left|\sum_{\nu}\partial_{i\mu}^{r}(G\BB^2)_{\nu\nu}\right|\prec(1+\phi)^{r+1},\\
\left|\partial_{i\mu}^{r} \frac{1}{n}(G\BB^2G\AA^2)_{\uu\vh}\right|
&= \frac{1}{n}\left|\partial_{i\mu}^{r}\pB{\sum_{\nu}(G\BB^2)_{\uu\nu}(G\AA^2)_{\nu\vh}}\right|
= \frac{1}{n}\left|\sum_{\nu}\sum_{s=0}^{r}\partial_{i\mu}^{s}(G\BB^2)_{\uu \nu}\partial_{i\mu}^{r-s}(G\AA^2)_{\nu\vh}\right|\\
&\prec \frac{1}{n}\sum_{\nu}\sum_{s=0}^{r} (1+\phi)^{s+1}(1+\phi)^{r-s+1} \prec(1+\phi)^{r+2}.
\end{aligned}
\nonumber
\end{equation}
Combining these estimates with the second expression on the right-hand side of (\ref{M(G)expansion}), we obtain $\abs{\partial_{i\mu}^{r}M(G)_{\uu\vh}}\prec(1+\phi)^{r+2}$.
\end{proof}

We now return to the proof of the estimate $M(G)_{\uu\vh}\prec\Phi$. By Lemma \ref{momentmethod}, it suffices to prove that
$$\mathcal{M}:=\expect\qB{|M(G)_{\uu\vh}|^{2e}}^{\frac{1}{2e}}\prec\Phi$$
for any fixed $e\in\N_+$ with $e\ge 4$. Using \eqref{M(G)expansion} and applying the cumulant expansion (Lemma \ref{lem:cumulant_expansion}) to $x_{i\mu}$, we expand $\mathcal{M}^{2e}$ as
\begin{align}
\mathcal{M}^{2e}&=\expect\qB{|M(G)_{\uu\vh}|^{2e}}=\expect\qB{M(G)_{\uu\vh}^{e}\overline{M(G)}_{\uu\vh}^{e}}\nonumber\\
&=z^{1/2}\sum_{i,\mu}\expect\qB{(G\BB)_{\uu\mu}\AA_{i\vh}x_{i\mu}M(G)_{\uu\vh}^{e-1}\overline{M(G)}_{\uu\vh}^{e}}\nonumber\\
&\quad +\expect\qB{\pB{zm_{2}(z)(G\AA^2)_{\uu\vh}+\frac{z}{n}(G\BB^2G\AA^2)_{\uu\vh}}M(G)_{\uu\vh}^{e-1}\overline{M(G)}_{\uu\vh}^{e}} \nonumber\\
&=\expect\qB{\pB{zm_{2}(z)(G\AA^2)_{\uu\vh}+\frac{z}{n}(G\BB^2G\AA^2)_{\uu\vh}}M(G)_{\uu\vh}^{e-1}\overline{M(G)}_{\uu\vh}^{e}}+\sum_{k=1}^\ell X_k+\sum_{i,\mu}\mathcal{R}_{\ell+1}^{i\mu},\label{aims}
\end{align}
where
$$X_k:=z^{1/2}\sum_{i,\mu}\frac{\kappa_{k+1}(x_{i\mu})}{k!}\expect\qB{\AA_{i\vh}\partial_{i\mu}^{k} \left((G\BB)_{\uu\mu}M(G)_{\uu\vh}^{e-1}\overline{M(G)}_{\uu\vh}^{e}\right)}$$
for $k\in\N$, $\ell$ is a positive integer to be chosen later, and the remainder terms $\cal R_{\ell+1}^{i\mu}$ satisfy
\begin{equation}\label{remainder^imu}
\begin{aligned}
\abs{\cal R_{\ell+1}^{i\mu}}\leq& C_\ell\expect\qb{|x_{i\mu}|^{\ell+2}}\expect\qbb{\sup_{|\xi_{i\mu}|\leq s}\left|\partial_{i\mu}^{\ell+1} \left((\wt G\BB)_{\uu\mu}M(\wt G)_{\uu\vh}^{e-1}\overline{M(\wt G)}_{\uu\vh}^{e}\right)\right|}\\
&+C_\ell\expect\qb{|x_{i\mu}|^{\ell+2}\mathds{1}(|x_{i\mu}|>s)}\expect\qbb{\sup_{\xi_{i\mu}\in \mathbb{R}}\left|\partial_{i\mu}^{\ell+1} \left((\wt G\BB)_{\uu\mu}M(\wt G)_{\uu\vh}^{e-1}\overline{M(\wt G)}_{\uu\vh}^{e}\right)\right|}
\end{aligned}
\end{equation}
for any $s>0$. Here $\xi_{i\mu}$ is a real {\it deterministic} parameter, and $\wt G$ is obtained from $G$ by replacing $x_{i\mu}$ with $\xi_{i\mu}$. 
The proof of $\cal M\prec\Phi$ is divided into several parts, summarized in the following lemma.
\begin{lemma}\label{task1}
We have the following estimates.
\begin{enumerate}
    \item $\expect\qB{\pB{zm_{2}(z)(G\AA^2)_{\uu\vh}+\frac{z}{n}(G\BB^2G\AA^2)_{\uu\vh}}M(G)_{\uu\vh}^{e-1}\overline{M(G)}_{\uu\vh}^{e}}+X_1\prec\Phi^2\cal M^{2e-2}$.
    \item For $k \geq 2$, $X_k\prec\sum_{s=1}^{2e}\Phi^s\cal M^{2e-s}$.
    \item For any fixed $D>0$, there exists a constant $\ell \equiv \ell(D) \geq 1$  such that $\cal R_{\ell+1}^{i\mu}=\OO(n^{-D})$ uniformly in $i,\mu$.
\end{enumerate}
\end{lemma}

Combining Lemma \ref{task1} with (\ref{aims}), with $\ell\equiv\ell(e+2)$, we obtain
\begin{equation}
\mathcal{M}^{2e}\prec \sum_{s=1}^{2e} \Phi^{s}\mathcal{M}^{2e-s} +\OO(n^{-e}).\label{eq_pre_young}
\end{equation}
For any fixed $\epsilon>0$, Young's inequality gives
\begin{align*}
 \Phi^{s}\mathcal{M}^{2e-s} \leq 
 \frac{2e-s}{2e}(n^{-\epsilon}\mathcal{M}^{2e-s})^{\frac{2e}{2e-s}}+\frac{s}{2e}(n^{\epsilon}\Phi^i)^{\frac{2e}{s}}
 \prec n^{-\epsilon}\mathcal{M}^{2e} + n^{2e\epsilon}\Phi^{2e}.
\end{align*}
Since $\Phi\ge q+\Theta\ge n^{-1/2}$, we have $n^{-e}\leq\Phi^{2e}$. Substituting these estimates into \eqref{eq_pre_young} yields
$$\mathcal{M}^{2e} \prec n^{-\epsilon}\mathcal{M}^{2e} + n^{2e\epsilon}\Phi^{2e} \implies \mathcal{M}^{2e} \prec n^{2e\epsilon}\Phi^{2e}.$$
By Markov's inequality, this implies $\mathcal{M}\prec n^\epsilon \Phi$. Since $\epsilon>0$ is arbitrary, we conclude that $\mathcal{M}\prec \Phi$, which completes the proof of the first estimate in \eqref{eq_anisotropic_average}. 

The rest of this subsection is devoted to the proof of Lemma \ref{task1}.

\begin{proof}[Proof of Lemma \ref{task1} (i)]
We compute
\begin{equation}\nonumber
\begin{aligned}
X_1&=z^{1/2}\sum_{i,\mu}\kappa_2(x_{i\mu})\expect\qB{\AA_{i\vh}\partial_{i\mu} \left((G\BB)_{\uu\mu}M(G)_{\uu\vh}^{e-1}\overline{M(G)}_{\uu\vh}^{e}\right)}\\
&=\frac{z^{1/2}}{n}\sum_{i,\mu}\expect\qB{\AA_{i\vh}\partial_{i\mu}(G\BB)_{\uu\mu} \left(M(G)_{\uu\vh}^{e-1}\overline{M(G)}_{\uu\vh}^{e}\right)} +\frac{z^{1/2}}{n}\sum_{i,\mu}\expect\qB{\AA_{i\vh}(G\BB)_{\uu\mu}\partial_{i\mu} \left(M(G)_{\uu\vh}^{e-1}\overline{M(G)}_{\uu\vh}^{e}\right)}.
\end{aligned}
\end{equation}
We emphasize that a crucial cancellation occurs here. A direct computation shows that
$$\frac{z^{1/2}}{n}\sum_{i,\mu}\partial_{i\mu}(G\BB)_{\uu\mu}\AA_{i\vh}+zm_{2}(z)(G\AA^2)_{\uu\vh}+\frac{z}{n}(G\BB^2G\AA^2)_{\uu\vh}=0,$$
which yields
\begin{equation}\label{X_1aim}
\begin{aligned}
&~\expect\qB{\pB{zm_{2}(z)(G\AA^2)_{\uu\vh}+\frac{z}{n}(G\BB^2G\AA^2)_{\uu\vh}}M(G)_{\uu\vh}^{e-1}\overline{M(G)}_{\uu\vh}^{e}}+X_1\\
=&~\frac{z^{1/2}}{n}\sum_{i,\mu}\expect\qB{\AA_{i\vh}(G\BB)_{\uu\mu}\partial_{i\mu} \left(M(G)_{\uu\vh}^{e-1}\overline{M(G)}_{\uu\vh}^{e}\right)}.
\end{aligned}
\end{equation}
The partial derivative on the right-hand side of (\ref{X_1aim}) may act on either $M(G)_{\uu\vh}^{e-1}$ or $\overline{M(G)}_{\uu\vh}^{e}$. We first consider the former case. By (\ref{M(G)'}), we have
\begin{equation}
\nonumber
\begin{aligned}
&\frac{z^{1/2}}{n}\sum_{i,\mu}(e-1)\expect\qB{(G\BB)_{\uu\mu}\AA_{i\vh}\partial_{i\mu}\pb{M(G)_{\uu\vh}}M(G)_{\uu\vh}^{e-2}\overline{M(G)}_{\uu\vh}^{e}}\\
=& -\frac{z^2}{n}(e-1)\expect\Bigl[M(G)_{\uu\vh}^{e-2}\overline{M(G)}_{\uu\vh}^{e}\Bigl((G\AA^2)_{\uu\vh}(G\BB^2 G)_{\uu\vh}+(G\BB^2 G)_{\uu\uu}(\AA^2 G)_{\vh\vh}\\
&\quad +m_2(z)(G\AA^2)_{\uu\vh}(G\BB^2 G\AA^2)_{\uu\vh}+m_2(z)(G\BB^2 G)_{\uu\uu}(\AA^2 G\AA^2)_{\vh\vh}\\
&\quad +\frac{1}{n}(G\BB^2 G\BB^2G\AA^2)_{\uu\vh}(G\AA^2)_{\uu\vh}+\frac{1}{n}(\AA^2 G\BB^2G\BB^2 G)_{\vh\uu}(G\AA^2)_{\uu\vh}\\
&\quad +\frac{1}{n}(G\AA^2)_{\uu\vh}(G\BB^2 G\BB^2G\AA^2)_{\uu\vh}+\frac{1}{n}(G\BB^2 G)_{\uu\uu}(\AA^2 G\BB^2G\AA^2)_{\vh\vh}\\
&\quad +\frac{1}{n}(G\BB^2 G\AA^2)_{\uu\vh}(G\BB^2 G\AA^2)_{\uu\vh}+\frac{1}{n}(G\BB^2 G\BB^2 G)_{\uu\uu}(\AA^2 G\AA^2)_{\vh\vh}\Bigl)\Big]\\
\prec&(1+\phi)^2\Theta^2\expect[\abs{M(G)_{\uu\vh}}^{2e-2}]\leq\Phi^2\mathcal{M}^{2e-2},
\end{aligned}
\end{equation}
where we used \(m_2(z)=\frac{1}{n}\sum_\nu(G\BB^2)_{\nu\nu}\prec1+\phi\) and Lemma \ref{multipleG} for the stochastic bound, and Hölder's inequality in the final step. The case where the partial derivative acts on \smash{$\overline{M(G)}_{\uu\vh}^{e}$} is handled in the same way, except that some factors $G$ are replaced by $\overline{G}$; this does not affect the proof since Lemma \ref{multipleG} still applies.
\end{proof}

\begin{proof}[Proof of Lemma \ref{task1} (ii)]
For $k\ge 2$, we write
\begin{align}
\label{xk}
X_{k} &= \frac{z^{1/2}}{k!}\sum_{i,\mu} \sum_{\substack{r,s,t\ge0\\ r+s+t=k}}{\kappa_{k+1}(x_{i\mu})} \AA_{i\vh}\expect\qB{\partial_{i\mu}^{r} (G\BB)_{\uu\mu}\cdot\partial_{i\mu}^{s}M(G)_{\uu\vh}^{e-1}\cdot\partial_{i\mu}^{t}\overline{M(G)}_{\uu\vh}^{e}}.
\end{align}
It suffices to control each term with fixed $r,s,t$. Since the complex conjugates play no role in the following analysis, we again drop them to simplify notation and estimate the quantities
\begin{equation}
\sum_{i,\mu}\kappa_{k+1}(x_{i\mu})\AA_{i\vh} \expect\qB{\partial_{i\mu}^{r} (G\BB)_{\uu\mu}\cdot\partial_{i\mu}^{k-r}M(G)_{\uu\vh}^{2e-1}}
\nonumber
\end{equation}
for $r=0,1,\cdots,k$. Each such quantity can be expanded further into a finite sum of terms
\begin{equation}\nonumber
\begin{aligned}
\sum_{i,\mu}\kappa_{k+1}(x_{i\mu})\AA_{i\vh} \expect\qB{\partial_{i\mu}^{r} (G\BB)_{\uu\mu}\cdot\pB{\prod_{m=1}^{w} \partial_{i\mu}^{l_m}M(G)_{\uu\vh}}M(G)_{\uu\vh}^{2e-1-w}},
\end{aligned}
\end{equation}
where the sum ranges over $w = 0,1,\cdots,(k-r) \wedge (2e-1)$ and integers $ l_{1},\cdots, l_{w}\ge 1$ with $l_{1} + \cdots + l_{w} = k-r$. We aim to bound each such term by $\Phi^{w+1}\mathcal{M}^{2e-w-1}$. By \eqref{eq:kappa_k+1} and Hölder's inequality, it is enough to prove
\begin{equation}\label{xkaim}
 n^{-3/2}q^{k-2} \sum_{i,\mu}\left |\AA_{i\vh}\partial_{i\mu}^{r} (G\BB)_{\uu\mu}\cdot\pB{\prod_{m=1}^{w} \partial_{i\mu}^{l_m}M(G)_{\uu\vh}}\right |\prec\Phi^{w+1}.
\end{equation}
Note that by (\ref{Gderivative}), $\partial_{i\mu}^{r}(G\BB)_{\uu\mu}$ is a scalar multiple of a sum of $2^{r}$ terms, each of which is a product of $r+1$ entries of the form $(\Ga_1G\Ga_2)_{\xx\yy}$, where $\Ga_1,\Ga_2\in\{\AA,\BB,I_{p+n}\},\ \xx,\yy\in\{\uu,i,\mu\}$. Moreover, at least one entry is of the form $(G\AA)_{\uu i}$ or $(G\BB)_{\uu\mu}$. 
For any $r\ge 0$, Lemma \ref{multipleG}, the Cauchy-Schwarz inequality, and \eqref{ward2} give
\begin{equation}
\begin{aligned}
\sum_{i,\mu}\left |\AA_{i\vh}\partial_{i\mu}^{r}(G\BB)_{\uu\mu}\right|\prec (1+\phi)^{r}\sum_{i,\mu}\Bigl|\AA_{i\vh}\pb{\left|(G\AA)_{\uu i}\right|+\left|(G\BB)_{\uu\mu}\right|}\Bigr|
\prec(1+\phi)^{r}n^{3/2}\Theta.
\end{aligned}
\nonumber
\end{equation}
Combining this estimate with (\ref{roughbound}), we obtain, for $w\leq k-2$, that
\begin{equation}
\begin{aligned}
&\quad  n^{-3/2}q^{k-2} \sum_{i,\mu}\left |\AA_{i\vh}\partial_{i\mu}^{r} (G\BB)_{\uu\mu}\cdot\pB{\prod_{m=1}^{w} \partial_{i\mu}^{l_m}M(G)_{\uu\vh}}\right|\\ &\prec n^{-3/2}q^{k-2}\cdot (1+\phi)^{r}n^{3/2}\Theta\cdot (1+\phi)^{k-r+2w}=q^{k-2}(1+\phi)^{k+2w}\Theta\\
&\prec\pb{(1+\phi)^{3}q}^{k-2}\pb{(1+\phi)^{3}\Theta} \prec\pb{(1+\phi)^{3}(q+\Theta)}^{w+1} = \Phi^{w+1},
\end{aligned}
\nonumber
\end{equation}
where we also used $(1+\phi)^3q\leq 1$ in the penultimate step. It remains to estimate the left-hand side of (\ref{xkaim}) for $w \ge k-1$, which we assume from now on. Since $w\le k-r$, we must have $r = 0 $ or $1$. Thus, it remains to consider the three cases $(r,w) = (0,k)$, $(r,w) = (1,k-1)$, and $(r,w) = (0, k-1)$, which we treat separately.

\medskip
\noindent {\bf Case 1}: $(r,w) = (0,k)$. In this case, we must have $l_{1} = l_{2} = \cdots = l_{k} = 1$. We estimate the left-hand side of (\ref{xkaim}) as follows:
\begin{equation}
\label{M(G)case1}
\begin{aligned}
&~ n^{-3/2}q^{k-2} \sum_{i,\mu}\left |\AA_{i\vh}(G\BB)_{\uu\mu}\pb{ \partial_{i\mu}M(G)_{\uu\vh}}^k\right|\\
\prec&~ n^{-3/2}q^{k-2}(1+\phi)^{3k-3}\sum_{i,\mu}\abs{\AA_{i\vh}(G\BB)_{\uu\mu}}\Bigl(|(G\AA)_{\uu i}(\BB G)_{\mu\vh}|+|(G\BB)_{\uu\mu}(\AA G)_{i\vh}|\\
&~+|m_2(z)(G\AA)_{\uu i}(\BB G\AA^2)_{\mu\vh}|+|m_2(z)(G\BB)_{\uu\mu}(\AA G\AA^2)_{i\vh}|+(1+\phi)\Theta^2\Bigr).
\end{aligned}
\end{equation}
Here, we applied the trivial bound $|\partial_{i\mu}M(G)_{\uu\vh}|\prec(1+\phi)^3$ from \eqref{roughbound} to $k-1$ copies of $\partial_{i\mu}M(G)_{\uu\vh}$; for the remaining copy of $\partial_{i\mu}M(G)_{\uu\vh}$, we use the expansion in (\ref{M(G)'}) and control the last six terms on the right-hand side by $(1+\phi)\Theta^2$ using Lemma \ref{multipleG}. 
The first four terms on the right-hand side of (\ref{M(G)case1}) can be bounded by analogous arguments; we illustrate this by estimating the third term. By the Cauchy-Schwarz inequality, Lemma \ref{multipleG}, and \eqref{ward2}, we have
\begin{equation}
\nonumber
\begin{aligned}
&~\sum_{i,\mu}\abs{\AA_{i\vh}(G\BB)_{\uu\mu}}\cdot|m_2(z)(G\AA)_{\uu i}(\BB G\AA^2)_{\mu\vh}|\\
\prec&~(1+\phi)\pB{\sum_i|\AA_{i\vh}(G\AA)_{\uu i}|}\pB{\sum_\mu|(G\BB)_{\uu\mu}(\BB G\AA^2)_{\mu\vh}|} \prec(1+\phi)n^{3/2}\Theta^3.
\end{aligned}
\end{equation}
Substituting the above bounds into \eqref{M(G)case1}, we obtain
\begin{equation}
\begin{aligned}
& n^{-3/2}q^{k-2} \sum_{i,\mu}\left |\AA_{i\vh}(G\BB)_{\uu\mu}\pb{ \partial_{i\mu}M(G)_{\uu\vh}}^k\right|\\
\prec &~n^{-3/2}q^{k-2}(1+\phi)^{3k-3}\cdot(1+\phi)n^{3/2}\Theta^3
= (1+\phi)^{3k-2}q^{k-2}\Theta^3\\
\prec &~\pb{(1+\phi)^{3}(q+\Theta)}^{k+1} = \Phi^{w+1}.
\end{aligned}
\nonumber
\end{equation}

\medskip
\noindent {\bf Case 2}: $(r,w) = (1, k-1)$. In this case, we must have $l_{1} = l_{2} = \cdots = l_{k-1} = 1$. We estimate the left-hand side of (\ref{xkaim}) as follows:
\begin{equation}
\begin{aligned}
&~ n^{-3/2}q^{k-2}\sum_{i,\mu}\left |\AA_{i\vh}\partial_{i\mu}(G\BB)_{\uu\mu}\pb{ \partial_{i\mu}M(G)_{\uu\vh}}^{k-1}\right|\\
\prec&~ n^{-3/2}q^{k-2}(1+\phi)^{3k-6}\sum_{i,\mu}\left|\AA_{i\vh}\partial_{i\mu}(G\BB)_{\uu\mu}\right|\Bigl(|(G\AA)_{\uu i}(\BB G)_{\mu\vh}|+|(G\BB)_{\uu\mu}(\AA G)_{i\vh}|\\
&~+|m_2(z)(G\AA)_{\uu i}(\BB G\AA^2)_{\mu\vh}|+|m_2(z)(G\BB)_{\uu\mu}(\AA G\AA^2)_{i\vh}|+(1+\phi)\Theta^2\Bigr)\\
\prec&~ n^{-3/2}q^{k-2}(1+\phi)^{3k-6}\cdot(1+\phi)^3n^{3/2}\Theta^2
\prec \pb{(1+\phi)^{3}(q+\Theta)}^{k} = \Phi^{w+1}.
\end{aligned}
\nonumber
\end{equation}
In the first step, as in Case 1, we applied the trivial bound $|\partial_{i\mu}M(G)_{\uu\vh}|\prec(1+\phi)^3$ to $k-2$ copies of $\partial_{i\mu}M(G)_{\uu\vh}$, and expanded the remaining copy using (\ref{M(G)'}).
For the second step, we illustrate the argument with one representative term. Applying \eqref{Gderivative} to $\partial_{i\mu}(G\BB)_{\uu\mu}$ yields
\begin{equation}
\nonumber
\begin{aligned}
&\sum_{i,\mu}\abs{\AA_{i\vh}\partial_{i\mu}(G\BB)_{\uu\mu}}\cdot|m_2(z)(G\AA)_{\uu i}(\BB G\AA^2)_{\mu\vh}|\\
\prec&\sum_{i,\mu}\abs{\AA_{i\vh}(G\AA)_{\uu i}(\BB G\BB)_{\mu\mu}}\cdot|m_2(z)(G\AA)_{\uu i}(\BB G\AA^2)_{\mu\vh}| +\sum_{i,\mu}\abs{\AA_{i\vh}(G\BB)_{\uu\mu}(\AA G\BB)_{i\mu}}\cdot|m_2(z)(G\AA)_{\uu i}(\BB G\AA^2)_{\mu\vh}|\\
\prec&(1+\phi)^3\sum_{i,\mu}\abs{\AA_{i\vh}(G\AA)_{\uu i}(\BB G\AA)_{\mu\vh}}+(1+\phi)^3\sum_{i,\mu}\abs{\AA_{i\vh}(G\BB)_{\uu\mu}(G\AA)_{\uu i}} \prec (1+\phi)^3n^{3/2}\Theta^2.
\end{aligned}
\end{equation}
Here we again used the Cauchy-Schwarz inequality, Lemma \ref{multipleG}, and \eqref{ward2} in the second step above. 

\medskip
\noindent {\bf Case 3}: $(r,w) = (0, k-1)$. 
In this case, without loss of generality, we assume that $l_{1} = l_{2} = \cdots = l_{k-2} = 1$ and $ l_{k-1}=2$. We estimate the left-hand side of (\ref{xkaim}) as follows:
\begin{equation}
\begin{aligned}
&~ n^{-3/2}q^{k-2} \sum_{i,\mu}\left |\AA_{i\vh}(G\BB)_{\uu\mu}\partial_{i\mu}^2M(G)_{\uu\vh}\pb{\partial_{i\mu}M(G)_{\uu\vh}}^{k-2}\right|\\
\prec&~ n^{-3/2}q^{k-2}(1+\phi)^{3k-6}\sum_{i,\mu}\left |\AA_{i\vh}(G\BB)_{\uu\mu}\partial_{i\mu}^2M(G)_{\uu\vh}\right|\\
\prec&~ n^{-3/2}q^{k-2}(1+\phi)^{3k-6}\cdot(1+\phi)^3n^{3/2}\Theta^2
\prec\pb{(1+\phi)^{3}(q+\Theta)}^{k} = \Phi^{w+1}.
\end{aligned}
\nonumber
\end{equation}
In the first step, we applied the trivial bound $|\partial_{i\mu}M(G)_{\uu\vh}|\prec(1+\phi)^3$ to all $k-2$ copies of $\partial_{i\mu}M(G)_{\uu\vh}$. For the second step, we claim that
\begin{equation}
  \begin{split}    
\abs{\partial_{i\mu}^{2}M(G)_{\uu\vh}}&\prec (1+\phi)^3\left(\left|(G\AA)_{\uu i}\right|+\left|(G\AA)_{\vh i}\right|+|(\AA^2 G\AA)_{\vh i}|\right)\\
&+(1+\phi)^3\left(\left|(\BB G)_{\mu\vh}\right|+\left|(\BB G)_{\mu\uu}\right|+|(\BB G\AA^2)_{\mu\vh}|+\Theta^{2}\right).
  \end{split}\label{eq_second_derivative}
\end{equation}
Combining \eqref{eq_second_derivative} with the Cauchy-Schwarz inequality and \eqref{ward2} gives the second step. 

To prove \eqref{eq_second_derivative}, recall from (\ref{M(G)'}) that $\partial_{i\mu}M(G)$ consists of ten terms. We illustrate the estimates by considering two representative derivatives. First, the derivative of the fifth term
$$\partial_{i\mu}\pB{\frac{1}{n}(\BB G\BB^2 G\AA^2)_{\mu i}(G\AA^2)_{\uu\vh}}$$
can be written as a sum of six terms, each of the form $$\frac{1}{n}(\Ga_1G\Ga_2G\Ga_3)_{\xx_1\yy_1}(\Ga_4G\Ga_5)_{\xx_2\yy_2}(\Ga_6G\Ga_7)_{\xx_3\yy_3},$$
where $\Ga_1,\cdots,\Ga_7\in\{\AA^2,\BB^2,\AA,\BB,I_{p+n}\},$ and $\xx_1,\cdots,\yy_3\in\{\uu,\vh,i,\mu\}.$ Thus, using Lemma \ref{multipleG}, we get
$$\left|\partial_{i\mu}\pB{\frac{1}{n}(\BB G\BB^2 G\AA^2)_{\mu i}(G\AA^2)_{\uu\vh}}\right|\prec(1+\phi)^2\Theta^2.$$
The derivatives of the sixth to tenth terms in (\ref{M(G)'}) are estimated similarly. Second, the derivative of the third term in (\ref{M(G)'}) can be estimated as follows:
\begin{equation}
\begin{aligned}
&~ \partial_{i\mu}\pB{m_2(z)(G\AA)_{\uu i}(\BB G\AA^2)_{\mu\vh}} \\
=&~ \partial_{i\mu}m_2(z)\pB{(G\AA)_{\uu i}(\BB G\AA^2)_{\mu\vh}}+\partial_{i\mu}(G\AA)_{\uu i}\pB{m_2(z)(\BB G\AA^2)_{\mu\vh}}+\partial_{i\mu}(\BB G\AA^2)_{\mu\vh}\pB{m_2(z)(G\AA)_{\uu i}} \\
\prec&~(1+\phi)^2|\partial_{i\mu}m_2(z)|+(1+\phi)^3|(\BB G\AA^2)_{\mu\vh}|+(1+\phi)^3|(G\AA)_{\uu i}| \prec (1+\phi)^3\pb{|(G\AA)_{\uu i}|+|(\BB G\AA^2)_{\mu\vh}|+\Theta^2},
\end{aligned}
\nonumber
\end{equation}
where we used \eqref{Gderivative} and Lemma \ref{multipleG} in the second step, together with
\begin{equation}
    \partial_{i\mu}m_2(z)=-\frac{z^{1/2}}{n}\pb{(\BB G\BB^2G\AA)_{\mu i}+(\AA G\BB^2G\BB)_{i\mu}}\prec\Theta^2 \label{eq:partial_m2}
\end{equation}
from \eqref{ward2} in the third step. The derivatives of the first three terms in (\ref{M(G)'}) are estimated similarly. This concludes the proof of Case 3.
\end{proof}

\begin{proof}[Proof of Lemma \ref{task1} (iii)]
For simplicity, we introduce the notation
\begin{equation}
\Delta^{i\mu}:=\partial_{i\mu}\begin{pmatrix}
  0&X \\
  X^*&0
\end{pmatrix},\quad \DD:=\partial_{i\mu}H=z^{1/2}\begin{pmatrix}
  A^{1/2}&0 \\
  0&B^{1/2}
\end{pmatrix}\Delta^{i\mu}\begin{pmatrix}
  A^{1/2}&0 \\
  0&B^{1/2}
\end{pmatrix}.\label{defn_DD}
\end{equation}
We then define the following resolvent, obtained from $G$ by setting the entry $x_{i\mu}$ in $X$ to 0:
$$ \hat G:=\pb{H-x_{i\mu}\DD-z}^{-1}$$
By the definition of $\wt G$, we can write
$$\wt G\equiv \wt G(\xi_{i\mu})=\pb{H-(x_{i\mu}-\xi_{i\mu})\DD-z}^{-1}.$$ 
For any $K\in\N_+$, we have the resolvent expansions
\begin{align}
\label{GtohatG}
&\hat G=\sum_{k=0}^{K-1}(x_{i\mu}G\DD)^kG+(x_{i\mu}G\DD)^K\hat G,\\
\label{hatGtowtG}
&\wt G=\sum_{k=0}^{K-1}(-\xi_{i\mu}\hat G\DD)^k\hat G+(-\xi_{i\mu}\hat G\DD)^K\wt G.
\end{align}
Note that for any deterministic unit vectors $\xx,\yy\in \C^{p+n}$, $\Delta^{i\mu}$ has only two nonzero components, so $\DD$ has rank at most 2. Hence, $\p{(x_{i\mu}G\DD)^kG}_{\xx\yy}$ can be written as $z^{k/2}x_{i\mu}^k$ times a sum of $2^k$ terms, each of which is a product of $k+1$ entries of the form $(\Ga_1G\Ga_2)_{\xx'\yy'}$, where
$\Ga_1,\Ga_2\in\{\AA,\BB,I_{p+n}\}$ and $\xx',\yy'\in\{\xx,\yy,i,\mu\}.$ Thus, by the bounded support condition \eqref{eq_support} and Lemma \ref{multipleG}, we obtain
$$\pb{(x_{i\mu}G\DD)^kG}_{\xx\yy}\prec q^k(1+\phi)^{k+1}\prec 1+\phi.$$ 
Similarly, we have $$\pb{(x_{i\mu}G\DD)^K\hat G}_{\xx\yy}\prec q^K(1+\phi)^K\norm{\hat G}.$$ 
Using the trivial bound $\norm{\hat G}\leq (\Im z^{1/2})^{-1}=\OO(n)$ and choosing $K$ sufficiently large so that $\p{q(1+\phi)}^K\leq n^{-1}$, we obtain $\pb{(x_{i\mu}G\DD)^K\hat G}_{\xx\yy}\prec 1.$ It then follows from (\ref{GtohatG}) that
\begin{equation}
\hat G_{\xx\yy} \prec 1+\phi.\label{eq_hatG}
\end{equation}
Using \eqref{eq_hatG} and the same argument, we derive from \eqref{hatGtowtG} that
\begin{equation}
\sup_{|\xi_{i\mu}|\leq q^{1/2}} \wt G_{\xx\yy}(\xi_{i\mu}) \prec 1+\phi.\label{eq_wtG}
\end{equation}

Following the same argument as in the proof of Lemma \ref{lem:roughbound}, and using \eqref{eq_wtG} together with the trivial bound $\norm{\wt G}=\OO(n)$, we obtain, for any $r\in\N$,
$$\sup_{|\xi_{i\mu}|\leq q^{1/2}}\left|\partial_{i\mu}^{r}M(\wt G)_{\uu\vh}\right|\prec(1+\phi)^{r+2},\quad
\sup_{\xi_{i\mu}\in\R}\left|\partial_{i\mu}^{r}M(\wt G)_{\uu\vh}\right|=\OO(n^{r+2}).$$
Consequently, taking $s=q^{1/2}$ in \eqref{remainder^imu}, we obtain
\begin{equation}\nonumber
\begin{aligned}
\sup_{|\xi_{i\mu}|\leq q^{1/2}}\left|\partial_{i\mu}^{\ell+1} \left((\wt G\BB)_{\uu\mu}M(\wt G)_{\uu\vh}^{e-1}\overline{M(\wt G)}_{\uu\vh}^{e}\right)\right|&\prec (1+\phi)^{\ell+4e},\\
\sup_{\xi_{i\mu}\in \mathbb{R}}\left|\partial_{i\mu}^{\ell+1} \left((\wt G\BB)_{\uu\mu}M(\wt G)_{\uu\vh}^{e-1}\overline{M(\wt G)}_{\uu\vh}^{e}\right)\right|&=\OO(n^{\ell+4e}).
\end{aligned}
\end{equation}
Substituting these bounds into (\ref{remainder^imu}), we get
\begin{equation}\nonumber
\begin{aligned}
\abs{\cal R_{\ell+1}^{i\mu}}&\prec \expect\qb{|x_{i\mu}|^{\ell+2}}(1+\phi)^{\ell+4e}+\expect\qb{|x_{i\mu}|^{\ell+2}\mathds{1}(|x_{i\mu}|>q^{1/2})}n^{\ell+4e}\\
&\prec q^{\ell+2}(1+\phi)^{\ell+4e}+q^{\ell+2}n^{\ell+4e}\prob\qb{|x_{i\mu}|>q^{1/2}}^{1/2}\prec n^{-D}
\end{aligned}
\end{equation}
uniformly in $i,\mu$, where the last step follows by choosing $\ell$ sufficiently large (depending on $D$) and using $\max_{i,\mu}|x_{i\mu}|\prec q$.
\end{proof}

\subsection{Averaged Estimation}

In this subsection, we prove the second assertion of Theorem \ref{maintheorem}, namely, \smash{$\abs{\underline{\Gamma M(G)}}\prec\Phi^{2}$} for any deterministic matrix $\Gamma\in\C^{(p+n)\times (p+n)}$ with $\norm{\Gamma}=\OO(1)$. We decompose $\Ga$ into a $p\times (p+n)$ upper block and an $n\times (p+n)$ lower block. By symmetry, it suffices to consider the case where $\Ga$ has only an upper block, that is, $\ut\Ga=0$ for all $\uu\in\C^{p+n}$ (recall \eqref{eq_projection}). The case where $\Ga$ has only a lower block is completely analogous. Then, by (\ref{M(G)expansion}), we can write
\begin{equation}
\label{expansion3}
\begin{aligned}
\underline{\Ga M(G)}=\frac{1}{p+n}\sum_{j,a}\Ga_{ja}M(G)_{aj}
=\frac{z^{1/2}}{p+n}\sum_{i,\mu}(\AA\Ga G\BB)_{i\mu}x_{i\mu}+zm_{2}(z)\underline{\Ga G\AA^2}+\frac{z}{n}\underline{\Ga G\BB^2 G\AA^2}.
\end{aligned}
\end{equation}
For any deterministic matrix $\Ga_1\in\C^{(p+n)\times (p+n)}$, we have
\begin{equation}\label{TrGderrule}
\begin{aligned}
\partial_{i\mu}\Tr(\Ga_1G)=\partial_{i\mu}\Tr(\Ga_1^TG)=-z^{1/2}\pb{(\BB G\Ga_1 G\AA)_{\mu i}+(\AA G\Ga_1 G\BB)_{i\mu}}.
\end{aligned}
\end{equation}
Therefore, differentiating the expression in (\ref{expansion3}) and using \eqref{eq:partial_m2} and \eqref{TrGderrule} yields
\begin{equation}
\label{partialbmg}
\begin{aligned}
\partial_{i\mu}\underline{\Ga M(G)}
=&\frac{z^{1/2}}{p+n}(\AA \Ga G\BB)_{i\mu}-\frac{z}{p+n}\sum_{j,\nu}\pb{(\AA\Ga G\AA)_{ji}(\BB G\BB)_{\mu\nu}+(\AA\Ga G\BB)_{j\mu}(\AA G\BB)_{i\nu}}x_{j\nu}\\
&+\partial_{i\mu}\pB{zm_{2}(z)\underline{\Ga G\AA^2}+\frac{z}{n}\underline{\Ga G\BB^2 G\AA^2}}\\
=&\frac{z^{1/2}}{p+n}(\AA\Ga G\BB)_{i\mu}-z^{1/2}\pb{(\BB GFG\AA)_{\mu i}+(\AA GFG\BB)_{i\mu}},
\end{aligned}
\end{equation}
where
\begin{align}
F_{\uu\vv}:=&\frac{1}{p+n}\pB{z^{1/2}\sum_{j,\nu}\BB_{\uu\nu}(\AA\Ga)_{j\vv}x_{j\nu}+
zm_2(z)(\AA^2\Ga)_{\uu\vv}+\frac{z}{n}\pb{\Tr(\Ga G\AA^2)\BB^2+\BB^2G\AA^2\Ga+(\AA^2\Ga G\BB^2)^T}_{\uu\vv}} \nonumber\\
=& \frac{1}{p+n}\pB{H\Ga+
zm_2(z)\AA^2\Ga+\frac{z}{n}\pb{\Tr(\Ga G\AA^2)\BB^2+\BB^2G\AA^2\Ga+\BB^2G\Ga^T\AA^2}}_{\uu\vv}.\label{f}
\end{align}
We remark that, on the right-hand side of (\ref{f}), each term is interchangeable with its transpose due to the structure of (\ref{TrGderrule}). The current form of $F_{\uu\vv}$ is chosen carefully to produce a crucial cancellation that will be identified later; see the proof of Lemma \ref{qlemma} (ii) below.

We now return to the proof of the estimate $\abs{\underline{\Gamma M(G)}}\prec\Phi^{2}$. By Lemma \ref{momentmethod}, it suffices to prove that
$$\mathcal{N}:=\expect\qb{|\underline{\Ga M(G)}|^{2e}}^{\frac{1}{2e}}\prec\Phi^2$$
for any fixed $e\in\N_+$ with $e\ge 4$. Using \eqref{expansion3} and applying the cumulant expansion (Lemma \ref{lem:cumulant_expansion}) to $x_{i\mu}$, we expand $\mathcal{N}^{2e}$ as
\begin{equation}
\label{aims2}
\begin{aligned}
\mathcal{N}^{2e}=&\expect\qB{\pB{zm_2(z)\underline{\Ga G\AA^2}+\frac{z}{n}\underline{\Ga G\BB^2 G\AA^2}}\underline{\Ga M(G)}^{e-1}\overline{\underline{\Ga M(G)}}^{e}}+\sum_{k=1}^\ell Y_k+\sum_{i,\mu}\cal R'^{i\mu}_{\ell+1}.
\end{aligned}
\end{equation}
where, for any fixed $k\in\N$,
$$Y_k:=\frac{z^{1/2}}{p+n}\sum_{i,\mu}\frac{\kappa_{k+1}(x_{i\mu})}{k!}\expect\qB{\partial_{i\mu}^{k} \left((\AA\Ga G\BB)_{i\mu}\underline{\Ga M(G)}^{e-1}\overline{\underline{\Ga M(G)}}^{e}\right)},$$
$\ell$ is a positive integer to be chosen later in the proof, and the remainder terms \smash{$\cal R'^{i\mu}_{\ell+1}$} are defined analogously to $\cal R^{i\mu}_{\ell+1}$ in (\ref{remainder^imu}).
The proof of $\cal N\prec\Phi^2$ is then reduced to the following lemma, in the same way that the proof of $\cal M\prec\Phi$ is reduced to Lemma \ref{task1}; see the argument around \eqref{eq_pre_young}.

\begin{lemma}\label{task2}
We have the following estimates.
\begin{enumerate}
    \item $\expect\qB{\pB{zm_2(z)\underline{\Ga G\AA^2}+\frac{z}{n}\underline{\Ga G\BB^2 G\AA^2}}\underline{\Ga M(G)}^{e-1}\overline{\underline{\Ga M(G)}}^{e}}+Y_1\prec\Phi^4\cal N^{2e-2}$.
    \item For $k \geq 2$, $Y_k\prec\sum_{s=1}^{2e}\Phi^{2s}\cal M^{2e-s}$.
    \item For any fixed $D>0$, there exists a constant $\ell \equiv \ell(D) \geq 1$  such that $\cal R'^{i\mu}_{\ell+1}=\OO(n^{-D})$ uniformly in $i,\mu$.
\end{enumerate}
\end{lemma}

Set $Q:=GFG$, where $F$ is defined in \eqref{f}. The following bounds on $Q$ and its derivatives are the key ingredients in the proof of Lemma \ref{task2}.
\begin{lemma}
\label{qlemma}
We have the following estimates.
\begin{enumerate}
    \item For any fixed integer $l\ge 0$, we have $\partial_{i \mu}^{l}Q \prec (1+\phi)^{l+1}\Theta^{2}$.
    \item $Q\prec(1+\phi)^{4}(q+\Theta)^{3}.$
\end{enumerate}
\end{lemma}

We emphasize that the bound in (ii) for $l=0$ is stronger than the $l=0$ case of the bound in (i). As observed in \cite{isotropic}, its proof relies on a crucial cancellation that is not present for $l>0$. In fact, the special form of the derivatives of $Q$ allows us to derive self-improving bounds on $Q$, which can then be iterated to obtain the optimal bound stated in Lemma \ref{qlemma} (ii). We postpone the proof of Lemma \ref{qlemma} to the next subsection. 
The rest of this subsection is devoted to proving Lemma \ref{task2} based on Lemma \ref{qlemma}.

\begin{proof}[Proof of Lemma \ref{task2} (i)]
As in the proof of Lemma \ref{task1} (i), a crucial cancellation occurs here. A direct computation shows that
$$\frac{z^{1/2}}{n(p+n)}\sum_{i,\mu}\partial_{i\mu}(\AA\Ga G\BB)_{i\mu}+zm_2(z)\underline{\Ga G\AA^2}+\frac{z}{n}\underline{\Ga G\BB^2 G\AA^2}=0,$$
which yields
\begin{equation}\label{Y_1aim}
\begin{aligned}
&~\expect\qB{\pB{zm_2(z)\underline{\Ga G\AA^2}+\frac{z}{n}\underline{\Ga G\BB^2 G\AA^2}}\underline{\Ga M(G)}^{e-1}\overline{\underline{\Ga M(G)}}^{e}}+Y_1\\
=&~\frac{z^{1/2}}{n(p+n)}\sum_{i,\mu}\expect\qB{(\AA\Ga G\BB)_{i\mu}\partial_{i\mu}\left(\underline{\Ga M(G)}^{e-1}\overline{\underline{\Ga M(G)}}^{e}\right)}.
\end{aligned}
\end{equation}
By Hölder's inequality, it suffices to prove that
$$\frac{z^{1/2}}{n(p+n)}\sum_{i,\mu}\left|(\AA\Ga G\BB)_{i\mu}\partial_{i\mu}\underline{\Ga M(G)}\right|\prec\Phi^4.$$
By (\ref{partialbmg}), this estimate follows from
\begin{align*}
  n^{-3}\sum_{i,\mu}|(\AA\Ga G\BB)_{i\mu}|^2\prec n^{-1}\Theta^2\prec\Theta^4  ,
\end{align*}
where we applied \eqref{ward2} in the first step, and from
$$n^{-2}\sum_{i,\mu}\left|(\AA\Ga G\BB)_{i\mu}\pb{(\BB Q\AA)_{\mu i}+(\AA Q\BB)_{i\mu}}\right|\prec n^{-2}\Phi^3\sum_{i,\mu}|(\AA\Ga G\BB)_{i\mu}|\prec\Phi^3\Theta\prec\Phi^4,$$
where we used Lemma \ref{qlemma} (ii) in the first step, and the Cauchy-Schwarz inequality together with \eqref{ward2} in the second step.
\end{proof}

\begin{proof}[Proof of Lemma \ref{task2} (ii)]
For $k\ge 2$, we have
\begin{equation}
\label{yk}
\begin{aligned}
Y_{k}\lesssim n^{-5/2}q^{k-2} \sum_{i,\mu}\sum_{r,s,t\ge 0,r+s+t=k}\expect\left|\partial_{i\mu}^{r}(\AA\Ga G\BB)_{i\mu}\cdot \partial_{i\mu}^{s}\underline{\Ga M(G)}^{e-1}\cdot \partial_{i\mu}^{t}\overline{\underline{\Ga M(G)}}^{e}\right|.
\end{aligned}
\end{equation}
It suffices to control each term with fixed $r,s,t$. Since the complex conjugates play no role in the following analysis, we again drop them to simplify notation and estimate
\begin{equation}
\begin{aligned}
n^{-5/2}q^{k-2}\sum_{i,\mu}\expect\left |\partial_{i\mu}^{r} (\AA\Ga G\BB)_{i\mu} \cdot  \partial_{i\mu}^{k-r}\underline{\Ga M(G)}^{2e-1}\right |
\end{aligned}
\nonumber
\end{equation}
for $r=0,1,\cdots,k$. Each such quantity can be further expanded into a finite sum of terms
\begin{equation}\nonumber
\begin{aligned}
n^{-5/2}q^{k-2}\sum_{i,\mu}\expect\left |\partial_{i\mu}^{r} (\AA\Ga G\BB)_{i\mu} \pB{\prod_{m=1}^{w} \partial_{i\mu}^{l_m}\underline{\Ga M(G)}}\underline{\Ga M(G)}^{2e-1-w}\right |,
\end{aligned}
\end{equation}
where the sum ranges over $w = 0,1,\cdots,(k-r) \wedge (2e-1)$ and integers $ l_{1},\cdots, l_{w} > 0$ with $l_{1} + \cdots + l_{w} = k-r$.
We aim to bound each such term by $\Phi^{2w+2}\mathcal{N}^{2e-w-1}.$ By Hölder's inequality, it is enough to show
\begin{equation}
\label{ykaim}
\begin{aligned}
 n^{-5/2}q^{k-2} \sum_{i,\mu}\expect\left |\partial_{i\mu}^{r} (\AA\Ga G\BB)_{i\mu} \pB{\prod_{m=1}^{w} \partial_{i\mu}^{l_m}\underline{\Ga M(G)}}\right |\prec\Phi^{2w+2}.\\
\end{aligned}
\end{equation}

For any fixed $t \ge 1$, by (\ref{partialbmg}) we have
$$\partial_{i\mu}^{t}\underline{\Ga M(G)}
=\partial_{i\mu}^{t-1}\qa{\frac{z^{1/2}}{p+n}(\AA\Ga G\BB)_{i\mu}-z^{1/2}\pb{(\BB Q\AA)_{\mu i}+(\AA Q\BB)_{i\mu}}}.$$
Bounding the partial derivatives of $G$ using \eqref{Gderivative} and Lemma \ref{multipleG}, and controlling the derivatives of $Q$ using Lemma \ref{qlemma} (i), we obtain
\begin{equation}
\label{TrGaMGder0}\left|\partial_{i\mu}^{t-1}(\AA\Ga G\BB)_{i\mu}\right|\prec(1+\phi)^{t},\quad \left|\partial_{i\mu}^{t-1}\pb{(\BB Q\AA)_{\mu i}+(\AA Q\BB)_{i\mu}}\right|\prec (1+\phi)^{t+1}\Theta^{2}.\end{equation}
Together with $(p+n)^{-1}\prec \Theta^2$, these estimates yield
\begin{equation}
\label{TrGaMGder}
\begin{aligned}
\left|\partial_{i\mu}^{t}\underline{\Ga M(G)}\right|\prec (1+\phi)^{t+1}\Theta^{2}.
\end{aligned}
\end{equation}
Substituting \eqref{TrGaMGder} and the first estimate in \eqref{TrGaMGder0} (with $t-1$ replaced by $r$) into (\ref{ykaim}), we conclude  
\begin{equation}
\nonumber
\begin{aligned}
 n^{-5/2}q^{k-2} \sum_{i,\mu}\expect\left |\partial_{i\mu}^{r} (\AA\Ga G\BB)_{i\mu} \pB{\prod_{m=1}^{w} \partial_{i\mu}^{l_m}\underline{\Ga M(G)}}\right | 
\prec&~ n^{-5/2}q^{k-2}n^2(1+\phi)^{r+1}\cdot (1+\phi)^{k-r+w}\Theta^{2w}\\
\prec&~ \pb{(1+\phi)^{3}q}^{k-1}\pb{(1+\phi)^3\Theta}^{2w}\prec\Phi^{2w+2}
\end{aligned}
\end{equation}
when $k \ge 3$, where in the second step we used $n^{-1/2}\leq q$, and in the last step we used  $(1+\phi)^3q\leq 1$. It remains to prove \eqref{ykaim} for $k=2$. Since $w\le k-r$, $w$ can only be 0, 1, or 2. We treat these cases separately.

\medskip
\noindent{\bf Case 1}: $w=0$. 
In this case, it remains to control
$n^{-5/2}\sum_{i,\mu}\left |\partial_{i\mu}^{2} (\AA\Ga G\BB)_{i\mu}\right|.$ 
By (\ref{Gderivative}), $\partial_{i\mu}^{2} (\AA\Ga G\BB)_{i\mu}$ can be written as a sum of eight terms, each of which is a product of three entries of the form $(\Ga_1G\Ga_2)_{\xx\yy}$, where
$\Ga_1,\Ga_2\in\{\AA\Ga,\AA,\BB\}$ and $ \xx,\yy\in\{i,\mu\}.$
Moreover, at least one of the three entries must satisfy $\{\xx,\yy\}=\{i,\mu\}$. Applying the Cauchy-Schwarz inequality and \eqref{ward2} to this entry, and applying Lemma \ref{multipleG} to the remaining two entries, we obtain
$$ n^{-5/2} \sum_{i,\mu}\left |\partial_{i\mu}^{2} (\AA\Ga G\BB)_{i\mu}\right|\prec n^{-1/2}(1+\phi)^2\Theta\prec\Phi^2.$$

\medskip
\noindent{\bf Case 2}: $w=1$. 
This case can be divided further into two subcases: $l_{1}=1$ and $l_{1}=2$. When $l_{1}=1$, using \eqref{partialbmg}, we estimate the left-hand side of (\ref{ykaim}) as follows:
\begin{equation}
\nonumber
\begin{aligned}
 n^{-5/2} \sum_{i,\mu}\left |\partial_{i\mu}(\AA\Ga G\BB)_{i\mu}\partial_{i\mu}\underline{\Ga M(G)}\right| 
&\prec n^{-5/2}(1+\phi)^2\sum_{i,\mu}\left |\frac{z^{1/2}}{p+n}(\AA\Ga G\BB)_{i\mu}-z^{1/2}\pb{(\BB Q\AA)_{\mu i}+(\AA Q\BB)_{i\mu}}\right|\\
&\prec n^{-5/2}(1+\phi)^2\pb{n\Theta+n^2(1+\phi)^4(q+\Theta)^3}\prec\Phi^4.
\end{aligned}
\end{equation}
Here, in the first step, we also used the rough bound  $\partial_{i\mu}(\AA\Ga G\BB)_{i\mu}\prec (1+\phi)^2$ from \eqref{Gderivative} and Lemma \ref{multipleG}; in the second step, we used \eqref{ward2} and Lemma \ref{qlemma} (ii). 
When $l_{1}=2$, using \eqref{TrGaMGder}, we estimate the left-hand side of (\ref{ykaim}) as follows:
\begin{equation}
\nonumber
\begin{aligned}
 n^{-5/2} \sum_{i,\mu}\left |(\AA\Ga G\BB)_{i\mu}\partial_{i\mu}^2\underline{\Ga M(G)}\right| 
&\prec n^{-5/2}(1+\phi)^3\Theta^2\sum_{i,\mu}|(\AA\Ga G\BB)_{i\mu}|\\
&\prec n^{-5/2}(1+\phi)^3\Theta^2\cdot n^{2}\Theta
\prec \Phi^{4}.
\end{aligned}
\end{equation}

\medskip
\noindent{\bf Case 3}: $w=2$. 
In this case, we must have $l_1=l_2=1$. Using \eqref{partialbmg}, we estimate the left-hand side of (\ref{ykaim}) as follows:
\begin{align*}
 n^{-5/2} \sum_{i,\mu}\left |(\AA\Ga G\BB)_{i\mu}\pb{\partial_{i\mu}\underline{\Ga M(G)}}^2\right| 
&\prec n^{-5/2}\sum_{i,\mu}|(\AA\Ga G\BB)_{i\mu}|\pB{n^{-2}|(\AA\Ga G\BB)_{i\mu}|^2+\left|\pb{(\BB Q\AA)_{\mu i}+(\AA Q\BB)_{i\mu}}\right|^2}\\
&\prec n^{-9/2}\sum_{i,\mu}|(\AA\Ga G\BB)_{i\mu}|^3+n^{-5/2}(1+\phi)^8(q+\Theta)^6\sum_{i,\mu}|(\AA\Ga G\BB)_{i\mu}|\\
&\prec n^{-5/2}(1+\phi)^2\Theta+n^{-1/2}(1+\phi)^9(q+\Theta)^6
\prec\Phi^6,
\end{align*}
where in the first step we used Lemma \ref{qlemma} (ii), and in the third step we used the Cauchy-Schwarz inequality, \eqref{ward2}, and Lemma \ref{multipleG}.
\end{proof}

\begin{proof}[Proof of Lemma \ref{task2} (iii)]
Using the definition (\ref{M(G)expansion}) and the identity (\ref{TrGderrule}), we derive the following alternative form of $\partial_{i\mu}\underline{\Ga M(G)}$:
\begin{equation}\label{partialbmg'}
\begin{aligned}
\partial_{i\mu}\underline{\Ga M(G)}=&\partial_{i\mu}\pB{\underline{\Ga}+z\underline{\Ga G}+zm_2(z)\underline{\Ga G\AA^2}+\frac{z}{n}\underline{\Ga G\BB^2 G\AA^2}}\\
=&-\frac{z^{3/2}}{p+n}\pb{(\BB G\hat F G\AA)_{\mu i}+(\AA G\hat F G\BB)_{i\mu}},
\end{aligned}
\end{equation}
where
$$\hat F:=\Ga+m_2(z)\AA^2\Ga+\frac{1}{n}\pb{\Tr(\Ga G\AA^2)\BB^2+\BB^2G\AA^2\Ga+\AA^2\Ga G\BB^2}.$$
Compared with (\ref{partialbmg}), the advantage of (\ref{partialbmg'}) is that it involves only entries of $G$, and no entries of $H$. This allows us to estimate \smash{$\partial_{i\mu}^r\underline{\Ga M(\wt G)}$}, and hence $\cal R'^{i\mu}_{\ell+1}$, in a way completely analogous to the proof of Lemma \ref{task1} (iii). We therefore omit the details.
\end{proof}

\subsection{Estimation of \texorpdfstring{$Q$}{Q}: Proof of Lemma \ref{qlemma}}

We fix arbitrary deterministic unit vectors $\uu,\vv\in\mathbf{S}$. 
We first collect several relevant quantities. By (\ref{f}), we can write
\begin{equation}\label{q}
\begin{aligned}
Q_{\uu\vv}=&(GFG)_{\uu\vv}=\frac{z^{1/2}}{p+n}\sum_{i,\mu}(G\BB)_{\uu\mu}(\AA\Ga G)_{i\vv}x_{i\mu}\\
&+\frac{z}{p+n} \pB{m_2(z)(G\AA^2\Ga G)+\frac{1}{n}\pb{\Tr(\Ga G\AA^2)G\BB^2 G+G\BB^2G\AA^2\Ga G+G\BB^2G\Ga^T\AA^2G}}_{\uu\vv},
\end{aligned}
\end{equation}
where the first term on the right-hand side can be rewritten as $\frac{1}{p+n}\pa{GH\Ga G}_{\uu\vv}$.
The derivative of $Q_{\uu\vv}$ is given by
\begin{equation}\label{partialq}
\begin{aligned}
\partial_{i\mu} Q_{\uu\vv} = -z^{1/2}\pb{(G\AA)_{\uu i}(\BB Q)_{\mu\vv}+(G\BB)_{\uu\mu}(\AA Q)_{i\vv}+(Q\AA)_{\uu i}(\BB G)_{\mu\vv}+(Q\BB)_{\uu\mu}(\AA G)_{i\vv}}+R_{\uu\vv}^{i\mu},
\end{aligned}
\end{equation}
where the term $R_{\uu\vv}^{i\mu}:=\pb{G(\partial_{i\mu}F)G}_{\uu\vv}$ is computed from the expression (\ref{f}) as follows:
\begin{align}
\partial_{i\mu} F_{\uu\vv}=&~\frac{z^{1/2}}{p+n}\BB_{\uu\mu}(\AA\Ga)_{i\vv}+\frac{1}{p+n}\partial_{i\mu}\pB{zm_2(z)(\AA^2\Ga)+\frac{z}{n}\pb{\Tr(\Ga G\AA^2)\BB^2+\BB^2G\AA^2\Ga+\BB^2G\Ga^T\AA^2}}_{\uu\vv}\nonumber\\
=&~\frac{z^{1/2}}{p+n}\BB_{\uu\mu}(\AA\Ga)_{i\vv}-\frac{z^{3/2}}{n(p+n)}\Bigl((\BB G\BB^2 G\AA)_{\mu i}(\AA^2\Ga)_{\uu\vv}+(\AA G\BB^2 G\BB)_{i\mu}(\AA^2\Ga)_{\uu\vv}\nonumber\\
&~+(\BB G\AA^2\Ga G\AA)_{\mu i}(\BB^2)_{\uu\vv}+(\AA G\AA^2\Ga G\BB)_{i\mu}(\BB^2)_{\uu\vv} +(\BB^2 G\BB)_{\uu\mu}(\AA G\AA^2\Ga)_{i\vv}\nonumber\\
&~+(\BB^2 G\AA)_{\uu i}(\BB G\AA^2\Ga)_{\mu\vv}+(\BB^2 G\BB)_{\uu\mu}(\AA G\Ga^T\AA^2)_{i\vv}+(\BB^2 G\AA)_{\uu i}(\BB G\Ga^T\AA^2)_{\mu\vv}\Bigr).\label{partialf}
\end{align}
Hence, $R_{\uu\vv}^{i\mu}$ can be expressed as
\begin{align}
&R_{\uu\vv}^{i\mu}= \frac{z^{1/2}}{p+n}(G\BB)_{\uu\mu}(\AA\Ga G)_{i\vv}-\frac{z^{3/2}}{n(p+n)}\Bigl((\BB G\BB^2 G\AA)_{\mu i}(G\AA^2\Ga G)_{\uu\vv}+(\AA G\BB^2 G\BB)_{i\mu}(G\AA^2\Ga G)_{\uu\vv}\nonumber\\
&\quad+(\BB G\AA^2\Ga G\AA)_{\mu i}(G\BB^2G)_{\uu\vv}+(\AA G\AA^2\Ga G\BB)_{i\mu}(G\BB^2G)_{\uu\vv} +(G\BB^2 G\BB)_{\uu\mu}(\AA G\AA^2\Ga G)_{i\vv}\nonumber\\
&\quad+(G\BB^2 G\AA)_{\uu i}(\BB G\AA^2\Ga G)_{\mu\vv}+(G\BB^2 G\BB)_{\uu\mu}(\AA G\Ga^T\AA^2G)_{i\vv}+(G\BB^2 G\AA)_{\uu i}(\BB G\Ga^T\AA^2G)_{\mu\vv}\Bigr).\label{r}
\end{align}

\begin{proof}[Proof of Lemma \ref{qlemma} (i)]
We prove the claim by induction on $l$, starting with the case $l=0$. The last four terms in \eqref{q} can be bounded directly by $\OO_\prec\pb{(1+\phi)\Theta^2}$ using Lemma \ref{multipleG}. 
For the first term in \eqref{q}, we have 
$$(p+n)^{-1}|(GH\Ga G)_{\uu\vv}|\prec n^{-1}\norm{H}\norm{G^*\uu}\norm{\Ga}\norm{G\vv}\prec  \Theta^2,$$
where in the second step we used \eqref{ward2} to control $\norm{G^*\uu}$ and $\norm{G\vv}$, and Lemma \ref{priori} to bound $\norm{H}\prec 1$. Therefore, we conclude that $Q\prec(1+\phi)\Theta^{2}$. 

The inductive step proceeds as follows. Suppose that $l\geq 1$ and $\partial_{i\mu}^{m}Q\prec(1+\phi)^{m+1}\Theta^{2}$ for $0\leq m\leq l-1$. By (\ref{partialq}), we have
$$\partial_{i\mu}^{l} Q_{\uu\vv} = -z^{1/2}\partial_{i\mu}^{l-1}\pb{(G\AA)_{\uu i}(\BB Q)_{\mu\vv}+(G\BB)_{\uu\mu}(\AA Q)_{i\vv}+(Q\AA)_{\uu i}(\BB G)_{\mu\vv}+(Q\BB)_{\uu\mu}(\AA G)_{i\vv}}+\partial_{i\mu}^{l-1}R_{\uu\vv}^{i\mu}.$$
The first term $\partial_{i\mu}^{l-1}((G\AA)_{\uu i}(\BB Q)_{\mu\vv})$ is estimated using the Leibniz rule and the inductive hypothesis:
\begin{equation}
\begin{aligned}
|\partial_{i\mu}^{l-1}((G\AA)_{\uu i}(\BB Q)_{\mu\vv})|&\leq\sum_{m=0}^{l-1}\binom{l-1}{m}|\partial_{i\mu}^{m}(G\AA)_{\uu i}||\partial_{i\mu}^{l-1-m}(\BB Q)_{\mu\vv}|\\
&\prec\sum_{m=0}^{l-1} (1+\phi)^{m+1}\cdot (1+\phi)^{l-m}\Theta^{2}\prec(1+\phi)^{l+1}\Theta^{2}.
\nonumber
\end{aligned}
\end{equation}
The second, third, and fourth terms are estimated analogously. It remains to show that
\begin{equation}
\partial_{i\mu}^{l-1}R_{\uu\vv}^{i\mu}\prec(1+\phi)^{l+1}\Theta^{2}.\label{eq:der_R_l}   
\end{equation}
According to (\ref{r}), we need to estimate the derivatives of nine terms. For the first term in \eqref{r}, \eqref{Gderivative} and Lemma \ref{multipleG} give
$$\frac{z^{1/2}}{p+n}\partial_{i\mu}^{l-1}\pb{(G\BB)_{\uu\mu}(\AA\Ga G)_{i\vv}}\prec n^{-1}(1+\phi)^{l+1}\prec(1+\phi)^{l+1}\Theta^{2}.$$
The second term 
$$-\frac{z^{3/2}}{n(p+n)}\partial_{i\mu}^{l-1}\pb{(\BB G\BB^2 G\AA)_{\mu i}(G\AA^2\Ga G)_{\uu\vv}}$$
can be written as a scalar multiple of a sum of $\frac{2^{l-2}}{3}(l+2)!$ terms, each of the form $$\frac{1}{n(p+n)}(\Ga_1G\Ga_2G\Ga_3)_{\xx_1\yy_1}(\Ga_4G\Ga_5G\Ga_6)_{\xx_2\yy_2}(\Ga_7G\Ga_8)_{\xx_3\yy_3}\cdots(\Ga_{2l+3}G\Ga_{2l+4})_{\xx_{l+1}\yy_{l+1}},$$
where $\Ga_1,\cdots,\Ga_{2l+4}\in\{\AA^2,\BB^2,\AA,\BB,I_{p+n}\}$ and $\xx_1,\cdots,\yy_{l+1}\in\{\uu,\vv,i,\mu\}.$
Thus, applying Lemma \ref{multipleG} again, we obtain
$$\left|\frac{z^{3/2}}{n(p+n)}\partial_{i\mu}^{l-1}\pb{\BB G\BB^2 G\AA)_{\mu i}(G\AA^2\Ga G)_{\uu\vv}}\right|\prec(1+\phi)^{l-1}\Theta^4\le (1+\phi)^{l+1}\Theta^2$$
The derivatives of the remaining seven terms in \eqref{r} are estimated analogously. This proves \eqref{eq:der_R_l} and completes the inductive step.
\end{proof}

\begin{proof}[Proof of Lemma \ref{qlemma} (ii)]
For fixed $e\in \N_+$ with $e\ge 4$ and deterministic unit vectors $\uu,\vv\in \mathbf S$, set
$$\mathcal{L}:=\expect[|Q_{\uu\vv}|^{2e}]^{\frac{1}{2e}}.$$
Using (\ref{q}) and applying the cumulant expansion from Lemma \ref{lem:cumulant_expansion}, we may write
\begin{align}
\mathcal{L}^{2e}&= \expect  \pa{\cal L_0Q_{\uu\vv}^{e-1}\overline{Q}_{\uu\vv}^{e}} +\sum_{k=1}^\ell Z_k+\sum_{i,\mu}\cal R''^{i\mu}_{\ell+1}.\label{aims3}
\end{align}
where $\cal L_0$ is defined by
\[\cal L_0:= {\frac{z}{p+n}\pB{m_2(z)(G\AA^2\Ga G)+\frac{1}{n}\pb{\Tr(\Ga G\AA^2)G\BB^2 G+G\BB^2G\AA^2\Ga G+G\BB^2G\Ga^T\AA^2G}}_{\uu\vv}}, \]
$Z_k$ is defined for each $k\in \N_+$ by
\begin{equation}
\nonumber
\begin{aligned}
 Z_{k}:=\frac{z^{1/2}}{p+n}\sum_{i,\mu}\frac{\kappa_{k+1}(x_{i\mu})}{k!}\expect\qB{\partial_{i\mu}^{k}\left( (G\BB)_{\uu\mu}(\AA\Ga G)_{i\vv}Q_{\uu\vv}^{e-1}\overline{Q}_{\uu\vv}^{e} \right)},
 \end{aligned}
 \end{equation}
$\ell$ is a positive integer to be chosen later in the proof, and the remainder terms $\cal R''^{i\mu}_{\ell+1}$ are defined analogously to \smash{$\cal R^{i\mu}_{\ell+1}$} in (\ref{remainder^imu}). The proof of $Q_{\uu\vv}\prec(1+\phi)^4(q+\Theta)^3$ is reduced to the following lemma.

\begin{lemma}\label{task3}
Let $\lambda\geq (1+\phi)^4(q+\Theta)^3=:\Xi$ be any deterministic parameter satisfying $Q\prec\lambda$ and $\lambda\le n^C$ for a constant $C>0$. We have the following estimates.
\begin{enumerate}
    \item 
    $\expect  \pb{\cal L_0Q_{\uu\vv}^{e-1}\overline{Q}_{\uu\vv}^{e}}+Z_1\prec\Xi\lambda\cal L^{2e-2}.$
    \item For any fixed $k \geq 2$, $Z_k\prec\sum_{s=1}^{2e}(\Xi\lambda)^{s/2}\cal 
    L^{2e-s}$.
    \item For any fixed $D>0$, there exists a constant $\ell \equiv \ell(D) \geq 1$  such that $\cal R''^{i\mu}_{\ell+1}=\OO(n^{-D})$ uniformly in $i,\mu$.
\end{enumerate}
\end{lemma}
Indeed, Lemma \ref{task3} implies $\cal L\prec(\Xi\lambda)^{1/2}$, by the same argument used around \eqref{eq_pre_young} to derive $\cal M\prec\Phi$ from Lemma \ref{task1}. This gives the self-improving bound
$$Q \prec \lambda \Longrightarrow  Q\prec (\Xi\lambda)^{1/2}$$
for any $\lambda \ge\Xi$. Iterating this implication $k$ times, starting from the initial bound $Q\prec(1+\phi)(q+\Theta)^2$ given by Lemma \ref{qlemma} (i), yields
$$Q \prec \pb{(1+\phi)(q+\Theta)^2}^{\frac{1}{2^k}} \Xi^{1 - \frac{1}{2^k}}.$$
Since $k$ is arbitrary, we obtain $Q\prec \Xi = (1+\phi)^4(q+\Theta)^3$.
\end{proof}

The rest of this subsection is devoted to the proof of Lemma \ref{task3}.

\begin{proof}[Proof of Lemma \ref{task3} (i)]
As in the proof of Lemma \ref{task1} (i), a crucial cancellation occurs here. A direct computation shows that
\begin{equation}\nonumber
\begin{aligned}
&\frac{z^{1/2}}{n(p+n)}\sum_{i,\mu}\partial_{i\mu}\pb{(G\BB)_{\uu\mu}(\AA\Ga G)_{i\vv}}+ \cal L_0=0,
\end{aligned}
\end{equation}
which yields
\begin{equation}\label{Z_1aim}
\begin{aligned}
&\expect\pB{\cal L_0 Q_{\uu\vv}^{e-1}\overline{Q}_{\uu\vv}^{e}}+Z_1 
=\frac{z^{1/2}}{n(p+n)}\sum_{i,\mu}\expect\qB{(G\BB)_{\uu\mu}(\AA\Ga G)_{i\vv}\partial_{i\mu}\left( Q_{\uu\vv}^{e-1}\overline{Q}_{\uu\vv}^{e} \right)}.
\end{aligned}
\end{equation}
By Hölder's inequality, it suffices to prove that
\begin{equation}\label{partialqbound1}
\begin{aligned}
n^{-2}\sum_{i,\mu}\left|(G\BB)_{\uu\mu}(\AA\Ga G)_{i\vv}(\partial_{i\mu}Q_{\uu\vv})\right|\prec\Xi\lambda.
\end{aligned}
\end{equation}
According to (\ref{partialq}), there are five terms to estimate. The contribution from the first term in \eqref{partialq} can be estimated using the assumption $Q\prec \lambda$, the Cauchy-Schwarz inequality, and \eqref{ward2}:
$$n^{-2}\sum_{i,\mu}\left|(G\BB)_{\uu\mu}(\AA\Ga G)_{i\vv}(G\AA)_{\uu i}(\BB Q)_{\mu\vv}\right|
\prec n^{-2}\lambda\pB{\sum_i\left|(\AA\Ga G)_{i\vv}(G\AA)_{\uu i}\right|}\pB{\sum_\mu\left|(G\BB)_{\uu\mu}\right|}
\prec\Theta^{3}\lambda\prec\Xi\lambda.$$
The contributions from the second, third, and fourth terms in \eqref{partialq} can be estimated analogously. It remains to control $n^{-2}\sum_{i,\mu}\left|(G\BB)_{\uu\mu}(\AA\Ga G)_{i\vv}R_{\uu\vv}^{i\mu}\right|$.
By (\ref{r}) and Lemma \ref{multipleG}, we have
\begin{equation}\label{Rrough}
\begin{aligned}
\left|R_{\uu\vv}^{i\mu}\right|\prec n^{-1}\left|(G\BB)_{\uu\mu}(\AA\Ga G)_{i\vv}\right|+\Theta^4.
\end{aligned}
\end{equation}
Combining this bound with the Cauchy-Schwarz inequality and \eqref{ward2}, we obtain
\begin{align*}
 n^{-2}\sum_{i,\mu}\left|(G\BB)_{\uu\mu}(\AA\Ga G)_{i\vv}R_{\uu\vv}^{i\mu}\right|
&\prec n^{-3}\sum_{i,\mu}\left|(G\BB)_{\uu\mu}^2(\AA\Ga G)_{i\vv}^2\right|+n^{-2}\Theta^4\sum_{i,\mu}\left|(G\BB)_{\uu\mu}(\AA\Ga G)_{i\vv}\right|\\
&\prec\Theta^6\prec\Xi\lambda.
\end{align*}
This completes the proof of \eqref{partialqbound1}.
\end{proof}

\begin{proof}[Proof of Lemma \ref{task3} (ii)]
As in the estimates of (\ref{xk}) and (\ref{yk}), we drop the complex conjugates, expand the partial derivatives, and use
$\kappa_{k+1}(x_{i\mu}) = \OO(n^{-3/2}q^{k-2})$. Thus, we only need to estimate 
$$ n^{-5/2}q^{k-2} \sum_{i,\mu} \expect\left|\partial_{i\mu}^{r}\pb{(G\BB)_{\uu\mu}(\AA\Ga G)_{i\vv}}\partial_{i\mu}^{k-r}Q_{\uu\vv}^{2e-1}\right|$$ 
for $r=0,1,\cdots,k$. Each such quantity can be further expanded into a sum of terms
\begin{equation}
\label{aimlast}
\begin{aligned}
 n^{-5/2}q^{k-2} \sum_{i,\mu} \expect\left|\partial_{i\mu}^{r}\pb{(G\BB)_{\uu\mu}(\AA\Ga G)_{i\vv}}\pB{\prod_{m=1}^{w} \partial_{i\mu}^{l_m}Q_{\uu\vv}}Q_{\uu\vv}^{2e-1-w}\right|
\end{aligned}
\end{equation}
where the sum ranges over $w = 0,1,\cdots,(k-r) \wedge (2e-1)$ and integers $ l_{1},\cdots, l_{w} > 0$ with $l_{1} + \cdots + l_{w} = k-r$. 
We aim to bound (\ref{aimlast}) by $(\Xi\lambda)^{\frac{w+1}{2}}\mathcal{L}^{2e-w-1}.$  By Hölder's inequality, it is enough to show
\begin{equation}
\label{zkaim}
\begin{aligned}
n^{-5/2}q^{k-2} \sum_{i,\mu} \expect\left|\partial_{i\mu}^{r}\pb{(G\BB)_{\uu\mu}(\AA\Ga G)_{i\vv}}\pB{\prod_{m=1}^{w} \partial_{i\mu}^{l_m}Q_{\uu\vv}}\right|\prec(\Xi\lambda)^{\frac{w+1}{2}}.\\
\end{aligned}
\end{equation}

By Lemma \ref{qlemma} (i), established above, we have
$$\left|\partial_{i\mu}^{l_m}Q_{\uu\vv}\right| \prec (1+\phi)^{l_{m}+1}\Theta^2.$$
Moreover, by (\ref{Gderivative}), $\partial_{i\mu}^{r}\pb{(G\BB)_{\uu\mu}(\AA\Ga G)_{i\vv}}$ can be written as a scalar multiple of a sum of $2^r(r+1)!$ terms, each of which is a product of $r+2$ entries of the form $(\Ga_1G\Ga_2)_{\xx\yy}$, where $\Ga_1,\Ga_2\in\{\AA\Ga,\AA,\BB,I_{p+n}\}$ and $\xx,\yy\in\{\uu,\vv,i,\mu\}$.
Among these entries, one is either $(G\BB)_{\uu\mu}$ or $(G\AA)_{\uu i}$, and another is one of $(\AA\Ga G)_{i\vv}$, $(\AA G)_{i\vv}$, or $(\BB G)_{\mu\vv}$. Applying the Cauchy-Schwarz inequality and \eqref{ward2} to these two entries, and Lemma \ref{multipleG} to the remaining ones, we obtain
$$\sum_{i,\mu}\left|\partial_{i\mu}^{r}\pb{(G\BB)_{\uu\mu}(\AA\Ga G)_{i\vv}}\right|\prec n^2(1+\phi)^r\Theta^2.$$
Substituting the two estimates above into (\ref{zkaim}), we find that, when $w \le k-2$,
\begin{equation}
\nonumber
\begin{aligned}
&~n^{-5/2}q^{k-2} \sum_{i,\mu} \expect\left|\partial_{i\mu}^{r}\pb{(G\BB)_{\uu\mu}(\AA\Ga G)_{i\vv}}\pB{\prod_{m=1}^{w} \partial_{i\mu}^{l_m}Q_{\uu\vv}}\right|\\
\prec&~ n^{-5/2}q^{k-2}(1+\phi)^{k-r+w}\Theta^{2w}\cdot n^2(1+\phi)^r\Theta^2 \prec (1+\phi)^{1+w}\pb{(1+\phi)q}^{k-1}\Theta^{2w+2}\\
\prec&~ \pb{(1+\phi)(q+\Theta)}^{k+2w+1}
\leq \Xi^{w+1} \le (\Xi\lambda)^{\frac{w+1}{2}}.
\end{aligned}
\end{equation}
It remains to estimate the left-hand side of (\ref{zkaim}) for $w \ge k-1$, which we assume from now on. Since $w\le k-r$, we must have $r = 0 $ or $1$.
Thus, we are left with the three cases $(r,w) = (0,k)$, $(r,w) = (1,k-1)$, and $(r,w) = (0, k-1)$, which we treat separately.

\medskip
\noindent{\bf Case 1}: $(r,w) = (0,k)$. 
In this case, $l_{1} =l_{2} = ... = l_{k} = 1$. We estimate the left-hand side of (\ref{zkaim}) as follows:
\begin{equation}
\nonumber
\begin{aligned}
&~  n^{-5/2}q^{k-2} \sum_{i,\mu}\left|(G\BB)_{\uu\mu}(\AA\Ga G)_{i\vv}(\partial_{i\mu}Q_{\uu\vv})^k\right|\\
\prec&~ n^{-5/2}q^{k-2}(1+\phi)^{2k-2}\Theta^{2k-2} \sum_{i,\mu}\left|(G\BB)_{\uu\mu}(\AA\Ga G)_{i\vv}(\partial_{i\mu}Q_{\uu\vv})\right|\\
\prec&~ n^{-5/2}q^{k-2}(1+\phi)^{2k-2}\Theta^{2k-2}\cdot n^2\Xi\lambda \le (1+\phi)^{2k-2}(q+\Theta)^{3k-3}\Xi\lambda \le \Xi^{w}\lambda \le (\Xi\lambda)^{\frac{w+1}{2}},
\end{aligned}
\end{equation}
Here, in the first step we applied the rough bound $\partial_{i\mu}Q_{\uu\vv}\prec(1+\phi)^2\Theta^2$ from Lemma \ref{qlemma} (i) to $k-1$ copies of $\partial_{i\mu}Q_{\uu\vv}$, while in the second step we used (\ref{partialqbound1}).

\medskip
\noindent{\bf Case 2}: $(r,w) = (1,k-1)$. 
In this case, $l_{1} =l_{2} = \cdots= l_{k-1} = 1$. As in Case 1, we estimate the left-hand side of (\ref{zkaim}) as follows:
\begin{equation}
\nonumber
\begin{aligned}
& ~ n^{-5/2}q^{k-2}  \sum_{i,\mu}\left|\partial_{i\mu}\pb{(G\BB)_{\uu\mu}(\AA\Ga G)_{i\vv}}(\partial_{i\mu}Q_{\uu\vv})^{k-1}\right|\\
\prec&~ n^{-5/2}q^{k-2}(1+\phi)^{2k-4}\Theta^{2k-4}\sum_{i,\mu}\left|\partial_{i\mu}\pb{(G\BB)_{\uu\mu}(\AA\Ga G)_{i\vv}}(\partial_{i\mu}Q_{\uu\vv})\right|\\
\prec&~ n^{-5/2}q^{k-2}(1+\phi)^{2k-4}\Theta^{2k-4}\cdot n^{5/2}(1+\phi)^2\Theta^3\lambda \le  (1+\phi)^{2k-2}(q+\Theta)^{3k-3}\lambda \le \Xi^{w}\lambda \le (\Xi\lambda)^{\frac{w+1}{2}},
\end{aligned}
\end{equation} 
where in the second step we used the bound
\begin{equation}\label{eq:derivative_GBGAG_bound}
\sum_{i,\mu}\left|\partial_{i\mu}\pb{(G\BB)_{\uu\mu}(\AA\Ga G)_{i\vv}}(\partial_{i\mu}Q_{\uu\vv})\right| \prec n^{5/2}(1+\phi)^2\Theta^3\lambda.
\end{equation}
It remains to prove this bound. By (\ref{partialq}), (\ref{Rrough}), and Lemma \ref{multipleG}, we have
$$\left|\partial_{i\mu}Q_{\uu\vv}\right|\prec (1+\phi)\lambda+n^{-1}\left|(G\BB)_{\uu\mu}(\AA\Ga G)_{i\vv}\right|.$$
Therefore, \eqref{eq:derivative_GBGAG_bound} follows from the two estimates below:
\begin{align*}
\nonumber
& ~ (1+\phi)\lambda\sum_{i,\mu}\left|\partial_{i\mu}\pb{(G\BB)_{\uu\mu}(\AA\Ga G)_{i\vv}}\right|\\
\prec& ~  (1+\phi)^2\lambda\sum_{i,\mu}\left(|(G\BB)_{\uu\mu}(\AA\Ga G)_{i\vv}|+|(G\AA)_{\uu i}(\AA\Ga G)_{i\vv}|+|(G\BB)_{\uu\mu}(\AA G)_{i\vv}|+|(G\BB)_{\uu\mu}(\BB G)_{\mu\vv}|\right)\\
\prec& ~  (1+\phi)^2\lambda\cdot n^2\Theta^2\le n^{5/2}(1+\phi)^2\Theta^3\lambda,\end{align*}
and
\begin{align*}
& ~  n^{-1}\sum_{i,\mu}\left|\partial_{i\mu}\pb{(G\BB)_{\uu\mu}(\AA\Ga G)_{i\vv}}(G\BB)_{\uu\mu}(\AA\Ga G)_{i\vv}\right|\\
\prec& ~  n^{-1}(1+\phi)^2\sum_{i,\mu}\left(|(G\BB)_{\uu\mu}^2(\AA\Ga G)_{i\vv}|+|(G\BB)_{\uu\mu}(\AA\Ga G)_{i\vv}^2| \right)\prec  n(1+\phi)^2\Theta^3\prec n^{5/2}(1+\phi)^2\Theta^3\lambda,
\end{align*}
where we again used Lemma \ref{multipleG} and \eqref{ward2} to control the sums over $i$ and $\mu$.

\medskip
\noindent{\bf Case 3}: $(r,w) = (0, k-1)$. 
In this case, we may assume without loss of generality that $l_{1} = l_{2} = \cdots = l_{k-2} = 1$ and $l_{k-1}=2$. We estimate the left-hand side of (\ref{zkaim}) as follows:
\begin{equation}
\nonumber
\begin{aligned}
&~  n^{-5/2}q^{k-2}  \sum_{i,\mu}\left|(G\BB)_{\uu\mu}(\AA\Ga G)_{i\vv}(\partial_{i\mu}^2Q_{\uu\vv})(\partial_{i\mu}Q_{\uu\vv})^{k-2}\right|\\
\prec&~ n^{-5/2}q^{k-2}(1+\phi)^{2k-4}\Theta^{2k-4}\sum_{i,\mu}\left|(G\BB)_{\uu\mu}(\AA\Ga G)_{i\vv}\partial_{i\mu}^2Q_{\uu\vv}\right|\\
\prec&~ n^{-5/2}q^{k-2}(1+\phi)^{2k-4}\Theta^{2k-4}\cdot n^{5/2}(1+\phi)^2\Theta^3\lambda 
\le (1+\phi)^{2k-2}(q+\Theta)^{3k-3}\lambda \le \Xi^{w}\lambda \le (\Xi\lambda)^{\frac{w+1}{2}},
\end{aligned}
\end{equation}
Here, in the first step we applied the rough bound $\partial_{i\mu}Q_{\uu\vv}\prec(1+\phi)^2\Theta^2$ from Lemma \ref{qlemma} (i) to the $k-2$ copies of $\partial_{i\mu}Q_{\uu\vv}$, while in the second step we used the bound
\begin{equation}\label{eq:derivative_GBGAG_bound2}
\sum_{i,\mu}\left|(G\BB)_{\uu\mu}(\AA\Ga G)_{i\vv}\partial_{i\mu}^2Q_{\uu\vv}\right| \prec n^{5/2}(1+\phi)^2\Theta^3\lambda.
\end{equation}
It remains to prove this bound. By (\ref{partialq}), it suffices to estimate the derivatives of the five terms appearing there. For the derivative of the first term, the assumption $Q\prec \lambda$ and Lemma \ref{qlemma} (i) with $l=1$ give
$$\left|\partial_{i\mu}\pb{(G\AA)_{\uu i}(\BB Q)_{\mu\vv}}\right|\prec(1+\phi)\lambda\pb{|(G\BB)_{\uu\mu}|+|(G\AA)_{\uu i}|}+(1+\phi)^2\Theta^2|(G\AA)_{\uu i}|.$$
Combining this bound with the Cauchy-Schwarz inequality and \eqref{ward2}, we obtain
$$\sum_{i,\mu}\left|(G\BB)_{\uu\mu}(\AA\Ga G)_{i\vv}\partial_{i\mu}\pb{(G\AA)_{\uu i}(\BB Q)_{\mu\vv}}\right| \prec n^2\Theta^3\pb{(1+\phi)\lambda+(1+\phi)^2\Theta^2} \lesssim n^{5/2}(1+\phi)^2\Theta^3\lambda.$$
The derivatives of the second to fourth terms in \eqref{partialq} can be estimated analogously. It remains to show that 
$$\sum_{i,\mu}\left|(G\BB)_{\uu\mu}(\AA\Ga G)_{i\vv}\partial_{i\mu}R_{\uu\vv}^{i\mu}\right| \prec n^{5/2}(1+\phi)^2\Theta^3\lambda.$$
Taking the derivatives of (\ref{r}) and bounding the resulting terms with Lemma \ref{multipleG}, we obtain
\begin{equation}\nonumber
\begin{aligned}
\left|\partial_{i\mu}R_{\uu\vv}^{i\mu}\right|\prec n^{-1}(1+\phi)^2\pb{|(G\BB)_{\uu\mu}|+|(\AA\Ga G)_{i\vv}|}+(1+\phi)\Theta^4.
\end{aligned}
\end{equation}
Combining this bound with the Cauchy-Schwarz inequality and \eqref{ward2}, we obtain
$$\sum_{i,\mu}\left|(G\BB)_{\uu\mu}(\AA\Ga G)_{i\vv}\partial_{i\mu}R_{\uu\vv}^{i\mu}\right| \prec n(1+\phi)^2\Theta^3+n^2(1+\phi)\Theta^6 \prec n^{5/2}(1+\phi)^2\Theta^3\lambda,$$
which completes the proof of \eqref{eq:derivative_GBGAG_bound2}.
\end{proof}

\begin{proof}[Proof of Lemma \ref{task3} (iii)]
By rewriting the first term in (\ref{q}) as \smash{$\p{p+n}^{-1}\pa{GH\Ga G}_{\uu\vv}$}, we obtain an alternative form of $Q_{\uu\vv}$ that involves only entries of $G$, and no entries of $H$. This allows us to estimate $\cal R''^{i\mu}_{\ell+1}$ in a way completely analogous to the proof of Lemma \ref{task1} (iii). We therefore omit the details.
\end{proof}

\section{Proof of Theorem \ref{local_law}}\label{sec:det_step}

In this section, we prove Theorem \ref{local_law} using the estimates established in Theorem \ref{maintheorem}.
\subsection{Preliminaries}
Recall from (\ref{s-ceq}) that $M(G) = I_{p+n}+zG+G\mathcal{S}(G)$, while (\ref{S(Lambda)}) gives (recall \eqref{defn_m1m2} and \eqref{defn_AB})
$$\mathcal{S}(G)=zm_2(z)\AA^2+zm_1(z)\BB^2+\frac{z}{n}(\BB^2 G\AA^2+\AA^2 G\BB^2).$$
We introduce the slightly modified quantity
\begin{equation}
\begin{aligned}
\label{tildemg}
\wt M(G) = I_{p+n}+zG+G
    \begin{pmatrix}
   {zm_2(z)A} & {0}   \\
   {0} & {zm_1(z)B}  \\
   \end{pmatrix}.
\end{aligned}
\end{equation}
The following lemma shows that this quantity is close to $M(G)$.
\begin{lemma}
\label{tildeMG}
Let $z \in S(c_0,C_0,\epsilon)$ and suppose that $G-\Pi \prec \phi$ at $z$, where $\phi$ is a deterministic continuous function $\phi: S(c_0,C_0,\epsilon) \to \left[ n^{-1}, n^{\delta}\right]$ with $\delta:=(\tau\wedge\epsilon)/10$. Define
$$h(\phi):=\sqrt{\frac{\Im m_{2c}(z)+\phi}{n\eta}}.$$
Then
\begin{enumerate}
    \item[(1)] $\wt M(G) \prec (1+\phi)^{3}\pb{q+h(\phi)}.$
    \item[(2)] For any deterministic matrix $\Ga$ satisfying $\norm{\Ga} = \OO(1),$ $\abs{\underline{\Ga\wt M(G)}}\prec(1+\phi)^6\pb{q+h(\phi)}^2.$
\end{enumerate}
\end{lemma}
\begin{proof}
We first observe that Theorem \ref{maintheorem} implies that $M(G)$ satisfies the estimates in (1) and (2), since $\Theta\prec h(\phi)$. Indeed, \eqref{Immc} gives $\eta\lesssim \Im m_{2c}(z)$. 
It remains to prove that $\norm{\Im \Pi}\lesssim\Im m_{2c}(z)$. By (\ref{defn_pi}), we have
\begin{equation}\label{ImPi}
\Im \Pi=\Pi\begin{pmatrix}
   \Im \qb{z\pb{1+m_{2c}(z)A}} & 0   \\
   0 & \Im \qb{z\pb{1+m_{1c}(z)B}}  \\
   \end{pmatrix}\Pi^{*}.
\end{equation}
The estimate \eqref{Piii} implies that $\norm{\Pi}=\norm{\Pi^*}=\OO(1)$. Moreover, for $z=E+\ii \eta$, we have
$$|\Im[zm_{2c}(z)]|=|\eta\Re m_{2c}(z)+E\Im m_{2c}(z)|\lesssim {\eta+\Im m_{2c}(z)}\lesssim \Im m_{2c}(z)$$
by \eqref{Immc}. Hence
$$\norm{\Im[z\p{1+m_{2c}(z)A}]}\lesssim \eta+\abs{\Im[zm_{2c}(z)]}\norm{A}\lesssim\Im m_{2c}(z).$$
Similarly, we have 
$$\norm{\Im[z\p{1+m_{1c}(z)B}]}\lesssim \Im m_{1c}(z)\sim\Im m_{2c}(z).$$ 
Substituting these estimates into \eqref{ImPi}, we obtain $\norm{\Im \Pi}\prec\Im m_{2c}(z)$, which shows that $\Theta\prec h(\phi)$.

It remains to show that the difference between $M(G)$ and $\wt M(G)$ is negligible, which will conclude the proof of Lemma \ref{tildeMG}. By definition, we have
\begin{equation}\nonumber
\begin{aligned}
M(G)-\wt M(G)&=\frac{z}{n}G\BB^2 G\AA^2+\frac{z}{n}G\AA^2 G\BB^2,\ \
\underline{\Ga M(G)}-\underline{\Ga\wt M(G)}=\frac{z}{n}\underline{\Ga G\BB^2 G\AA^2}+\frac{z}{n}\underline{\Ga G\AA^2 G\BB^2}.
\end{aligned}
\end{equation}
By Lemma \ref{multipleG}, these quantities are stochastically dominated by $\Theta^2\prec h(\phi)^2$.
\end{proof}

To simplify notation, we introduce
$$R_{w_1,w_2}:=\begin{pmatrix}
   z(1 + w_2A) & 0   \\
   0 & z(1 + w_1B)  \\
   \end{pmatrix}^{-1},\quad w_1,w_2\in\mathbb{C}.$$
Then (\ref{defn_pi}) can be rewritten as $\Pi=-R_{m_{1c},m_{2c}}$, and equation (\ref{tildemg}) can be rearranged as
\begin{equation}
\label{GandMG}
\begin{aligned}
G = (\wt M(G)-I_{p+n})\begin{pmatrix}
   z(1 + m_2(z)A) & 0   \\
   0 & z(1 + m_1(z)B)  \\
   \end{pmatrix}^{-1}=(\wt M(G)-I_{p+n})R_{m_1,m_2}.
\end{aligned}
\end{equation}
Taking the difference between these equations yields
\begin{equation}
\label{GminusPi}
\begin{aligned}
G-\Pi = \wt M(G)R_{m_1,m_2}-R_{m_1,m_2}+R_{m_{1c},m_{2c}}
\end{aligned}
\end{equation}
The first term will be treated as an error term and estimated by the following lemma.
\begin{lemma}
\label{tildeMGRgS}
Under the assumptions of Lemma \ref{tildeMG}, assume that $\norm{R_{m_1,m_2}}=\OO(1)$ with high probability. Then the conclusions of Lemma \ref{tildeMG} hold with $\wt M(G)$ replaced by $\wt M(G)R_{m_1,m_2}$.
\end{lemma}
If $R_{m_1,m_2}$ were a deterministic matrix with $\OO(1)$ norm, Lemma \ref{tildeMGRgS} would follow immediately from Lemma \ref{tildeMG}. To handle the randomness of $m_1$ and $m_2$, we use a standard $n^{-3}$-net argument, as explained below. 
\begin{proof}[Proof of Lemma \ref{tildeMGRgS}]
Let $\Omega$ be an event of high probability such that $\mathbbm{1}_{\Omega} \norm{R_{m_1,m_2}} = \OO(1)$, and define
 $\cal{U} \deq \h{(m_1 \equiv m_1(z,X), m_2 \equiv m_2(z,X)) \col X \in \Omega} \subset \C^{2}$.
 Since $\abs{m_{1,2}} =\OO(\norm{G})=\OO(n)$, there exists an $n^{-3}$-net of $\cal U$ of size $\OO(n^{10})$, i.e., a subset $\hat {\cal U} \subset \cal U$ such that (1) $\abs{\hat {\cal U}} = \OO(n^{10})$, and (2) for each $ u  \in \cal U$ there exists a $ \hat u  \in \hat {\cal U}$ such that $\abs{ u - \hat u } \leq n^{-3}$. Combining Lemma \ref{tildeMG} with a union bound, we obtain, for any deterministic $\f v , \f w \in \bf S$,
\begin{equation} \label{sup_s_est}
\sup_{ \hat u = (\hat m_1,\hat m_2)  \in \hat {\cal U}} \abs{( \wt M(G)R_{ \hat m_1, \hat m_2 })_{\f v \f w}} \prec (1 + \phi)^3 \pb{q+h(\phi)}\,.
\end{equation}
Moreover, the resolvent identity
\begin{equation}
\label{residofR}
R_{\hat m_1,\hat m_2}=R_{m_1,m_2}+R_{\hat m_1,\hat m_2}
\begin{pmatrix}
   z(m_2 - \hat m_2)A & 0   \\
   0 & z(m_1 - \hat m_1)B  \\
   \end{pmatrix}
R_{m_1,m_2},
\end{equation}
together with $\norm{R_{m_1,m_2}}=\OO(1)$ and $|m_1-\hat m_1|+|m_2-\hat m_2|=\OO(n^{-3})$ implies that $\norm{R_{\hat m_1,\hat m_2}}=\OO(1)$. We then write
\begin{equation*}
\abs{(\wt M(G)R_{m_1,m_2})_{\f v \f w}} \leq \abs{\wt M(G)(R_{\hat m_1,\hat m_2})_{\f v \f w}} + \absb{\pb{\wt M(G)(R_{\hat m_1, \hat m_2} - R_{m_1,m_2})}_{\f v \f w}}.
\end{equation*}
The first term is bounded by $\OO_\prec\p{(1 + \phi)^3\pb{q+h(\phi)}}$ by \eqref{sup_s_est}, while the second term is bounded by
$$\norm{\wt M(G)} \norm{R_{\hat m_1, \hat m_2} - R_{m_1,m_2}} \lesssim (\abs{m_1 - \hat m_1}+\abs{m_2 - \hat m_2})\norm{R_{\hat m_1, \hat m_2}}\norm{R_{m_1,m_2}}\norm{\wt M(G)}= \OO(n^{-1}),$$
where in the last step we used the trivial bound $\norm{\wt M(G)} = \OO(n^2)$. Hence
$$\wt M(G)R_{m_1,m_2}  \prec {(1 + \phi)^3\pb{q+h(\phi)}},$$ 
where we also used $n^{-1/2} \prec {h(\phi)}$, which follows from $\eta\lesssim \Im m_{2c}(z)$ by \eqref{Immc}. 
The averaged quantity $\ul{\Ga\wt M(G)R_{m_1,m_2}}$ can be estimated in the same way, and we omit the details. 
\end{proof}

Next, we state a monotonicity estimate for the resolvent, which provides an a priori bound on $\norm{G(E+\ii\eta)}$ as $\eta$ decreases.
\begin{lemma}
\label{monocity}
For any $z=E+\ii\eta\in\C_+$ and $M>1$, one has
$$\norm{G(E+\ii\eta/M)}\leq M^{3/2}\norm{G(E+\ii\eta)}.$$
\end{lemma}
\begin{proof} 
By (\ref{eq_gz}), we may write
\begin{equation}\label{Gz}
z^{1/2}G(z)=\pbb{\begin{pmatrix}
  0&Y\\
  Y^*&0
\end{pmatrix}-z^{1/2}}^{-1},\quad \forall z\in\C\backslash\R_{<0}.\end{equation}
Thus, we have the resolvent identity
\begin{equation}\label{Gresolv}
\begin{aligned}
z_1^{1/2}G(z_1)-z_2^{1/2}G(z_2)=(z_1^{1/2}-z_2^{1/2})z_1^{1/2}z_2^{1/2}G(z_1)G(z_2)
\end{aligned}
\end{equation}
for every $z_1,z_2\in\C\backslash\R_{<0}$. For any deterministic unit vectors $\uu,\vv\in\bf S$, we obtain
\begin{equation}\nonumber
\begin{aligned}
\left|\pb{z_1^{1/2}z_2^{1/2}G(z_1)G(z_2)}_{\uu\vv}\right|&=\left|z_1^{1/2}z_2^{1/2}\sum_aG(z_1)_{\uu a}G(z_2)_{a\vv}\right| \leq\sqrt{\sum_a|z_1^{1/2}G(z_1)_{\uu a}|^2}\sqrt{\sum_a|z_2^{1/2}G(z_2)_{a\vv}|^2}\\
&=\sqrt{\frac{\Im (z_1^{1/2}G(z_1))_{\uu\uu}}{\Im z_1^{1/2}}}\sqrt{\frac{\Im (z_2^{1/2}G(z_2))_{\vv\vv}}{\Im z_2^{1/2}}}
\leq\sqrt{\frac{\norm{z_1^{1/2}G(z_1)}\norm{z_2^{1/2}G(z_2)}}{\abs{\Im z_1^{1/2}}\abs{\Im z_2^{1/2}}}},
\end{aligned}
\end{equation}
where in the third step we applied Ward's identity to \eqref{Gz} with spectral parameter $z^{1/2}$. Therefore,
$$\normBB{\frac{z_1^{1/2}G(z_1)-z_2^{1/2}G(z_2)}{z_1^{1/2}-z_2^{1/2}}}\leq\sqrt{\frac{\norm{z_1^{1/2}G(z_1)}\norm{z_2^{1/2}G(z_2)}}{\abs{\Im z_1^{1/2}}\abs{\Im z_2^{1/2}}}}.$$
Set $F(z):=\norm{G(z)}$. Then
\begin{equation}\nonumber
\begin{aligned}
\abs{F(z_1)-F(z_2)}&\leq\norm{G(z_1)-G(z_2)} \leq\norma{z_1^{-1/2}(z_1^{1/2}G(z_1)-z_2^{1/2}G(z_2))}+\norma{(z_1^{-1/2}-z_2^{-1/2})z_2^{1/2}G(z_2)}\\
&\leq\abs{z_1-z_2}\pBB{\frac{\abs{z_1^{-1/2}}}{\abs{z_1^{1/2}+z_2^{1/2}}}\sqrt{\frac{\norm{z_1^{1/2}G(z_1)}\norm{z_2^{1/2}G(z_2)}}{\abs{\Im z_1^{1/2}}\abs{\Im z_2^{1/2}}}}+\frac{\norm{z_2^{1/2}G(z_2)}}{\abs{z_1^{1/2}z_2^{1/2}(z_1^{1/2}+z_2^{1/2})}}}.
\end{aligned}
\end{equation}
It follows that $F$ is locally Lipschitz continuous on $\C_+$, and its almost everywhere defined partial derivative with respect to $\eta$ satisfies
\begin{equation}\nonumber
\begin{aligned}
\left|\frac{\dd F}{\dd\eta}(z)\right|\leq \frac{1}{2\abs{z^{1/2}}}\frac{F(z)}{\abs{\Im z^{1/2}}}+\frac{F(z)}{2\abs{z}}\leq\frac{3F(z)}{2\eta},
\end{aligned}
\end{equation}
where in the second step we used $\abs{z}\geq\eta$ and $ 2\abs{z^{1/2}}\abs{\Im z^{1/2}}\geq 2\abs{\Re z^{1/2}}\abs{\Im z^{1/2}}=\eta$. This implies
$$\frac{\dd}{\dd\eta}\pb{\eta^{3/2} F(z)}\geq 0$$ 
and hence $F(E+\ii\eta/M)\leq M^{3/2}F(E+\ii\eta)$, as claimed.
\end{proof}

Another key ingredient in the proof of the local law is a stability result for the equation $f(z,\alpha)=0$. It states, roughly speaking, that if the self-consistent equation \eqref{separable_MP} is approximately satisfied, i.e., $f(z, m_{2}(z))$ is small, and $m_2(\wt z)-m_{2c}(\wt z)$ is small for $\Im\, \wt z \ge \Im\, z$, then $m_{2}(z)-m_{2c}(z)$ is also small. For an arbitrary $z\in S(c_0,C_0, \e)$, define the discrete set
\begin{align}\label{eqn_def_L}
L(z):=\{z\}\cup \{z'\in S(c_0,C_0, \e): \text{Re}\, z' = \text{Re}\, z, \text{Im}\, z'\in [\text{Im}\, z, 1]\cap (n^{-10}\mathbb N)\} .
\end{align}
Thus, if $\text{Im}\, z \ge 1$, then $L(z)=\{z\}$; if $\text{Im}\, z<1$, then $L(z)$ is a 1-dimensional lattice with spacing $n^{-10}$, together with the point $z$. In particular, $|L(z)|\le n^{10}$.

\begin{lemma}[Stability, Lemma 5.11 in \cite{yang2019edge}]
\label{stability}
Let $c_0>0$ be a sufficiently small constant and fix constants $C_0,\epsilon>0$. The self-consistent equation $f(z,\al)=0$ is stable on $S(c_0,C_0, \epsilon)$ in the following sense. Suppose that the $z$-dependent function $\delta$ satisfies $n^{-2} \le \delta(z) \le (\log n)^{-1}$ for $z\in S(c_0,C_0, \epsilon)$ and is Lipschitz continuous with Lipschitz constant $\lesssim n^2$. Suppose moreover that, for each fixed $E$, the function $\eta \mapsto \delta(E+\ii\eta)$ is non-increasing for $\eta>0$. Suppose that $u_2: S(c_0,C_0,\epsilon)\to \mathbb C$ is the Stieltjes transform of a probability measure. Let $z\in S(c_0,C_0,\epsilon)$ and assume that for all $z'\in L(z)$ we have
\begin{equation}\label{Stability0}
\left| f(z', u_2)\right| \le \delta(z').
\end{equation}
Then
\begin{equation}
\left|u_2(z)-m_{2c}(z)\right|\le \frac{C\delta}{\sqrt{\kappa+\eta+\delta}},\label{Stability1}
\end{equation}
for some constant $C>0$ independent of $z$ and $n$, where $\kappa$ is defined in (\ref{KAPPA}). 
\end{lemma}

\subsection{Proof of the local laws}

We begin by recording some useful bounds from Lemma \ref{tildeMGRgS}. Denote
\begin{equation}
\nonumber
\begin{aligned}
\mathcal{E}_{12}&:=\mathcal{E}_{12}(z):=(m_{1}(z)-m_{1c}(z))-\frac{d_n}{z}(m_{2}(z)-m_{2c}(z))\int\frac{x^{2}\pi_{A}(\dd x)}{(1+xm_{2}(z))(1+xm_{2c}(z))},\\
\mathcal{E}_{21}&:=\mathcal{E}_{21}(z):=(m_{2}(z)-m_{2c}(z))-\frac{1}{z}(m_{1}(z)-m_{1c}(z))\int\frac{x^{2}\pi_{B}(\dd x)}{(1+xm_{1}(z))(1+xm_{1c}(z))},\\
\mathcal{E}_{01}&:=\mathcal{E}_{01}(z):=(m(z)-m_{c}(z))-\frac{1}{z}(m_1(z)-m_{1c}(z))\int\frac{x\pi_{A}(\dd x)}{(1+xm_{2}(z))(1+xm_{2c}(z))},\\
\mathcal{E}_1&:=\mathcal{E}_1(z):=zm_1(z)+d_n\int\frac{t\pi_A(\dd t)}{1+tm_2(z)},\quad
\mathcal{E}_2:=\mathcal{E}_2(z):=m_2(z)+z^{-1}\int\frac{x\pi_B(\dd x)}{1+xm_1(z)}.
\end{aligned}
\end{equation}
\begin{lemma}
\label{appliof4.2}
Under the assumptions of Lemma \ref{tildeMGRgS}, we have
$$|\mathcal{E}_{12}|+|\mathcal{E}_{21}|+|\mathcal{E}_{01}|+|\mathcal{E}_1|+|\mathcal{E}_2|\prec (1+\phi)^6(q+h(\phi))^2.$$
\end{lemma}
\begin{proof}
Multiplying both sides of (\ref{GminusPi}) by $\AA^2$ (resp. $\BB^2, \tilde{I}_{p}$) and taking traces, we obtain
$$|\mathcal{E}_{12}|+|\mathcal{E}_{21}|+|\mathcal{E}_{01}|\prec (1+\phi)^6(q+h(\phi))^2$$
by Lemma \ref{tildeMGRgS}. Similarly, multiplying both sides of (\ref{GandMG}) by $\AA^2$ (resp. $\BB^2$) and taking traces yields
$$|\mathcal{E}_{1}|+|\mathcal{E}_{2}|\prec (1+\phi)^6(q+h(\phi))^2$$
by Lemma \ref{tildeMGRgS}.
\end{proof}
Next, fix $E \in [\lambda_+-c_0,C_0 \lambda_+]$. We first establish the local law on the ``global scale" for $\eta$.

\begin{lemma}
\label{eta>>1}
Under the assumptions of Theorem \ref{local_law}, there exists a sufficiently large constant $C_1>0$ such that for all $\eta\ge C_1$, the following estimates hold:
$$|m_1(z)-m_{1c}(z)|+|m_2(z)-m_{2c}(z)|\prec q^2+(n\eta)^{-1}.$$
\end{lemma}
\begin{proof}
For $\eta\ge C_1$, we have the trivial bounds $G-\Pi=\OO_{\prec}(1)$ and $\|R_{m_1,m_2}\|=\OO(1)$. Hence, taking $\phi=1$ in Lemma \ref{appliof4.2} gives
\begin{equation}
\mathcal{E}_{12}=\OO_\prec(q^2+(n\eta)^{-1}),\quad  \mathcal{E}_{21}=\OO_\prec( q^2+(n\eta)^{-1}).\label{error12}
\end{equation} 
As $\eta\to+\infty$, we have $|1/z|\to 0$. Moreover, we claim that the integrals
$$\int\frac{x^{2}\pi_{A}(\dd x)}{(1+xm_{2}(z))(1+xm_{2c}(z))},\quad \int\frac{x^{2}\pi_{B}(\dd x)}{(1+xm_{1}(z))(1+xm_{1c}(z))},$$
which appear in $\mathcal{E}_{12},\mathcal{E}_{21}$, are bounded uniformly in $\eta$ and $C_1$. Indeed, for $x\in\supp\pi_{A}\cup\supp\pi_{B}$,  
\begin{itemize}
\item $x,x^{2}=\OO(1)$ by assumption (\ref{assm3});
\item $|1+xm_{1,2c}(z)|^{-1}=\OO(1)$ by \eqref{Piii};
\item $|1+xm_{1,2}(z)|^{-1}=\OO(1)$ since $|m_{1,2}(z)|=\OO(\eta^{-1})\to 0$ as $\eta\to+\infty$.
\end{itemize}
Therefore, the equations in \eqref{error12} can be written as
\begin{align*}
& (m_{1}(z)-m_{1c}(z))+a(z) (m_{2}(z)-m_{2c}(z)) =\OO_\prec(q^2+(n\eta)^{-1}),\\
& (m_{2}(z)-m_{2c}(z))+b(z) (m_{1}(z)-m_{1c}(z))  =\OO_\prec( q^2+(n\eta)^{-1}),
\end{align*} 
where $|a(z)|,|b(z)|\le C|z|^{-1}$ for some constant $C>0$ independent of $z$ and $C_1$. Solving this system of equations yields the desired bound on $|m_{1}(z)-m_{1c}(z)|+|m_{2}(z)-m_{2c}(z)|$.
\end{proof}

The following lemma provides a bound on $G-\Pi$ from a bound on $|m_{1}(z)-m_{1c}(z)|+|m_{2}(z)-m_{2c}(z)|$ and a rough estimate of $G-\Pi$.

\begin{lemma}
\label{avrtoani} 
At $z\in S(c_0,C_0,\epsilon)$, suppose that $\norm{\Pi}=\OO(1)$, $G-\Pi=\OO_{\prec}(n^{\delta})$ for $\delta=(\tau\wedge\e)/10$, and
$$|m_{1}(z)-m_{1c}(z)|+|m_{2}(z)-m_{2c}(z)|\prec\theta$$
for some control parameter $0<\theta\leq n^{-\delta}$. Then, at this $z$, we have
$$G-\Pi \prec \theta+q+\Psi ,$$ 
where $\Psi$ is defined in (\ref{eq_defpsi}).
\end{lemma}
\begin{proof}
As in (\ref{residofR}), we have the resolvent identity
\begin{equation}
\label{residofRc}
R_{m_1,m_2}=R_{m_{1c},m_{2c}}+R_{m_{1},m_{2}}
\begin{pmatrix}
   z(m_{2c} - m_{2})A & 0   \\
   0 & z(m_{1c} - m_{1})B  \\
   \end{pmatrix}
R_{m_{1c},m_{2c}}.
\end{equation}
Combining this identity with $\norm{R_{m_{1c},m_{2c}}}=\norm{\Pi}=\OO(1)$ and $|m_{1}(z)-m_{1c}(z)|+|m_{2}(z)-m_{2c}(z)|\prec\theta\leq n^{-\delta}$, we deduce that $\norm{R_{m_1,m_2}}=\OO(1)$ with high probability. 
Applying Lemma \ref{tildeMGRgS}, we then obtain from (\ref{GminusPi}) the self-improving bound
\begin{equation}
G-\Pi \prec \phi \quad\Longrightarrow\quad G-\Pi \prec \theta+(1+\phi)^{3}(q+h(\phi)) .\label{eq:self_improve}
\end{equation}
Let $\phi_0=n^{\delta}$ and define a sequence of parameters iteratively by $\phi_{k+1}=\theta+(1+\phi_{k})^{3}(q+h(\phi_{k})).$ Note that $G-\Pi=\OO_\prec(\phi_k)$ for any fixed $k\geq 1$ by \eqref{eq:self_improve}.
For $z=E+\mathrm i\eta\in S(c_0,C_0,\e)$, we have $\phi_1\lesssim 1+n^{3\delta}\sqrt{n^{\delta}/(n\eta)}\lesssim 1$. Hence, for all $k\geq 1$, we have $\phi_k\lesssim 1$ and
\begin{equation}
\nonumber
\begin{aligned}
\phi_{k+1}&=\theta+(1+\phi_k)^{3}\left(q+\sqrt{\frac{\Im m_{2c}(z)+\phi_k}{n\eta}}\right) \lesssim \theta+q+\sqrt{\frac{\Im m_{2c}(z)}{n\eta}}+\sqrt{\frac{\phi_k}{n\eta}}.
\end{aligned}
\end{equation}
Set $\wt\Psi:=\theta+q+\Psi.$ Then $\phi_{k+1}\lesssim \wt\Psi+\sqrt{{\phi_k}/\p{n\eta}}$, and
\begin{equation}
\nonumber
\begin{aligned}
\phi_{k+2}&\lesssim \wt\Psi+\sqrt{\frac{\phi_{k+1}}{n\eta}}\lesssim \wt\Psi+\sqrt{\frac{\wt\Psi+\sqrt{{\phi_k}/\p{n\eta}}}{n\eta}} \lesssim 
\wt\Psi+\frac{\phi_{k}^{1/4}}{(n\eta)^{3/4}}.
\end{aligned}
\end{equation}
Continuing this iteration, we deduce that for any fixed $t\geq 1$,
$$\phi_{k+t}\lesssim \wt\Psi+\frac{\phi_{k}^{2^{-t}}}{(n\eta)^{1-2^{-t}}}.$$
Since $t$ can be arbitrarily large, we conclude that $G-\Pi \prec \wt\Psi+\p{n\eta}^{-1} \lesssim \wt\Psi.$
\end{proof}

Next, fix $E \in [\lambda_+-c_0,C_0 \lambda_+]$. We prove the anisotropic local law (\ref{anisotropic}) and the averaged local law (\ref{average}) for $z \in \cal L:=\{z=E+\ii \eta:\eta>0\} \cap S(c_0,C_0,\epsilon)$. By Lemma \ref{eta>>1}, there exists a constant $C_1>1$ such that (\ref{average}) holds for $\eta\ge C_1$. Enlarging $C_0$ if necessary, we may assume without loss of generality that $C_0>C_1$. Define $\mathcal{L}_0\equiv \cal L_0(E):=\{z=E+\ii \eta:\eta\ge C_1\}\cap S(c_0,C_0,\epsilon)$ and
$$\mathcal{L}_k\equiv \mathcal{L}_k(E):=\{z=E+\ii \eta:C_1n^{-2\delta k/3}\le \eta\le C_1n^{-2\delta(k-1)/3}\}  \cap S(c_0,C_0,\epsilon),\quad 1\leq k\leq K:=\lceil3/2\delta\rceil.$$
Clearly, $\cal L\subset \cup_{k=0}^{K}\mathcal{L}_{k}$. 
We proceed by induction on $k$. For $z\in\mathcal{L}_0$, (\ref{average}) follows from Lemma \ref{eta>>1}, while (\ref{anisotropic}) follows from Lemma \ref{avrtoani} with $\theta=q^2+(n\eta)^{-1}$, since trivially $G-\Pi=\OO(1)$ for $z\in \cal L_0$. The induction step is given by the following lemma, which completes the proof of the local laws \eqref{anisotropic} and \eqref{average} for all $z\in S(c_0,C_0,\e)$.
\begin{lemma}\label{lemma:inductive_local}
Fix $1\leq k\leq K$. If (\ref{anisotropic}) and (\ref{average}) hold for all $z\in\mathcal{L}_{k-1}$, then (\ref{anisotropic}) and (\ref{average}) hold uniformly for all $z\in\mathcal{L}_{k}$.
\end{lemma}

\begin{proof} 
Let $b_k:=Cn^{-2\delta(k-1)/3}$ denote the upper edge of $\mathcal{L}_k$. By the inductive hypothesis, we have $G\prec 1$ at $z=E+\ii b_k$. Together with Lemma \ref{monocity}, this implies that, for any deterministic unit vectors $\uu,\vv\in\bf S$, 
\begin{equation}
\sup_{z\in \mathcal{L}_k} |G_{\uu\vv}(z)| \prec n^{\delta}. \label{roughG}
\end{equation}
We first establish the weak local law
\begin{equation}
|m_{1}(z)-m_{1c}(z)|+|m_{2}(z)-m_{2c}(z)|\prec q^{1/2}+(n\eta)^{-1/4},\quad \forall z\in \cal L_k.\label{weaklocal}
\end{equation}
Recall the discrete set defined in \eqref{eqn_def_L}. Let $L(E+\ii b_{k+1})\cap\mathcal{L}_k=\{z_1,\cdots,z_l\}$ with $z_j=E+\ii\eta_j$, where $\eta_1>\cdots>\eta_l$. Since $m_1,m_{1c},m_2,m_{2c},G_{ab},\Pi_{ab} $ are all Lipschitz continuous with Lipschitz constant $\OO(n^2)$, it suffices to prove the weak local law for $z\in \{z_1,\cdots,z_l\}$.
We then use a standard continuity argument to propagate the local law from $z_1$ to $z_l$. Since $l$ grows with $n$, this argument must be carried out deterministically, rather than through an iteration in the sense of stochastic domination. Thus, we suppress the stochastic domination notation, fix the constants $\eps,D>0$ in its definition, and keep track of the probability loss at each downward step. Since $l\leq Cn^{10}$, the total probability loss is controlled by a union bound and is negligible.

Set $z_0:=E+\ii\eta_0:=z_1+n^{-10}\ii \in\mathcal{L}_{k-1}$. Fix a small constant $\fc \in (0,  \p{\tau\wedge\epsilon}/{10})$ and a large constant $D>0$. Define the events
$$A_j:=\left\{|m_{1}(z_j)-m_{1c}(z_j)|+|m_{2}(z_j)-m_{2c}(z_j)|\leq \frac{n^{\fc} \alpha_j}{\sqrt{\kappa+\eta_{j}+\alpha_{j}}}\right\}, j=0,\ldots, l,$$
where $\alpha_{j}:= q+(n\eta_j)^{-1/2}$. By the inductive hypothesis, the local law \eqref{average} holds at $z_0$, which gives $\prob\qb{A_0^{c}}\leq n^{-D}$ for sufficiently large $n$. 
By Lemma \ref{appliof4.2} with $\phi=n^\delta$ (due to the rough bound \eqref{roughG}), we have
\begin{equation}
\label{ToEstimatef}
|\mathcal{E}_1|+|\mathcal{E}_2|+|\mathcal{E}_{12}|\prec(1+\phi)^{6}(q+h(\phi))^2\prec q+(n\eta)^{-1/2}
\end{equation}
uniformly in $z\in\mathcal{L}_k$. Thus, by a union bound, 
$$\prob\qb{E^c}\leq n^{-D},\quad E:=\cap_{i=1}^{l}\{|\mathcal{E}_1(z_i)|+|\mathcal{E}_2(z_i)|+|\mathcal{E}_{12}(z_i)|\leq n^{\fc/4}(q+(n\eta_i)^{-1/2})\}$$
holds for sufficiently large $n$. We next prove that on the event $E$, $A_0$ implies $\cap_{i=0}^lA_i$.

Fix any $1\leq j\leq l$. We note that
\begin{align}
f(z,m_2(z))&=\int\frac{x\pi_B(\dd x)}{-z+xd_n\int\frac{t\pi_A(\dd t)}{1+tm_2(z)}}-m_2(z) =\int\frac{x\pi_B(\dd x)}{-z+x(\mathcal{E}_1-zm_1(z))}+z^{-1}\int\frac{x\pi_B(\dd x)}{1+xm_1(z)}-\mathcal{E}_2 \nonumber\\
&=-\mathcal{E}_1\int\frac{x^2\pi_B(\dd x)}{z^2(1+xm_1(z))(1+xm_1(z)-x\mathcal{E}_1/z)}-\mathcal{E}_2.
\label{f(z,m_2(z))}
\end{align}
The integral on the right-hand side of \eqref{f(z,m_2(z))} is bounded by $\OO(1)$ for $z_{1},\cdots,z_{j}$ on the event $(\cap_{i=0}^{j-1}A_i)\cap E$. Indeed, for $x\in\supp\pi_{B}$, we have:
\begin{itemize}
\item $x,x^{2}=\OO(1)$ by assumption (\ref{assm3}).
\item $|1+xm_{1c}(z)|\geq\tau'>0$ by \eqref{Piii}.
\item On $A_{j-1}$, Lipschitz continuity gives $|m_1(z_j)-m_{1c}(z_j)|+|m_2(z_j)-m_{2c}(z_j)|\leq n^{-\delta}$.  
\end{itemize}
Thus, on the event $(\cap_{i=0}^{j-1}A_i)\cap E$, we have
$$
  f(z,m_2(z))\leq n^{\fc/2}(q+(n\eta)^{-1/2}) ,\quad \forall z\in L(z_j).
$$
By Lemma \ref{stability} with $\delta(z)=n^{\fc/2}(q+(n\eta)^{-1/2})$, we obtain that
$$ |m_2(z_j)-m_{2c}(z_j)|\leq \frac{C n^{\fc/2}(q+(n\eta_j)^{-1/2}))}{\sqrt{\kappa+\eta_{j}+n^{\fc/2}(q+(n\eta_j)^{-1/2})}} \leq \frac{n^{2\fc/3}\alpha_j}{\sqrt{\kappa+\eta_{j}+\alpha_j}}.
$$
Together with the bound $|\mathcal{E}_{12}(z_j)|\le n^{\fc/4}\alpha_j$ on the event $E$, this implies that
$$
  |m_1(z_j)-m_{1c}(z_j)|\lesssim \frac{n^{2\fc/3}\alpha_j}{\sqrt{\kappa+\eta_{j}+\alpha_j}},
$$
because the integral appearing in $\mathcal{E}_{12}$ is bounded by $\OO(1)$ for the same reasons as above. Therefore, on the event $E$, $\cap_{i=0}^{j-1}A_i$ implies $\cap_{i=0}^jA_i$.
By induction, we conclude that 
\[\prob\qb{(\cap_{i=0}^lA_i)^c}\leq\prob\qb{(E\cap A_0)^c}\leq 2n^{-D}\]
for sufficiently large $n$. Since $\fc$ can be chosen arbitrarily small and $D$ arbitrarily large, we get \eqref{weaklocal}.

Finally, we upgrade the weak local law \eqref{weaklocal} to the strong one through a self-improving argument. Fix $z=E+\ii\eta \in \cal L_k$, and write
$$
\mathcal E\equiv \mathcal E(z):=\frac{q^2}{\sqrt{\kappa+\eta}}\wedge q+\frac{1}{n\eta}.
$$ 
Let $\mathcal E\leq\theta\leq  q^{1/2}+(n\eta)^{-1/4}$ be a deterministic control parameter such that $|m_{1}(z)-m_{1c}(z)|+|m_{2}(z)-m_{2c}(z)|\prec \theta$ for $z\in \mathcal{L}_k.$ 
Lemma \ref{avrtoani} gives $G-\Pi \prec \theta+q+\Psi $. Hence, by (\ref{f(z,m_2(z))}) and Lemma \ref{appliof4.2} with $\phi=\theta+q+\Psi\le 1$, we have
\begin{equation}
f(z,m_2(z))\prec q^2+[h(\theta+q+\Psi)]^2 \prec
q^2+\frac{\Im m_{2c}(z)}{n\eta}+\frac{\theta+q}{n\eta}+\frac{1}{(n\eta)^2} ,\quad \forall z\in \mathcal{L}_{k},\label{eq:self_f}
\end{equation}
Applying the AM-GM inequality to $q/(n\eta)$ and using $\Im m_{2c}(z)\lesssim\sqrt{\kappa+\eta}$ from \eqref{Immc}, we bound the right-hand side by $\OO(\delta(z))$, where
\[\delta(z):=q^2+\frac{\sqrt{\kappa+\eta}}{n\eta}+\frac{\theta}{n\eta}+\frac{1}{(n\eta)^2}.\] 
By the inductive hypothesis, the self-consistent equation \eqref{eq:self_f} also holds for $z\in\cup_{i=1}^{k-1}\mathcal{L}_i$. Applying Lemma \ref{stability} again with the above $\delta(z)$, we obtain
$$
 |m_2(z)-m_{2c}(z)|\prec \frac{\delta(z)}{\sqrt{\kappa+\eta+\delta(z)}}
 \lesssim \mathcal E+\sqrt{\frac{\theta}{n\eta}}.
$$
Moreover, since $\mathcal{E}_{12}\prec q^2+[h(\theta+q+\Psi)]^2\lesssim \mathcal E$ and the weak local law \eqref{weaklocal} ensures that the integral appearing in $\mathcal{E}_{12}$ is bounded by $\OO(1)$ with high probability (by the same argument as below \eqref{f(z,m_2(z))}), we also get
\[|m_1(z)-m_{1c}(z)|\prec \mathcal E+\sqrt{\frac{\theta}{n\eta}}.\]
Thus, we have obtained the self-improving bound
$$|m_{1}(z)-m_{1c}(z)|+|m_{2}(z)-m_{2c}(z)|\prec\theta\Longrightarrow|m_{1}(z)-m_{1c}(z)|+|m_{2}(z)-m_{2c}(z)|\prec \mathcal E+\sqrt{\frac{\theta}{n\eta}}.$$
Iterating this bound starting from the weak law, for an arbitrarily large fixed number of steps, yields
$$|m_{1}(z)-m_{1c}(z)|+|m_{2}(z)-m_{2c}(z)|\lesssim \mathcal E+\frac{1}{n\eta}\prec \mathcal E.$$
Since $\mathcal{E}_{01}\prec q^2+[h(q+\Psi)]^2\lesssim \mathcal E$, we also obtain $|m(z)-m_c(z)|\prec \mathcal E$, which proves the averaged local law (\ref{average}).
Finally, applying Lemma \ref{avrtoani} with $\theta=\mathcal E$, we conclude that $G-\Pi \prec q+\Psi $, which proves the anisotropic law (\ref{anisotropic}). 
\end{proof}

\section{Proof of Theorem \ref{distribution of outliers}}\label{sec:pf_outlier_distribution}

This section proves Theorem \ref{distribution of outliers}. Recall that the arbitrarily high moment condition \eqref{eq_highmoment} is assumed throughout this section, so by Lemma \ref{momentmethod} we may take $q=n^{-1/2}$.

\subsection{Reduction to functionals of resolvents}

As in \eqref{linearize_block}, we define the following linearization of the spiked separable covariance matrices $\ctQ_{1,2}$ in \eqref{eq_sepamodel}:
\begin{equation*}
\widetilde{H}(X,z)= z^{1/2}\begin{pmatrix} 
0 & \wt A^{1/2} X \wt B^{1/2} \\
\wt B^{1/2} X^* \wt A^{1/2} & 0 
\end{pmatrix}, \quad z\in \mathbb C_+ \cup \mathbb R.
\end{equation*}
The nonzero eigenvalues of $z^{-1/2}\wt H$ are given by 
$$\pm \sqrt{\lambda_1(\wt Q_1)}, \ \pm \sqrt{\lambda_2(\ctQ_1)}, \ \cdots \ , \ \pm \sqrt{\lambda_{p\wedge n}(\ctQ_1)}.$$
Hence, $x>0$ is an eigenvalue of $\ctQ_1$ if and only if
\be\label{detH}
\det\left(\wt H(X,x) - x\right)=0.
\ee
With the notation in (\ref{AOBO}) and (\ref{eq_sepamodel}), we can write
\begin{equation}\label{Pepsilon0}
\widetilde{H}(X,z) = P H(X,z)P, 
\end{equation}
where
$$P:=\begin{pmatrix} 
\left(1+ V_o^a D^a (V_o^a)^*\right)^{1/2}  & 0  \\
0 & \left(1+ V_o^b D^b (V_o^b)^*\right)^{1/2} 
\end{pmatrix}.$$
Introduce the $ (p+n) \times (r+s)$ matrix $\mathbf U$ and the  $(r+s) \times (r+s)$ diagonal matrix $\mathcal{D}$ by
\begin{equation}\label{Depsilon0}
\mathbf{U}=
\begin{pmatrix}
V_{o}^a & 0 \\
0 & V_{o}^b
\end{pmatrix}, \quad \mathcal{D}=
\begin{pmatrix}
D^a(D^a+1)^{-1} & 0 \\
0 & D^b(D^b+1)^{-1}
\end{pmatrix}.
\end{equation}
The next lemma gives the master equation for the locations of the outlier eigenvalues.

\begin{lemma}\label{lem_pertubation} 
If $x\ne 0$ is not an eigenvalue of $\mathcal Q_1$, then it is an eigenvalue of $\widetilde{\mathcal Q}_1$ if and only if
\begin{equation}\label{masterx}
\det\left(\mathcal{D}^{-1}+x \mathbf U^* G(x) \mathbf U\right)=0.
\end{equation} 
\end{lemma}
\begin{proof}
Since $P$ is invertible, \eqref{detH} implies that $x \ne 0$ is an eigenvalue of $\wt H=PHP$ if and only if
\begin{align*}
0&=\det(PHP-x)=\det\Big(P(H-P^{-2}x)P\Big)\\
&=\det(P^2) \det(G(x)^{-1})\det \left(1+xG(x)(1-P^{-2}) \right) \\
&=\det(P^2) \det(G(x)^{-1})\det \left(1+xG(x) \mathbf{U} \mathcal{D} \mathbf{U}^* \right) \\
&=\det(P^2) \det(G(x)^{-1}) \det \Big(1+x \mathbf{U}^* G(x) \mathbf{U} \mathcal{D}  \Big) \nonumber \\
& =\det(P^2) \det(G(x)^{-1})\det(\mathcal{D})\det(\mathcal{D}^{-1}+x \mathbf{U}^* G(x) \mathbf{U}),
\end{align*}
where in the penultimate step we used the identity $\det(1+AB)=\det(1+BA)$ for any matrices of conformable dimensions. The claim follows.
\end{proof}

Using the characterization \eqref{masterx}, we reduce the study of the asymptotic distribution of the outliers to that of the resolvent functional $\mathbf U^* G(x) \mathbf U$ for the non-spiked separable covariance model.

\begin{lemma}\label{reduction to Green functions}
Under the assumptions of Theorem \ref{distribution of outliers}, for each $1\le i\le r$ and $1\le \mu \le s$, we have 
\begin{align}\label{eq:outlier_expansion1}
\wt\lambda_{\alpha(i)}-\theta_i^a&=-\frac{1}{\sigma_i^am'_{2c}(\theta_i^a)}\pa{\frac{d_i^a}{d_i^a+1}}^2\cdot\theta_i^a\vv_i^*\pa{G(\theta_i^a)-\Pi(\theta_i^a)}\vv_i+\OO_{\prec}(n^{-1}),\\
\wt\lambda_{\beta(\mu)}-\theta_\mu^b&=-\frac{1}{\sigma_\mu^bm'_{1c}(\theta_\mu^b)}\pbb{\frac{d_\mu^b}{d_\mu^b+1}}^2\cdot\theta_\mu^b \uu_\mu^*\pa{G(\theta_\mu^b)-\Pi(\theta_\mu^b)}\uu_\mu+\OO_{\prec}(n^{-1}).\label{eq:outlier_expansion2}
\end{align}
where $\vv_i:=((\vv_i^a)^T, 0)^T\in\RR^{p+n}$ and $\uu_\mu:= (0^T, (\vv_\mu^b)^T)^T\in\RR^{p+n}$ are the canonical extensions of the population eigenvectors $\vv_i^a$ and $\vv_\mu^b$ to $\RR^{p+n}$, respectively.
\end{lemma}

To prove this lemma, we need the following strengthened anisotropic local law outside the spectrum. 

\begin{theorem}[Anisotropic local law outside the spectrum]\label{lem_localout}
Under the assumptions of Theorem \ref{local_law}, fix any constant $\epsilon>0$. Then, for any 
\begin{equation}\label{eq_paraout}
z\in S_{\mathrm{out}}(C_0,\epsilon):=\left\{z= E+ \ii\eta: \lambda_+ + n^{-2/3+\e} + n^{-1/3+\e}  q^2  \le E \le C_0 \lambda_+,  \eta\in [0,1]\right\},
\end{equation}
and any deterministic unit vectors $\uu, \vv \in \mathbb{C}^{p+n}$, the following anisotropic local law holds:
\begin{equation}\label{aniso_outstrong}
\begin{split}
\left| \langle \mathbf u, G(X,z) \mathbf v\rangle - \langle \mathbf u, \Pi (z)\mathbf v\rangle \right| & \prec q + \sqrt{\frac{\Im m_{2c}(z)}{n \eta}}  \asymp q + n^{-1/2}(\kappa+\eta)^{-1/4} .
\end{split}
\end{equation}
\end{theorem}

\begin{proof}
The main idea is to compare $G(z)$ (resp. $\Pi(z)$) with $G(z+\ii\eta_0)$ (resp. $\Pi(z+\ii\eta_0)$) for $z$ satisfying $\Im z\leq \eta_0:=n^{-1/2}\kappa^{1/4}+n^{-1/2+\e}q$. Note that $z+\ii\eta_0$ belongs to $S(c_0,C_0,\e)$ even if $z$ itself does not. Hence Theorem \ref{local_law} applies to $G(z+\ii\eta_0)$ and $\Pi(z+\ii\eta_0)$, giving a bound on $G(z+\ii\eta_0)-\Pi(z+\ii\eta_0)$ that is stronger than the right-hand side of \eqref{aniso_outstrong}. Therefore, to prove \eqref{aniso_outstrong} at $z$, it suffices to control the differences $G(z)-G(z+\ii\eta_0)$ and $\Pi(z)-\Pi(z+\ii\eta_0)$.  
The deterministic quantity $\Pi(z)-\Pi(z+\ii\eta_0)$ can be controlled directly by the H{\"o}lder continuity of $m_{1c},m_{2c}$ (see Lemma C.7 of \cite{ding2020spikedseparablecovariancematrices}). For the random part $G(z)-G(z+\ii\eta_0)$, we use the spectral decomposition of $G$ together with eigenvalue rigidity \eqref{eq:rigidity}. 
We refer the reader to the proof of \cite[Theorem C.12]{ding2020spikedseparablecovariancematrices} for details. 
\end{proof}

\begin{proof}[Proof of Lemma \ref{reduction to Green functions}] 
We rewrite (\ref{masterx}) as
\begin{equation}\label{eq:D_VUV}
\det\left(\mathcal{D}^{-1}+x \mathbf U^*\Pi(x) \mathbf U+x \mathbf U^*\pb{G(x)-\Pi(x)} \mathbf U\right)=0,
\end{equation} 
where the deterministic part $\mathcal{D}^{-1}+x \mathbf U^*\Pi(x) \mathbf U$ is the diagonal matrix
\begin{equation}\label{eq:deterministic_part}
\textup{diag}\left\{\frac{1+d_i^a}{d_i^a}-\frac{1}{1+m_{2c}(x)\sigma_i^a}\right\}_{i=1}^r\oplus\textup{diag}\left\{\frac{1+d_\mu^b}{d_\mu^b}-\frac{1}{1+m_{1c}(x)\sigma_\mu^b}\right\}_{\mu=1}^s,
\end{equation} 
and the random part $x \mathbf U^*(G(x)-\Pi(x)) \mathbf U$ fluctuates on the scale $n^{-1/2}$ by Theorem \ref{lem_localout}, provided that $x$ stays at a distance of order 1 from the limiting spectrum. We will focus on the proof of \eqref{eq:outlier_expansion1}; the proof of the second expansion \eqref{eq:outlier_expansion2} is completely analogous. 

By the definition of $\theta_i^a$ in (\ref{g12c}), the $i$-th diagonal entry $\p{1+d_i^a}/{d_i^a}-\p{1+m_{2c}(x)\sigma_i^a}^{-1}$ vanishes at $x=\theta_i^a$. Moreover, by Assumption \ref{ass:nonoverlapping}, all other $r+s-1$ deterministic diagonal entries of $\mathcal{D}^{-1}+x \mathbf U^*\Pi(x) \mathbf U$ are bounded away from 0 by a constant. By Assumption \ref{ass:awayfrombulk}, $\theta_i^a$ is also separated from the limiting spectrum by a distance of order 1; that is, $\theta_i^a\ge \lambda_++c$ for some constant $c>0$ depending on $\tau$. 
Applying Theorem \ref{thm_outlier} with $q=n^{-1/2}$, we find that 
\begin{equation}\label{eq:outlier_i}
  \wt\lambda_{\al(i)}-\theta_i^a=\OO_\prec(n^{-1/2}),
\end{equation}
and that, with high probability, all other outlier eigenvalues remain at least a constant away from $\theta_i^a$. Applying Schur's complement formula to the $(i,i)$-th position of \eqref{eq:D_VUV} at \smash{$x=\wt\lambda_{\al(i)}$}, we obtain
\begin{equation}\label{outlier position}
\frac{1+d_i^a}{d_i^a}-\frac{1}{1+m_{2c}(x)\sigma_i^a}+x\vv_i^*\pb{G(x)-\Pi(x)}\vv_i+A^*\mathsf{C}^{-1}B=0,
\end{equation} 
where $A,B,\mathsf{C}$ are the complements of the $(i,i)$-position in the $(r+s)\times(r+s)$ matrix $\mathcal{D}^{-1}+x \mathbf U^* G(x) \mathbf U$.

By the local law \eqref{aniso_outstrong}, we have \smash{$\mathbf U^*\pb{G(x)-\Pi(x)}\mathbf U \prec n^{-1/2}$} for $x$ in a neighborhood of $\theta_i^a$. Hence
\begin{equation}\label{eq:vGv}
\vv_i^*\pb{G(x)-\Pi(x)}\vv_i\prec n^{-1/2},\quad \|A\|+\|B\|=\OO_{\prec}(n^{-1/2}),
\end{equation}
while $\mathsf{C}$ is an $\OO_{\prec}(n^{-1/2})$ random perturbation of the $(r+s-1)\times (r+s-1)$ deterministic diagonal matrix obtained by removing the $i$-th diagonal entry from \eqref{eq:deterministic_part}. 
Thus, with high probability, $\mathsf{C}$ is invertible with $\norm{\mathsf{C}^{-1}}=\OO(1)$, and $A^*\mathsf{C}^{-1}B=\OO_{\prec}(n^{-1})$. 
We may therefore rewrite \eqref{outlier position} as
\begin{equation}\label{outlier position2}
\pb{1+m_{2c}(\wt\lambda_{\al(i)})\sigma_i^a}^{-1}-\pb{1+m_{2c}(\theta_i^a)\sigma_i^a}^{-1}=\wt\lambda_{\al(i)}\vv_i^*\pb{G(\wt\lambda_{\al(i)})-\Pi(\wt\lambda_{\al(i)})}\vv_i+\OO_{\prec}(n^{-1}).
\end{equation} 
Taylor expanding the right-hand side around $x=\theta_i^a$ and using the local law estimate \eqref{eq:vGv} together with the eigenvalue estimate \eqref{eq:outlier_i}, we obtain, with high probability,
\begin{align}
&~\left|\wt\lambda_{\al(i)} \vv_i^*\pb{G(\wt\lambda_{\al(i)})-\Pi(\wt\lambda_{\al(i)})}\vv_i - \theta_i^a \vv_i^*\pb{G(\theta_i^a)-\Pi(\theta_i^a)}\vv_i\right| \nonumber\\
 \prec &~n^{-1}+ |\wt\lambda_{\al(i)}-\theta_i^a|\sup_{\xi \in E_n}\left|\vv_i^*\pb{G'(\xi)-\Pi'(\xi)}\vv_i\right| \label{eq:Taylor1}
\end{align} 
where $E_n:=\{\xi\in \R:|\xi-\theta_i^a|\le n^{-1/4}\}$. By Cauchy's integral formula,
\[ \vv_i^*\pb{G'(\xi)-\Pi'(\xi)}\vv_i = \frac{1}{2\pi \ii}\oint_{|w-\xi|=\epsilon} \frac{\vv_i^*\pb{G(w)-\Pi(w)}\vv_i}{(w-\xi)^{2}}\dd w, \]
where $\e>0$ is a small constant chosen so that the relevant contour does not contain any other outliers with high probability; this is possible by Assumption \ref{ass:nonoverlapping} and Theorem \ref{thm_outlier}. Combining this formula with the local law for $\vv_i^*\pb{G(w)-\Pi(w)}\vv_i$, we derive that, for each fixed $\xi\in E_n$,
\begin{equation}
\vv_i^*\pb{G'(\xi)-\Pi'(\xi)}\vv_i \prec n^{-1/2}\label{eq:control_G'}
\end{equation}
An $n^{-10}$-net and perturbation argument upgrades this estimate to a uniform one over $\xi\in E_n$; we omit the details since this argument is standard in the random matrix theory literature. Substituting \eqref{eq:control_G'} and \eqref{eq:outlier_i} into \eqref{eq:Taylor1}, we obtain
\begin{equation}
 \wt\lambda_{\al(i)}\vv_i^*\pb{G(\wt\lambda_{\al(i)})-\Pi(\wt\lambda_{\al(i)})}\vv_i =\theta_i^a \vv_i^*\pb{G(\theta_i^a)-\Pi(\theta_i^a)}\vv_i  +\OO_\prec (n^{-1}). \label{eq:Taylor2}
\end{equation}
On the other hand, Taylor expanding the left-hand side of \eqref{outlier position2} and using \eqref{Piii} and \eqref{eq:outlier_i}, we obtain that
\begin{equation}\label{eq:Taylor3}
\pb{1+m_{2c}(\wt{\lambda}_{\al(i)})\sigma_i^a}^{-1}-\pb{1+m_{2c}(\theta_i^a)\sigma_i^a}^{-1}=-\pa{\frac{d_i^a+1}{d_i^a}}^2\sigma_i^am'_{2c}(\theta_i^a)\pb{\wt{\lambda}_{\al(i)}-\theta_i^a}+\OO_{\prec}(n^{-1}).
\end{equation} 
Plugging \eqref{eq:Taylor2} and \eqref{eq:Taylor3} into \eqref{outlier position2}, we conclude \eqref{eq:outlier_expansion1}.
\end{proof}
\subsection{Distribution of Green functions}

By Lemma \ref{reduction to Green functions}, under the scaling $\sqrt{n}$, it suffices to prove that $\theta_i^a\cdot \sqrt{n}\vv_i^*\pb{G(\theta_i^a)-\Pi(\theta_i^a)}\vv_i$ and $\theta_\mu^b \cdot \sqrt{n} \uu_\mu^*\pb{G(\theta_\mu^b)-\Pi(\theta_\mu^b)}\uu_\mu$ are asymptotically Gaussian with desired variances. We focus on the former at $\theta_i^a$, since the latter can be treated in the same way. We prove the asymptotic Gaussianity of $\theta_i^a\cdot \sqrt{n}\vv_i^*\pb{G(\theta_i^a)-\Pi(\theta_i^a)}\vv_i$ by showing that its moments asymptotically match those of a Gaussian random variable. This moment matching follows from the recursive moment relation established in Lemma \ref{recursive moment estimate} below.
For technical reasons, we prove the recursive moment relation in Lemma \ref{recursive moment estimate} at a point $z$ with a small imaginary perturbation of $\theta_i^a$. A simple continuity argument will then transfer the asymptotic distribution from $z$ back to $\theta_i^a$. Set $z=\theta_i^a+\ii n^{-2}$ and define
$$\mathcal{P}_i^a(z):=\sqrt{n}\cdot z\vv_i^*\pb{G(z)-\Pi(z)}\vv_i.$$ 
By the local law \eqref{aniso_outstrong}, we have $\mathcal{P}_i^a(z)\prec 1$. Moreover, for this choice of $z$, the deterministic bound $\norm{G(z)}\lesssim (\Im z^{1/2})^{-1}\sim n^{2}$ holds. Hence, $|\mathcal{P}_i^a(z)| \lesssim n^{5/2}$ deterministically, and part (iii) of Lemma \ref{lem_stodomin} can be applied to give \( \E|\mathcal{P}_i^a(z)|^\ell \prec 1 \) for any fixed $\ell\in \N_+$. This is the reason for introducing a small imaginary part in $z$. 

\begin{lemma}\label{recursive moment estimate}
Under the assumptions of Theorem \ref{distribution of outliers}, set $z=\theta_i^a+\ii n^{-2}$ for some $1\le i \le r$. Then, for $\mathcal{P}_i^a\equiv \mathcal{P}_i^a(z)$ and any fixed $\ell\geq 2$, we have 
\begin{equation}\label{expansion result}
\begin{aligned}
\expect[(\mathcal{P}_i^a)^\ell]=&~(\ell-1)\pbb{\frac{d_i^a+1}{d_i^a}}^4(\sigma_i^a)^2\pB{2m_{2a}(\theta_i^a)+k_4b(\theta_i^a)s_4(\vv_i^a)}\expect[(\mathcal{P}_i^a)^{\ell-2}] +\OO_\prec(n^{-1/6}).
\end{aligned}
\end{equation} 
Moreover, for $\ell=1$, we have $\expect[\mathcal{P}_i^a]=\OO_\prec(n^{-\e})$ for some constant $\e>0$.
\end{lemma}

Combining this lemma with Lemma \ref{reduction to Green functions}, we now complete the proof of Theorem \ref{distribution of outliers}.

\begin{proof}[Proof of Theorem \ref{distribution of outliers}]
By Lemma \ref{recursive moment estimate}, the moments of $\mathcal{P}_i^a(z)$ asymptotically match those of a Gaussian random variable with mean 0 and variance
\begin{equation}
\pB{\frac{d_i^a+1}{d_i^a}}^4(\sigma_i^a)^2\pB{2m_{2a}(\theta_i^a)+k_4b(\theta_i^a)s_4(\vv_i^a)}.\label{eq:aymp_var}
\end{equation}
To compare $\mathcal{P}_i^a(z)$ with $\mathcal{P}_i^a(\theta_i^a)$, we use the rigidity estimate \eqref{eq:rigidity} and the supercritical assumption \eqref{eq:supercritical}, which give
\begin{equation}\label{eq:supercritical2}
\theta_i^a \ge \lambda_+ + c \ge  \lambda_{1}(\cal Q_1) + c+ \OO_\prec(n^{-2/3})
\end{equation}
for a sufficiently small constant $c>0$. Consequently, by the definition \eqref{def_green}, we get
\begin{equation}\nonumber
\sup_{\xi\in \C:|\xi-\theta_i^a|\le c/2}\left(\|\cal G_1 (\xi)\|+\|\cal G_2 (\xi)\|+ \|\cal G_1' (\xi)\|+\|\cal G_2' (\xi)\|\right) =\OO(1)\quad \text{with high probability},
\end{equation}
Together with the representation \eqref{green2} and Lemma \ref{priori}, this implies
\begin{equation}\label{eq:bound_G12}
\sup_{\xi\in \C:|\xi-\theta_i^a|\le c/2}\left(\|G (\xi)\| + \| G' (\xi)\|\right) =\OO(1)\quad \text{with high probability}.
\end{equation}
On the other hand, combining \eqref{eq:supercritical2} with the fact that $m_{\al c}$ is the Stieltjes transform of the density $\rho_{\al c}$ supported on $[0,\lambda_+]$, we get
\begin{equation}\label{eq:bound_m0}
  \sup_{\xi\in \C:|\xi-\theta_i^a|\le c/2}\left(|m_{1c} (\xi)|+|m_{2c} (\xi)|+ |m_{1c}' (\xi)|+|m_{2c}' (\xi)|\right) =\OO(1), 
\end{equation} 
Together with the representation \eqref{defn_pi} and the condition \eqref{Piii}, this implies
\begin{equation}\label{eq:bound_m12}
  \sup_{\xi\in \C:|\xi-\theta_i^a|\le c/2}\left(\|\Pi (\xi)\|+ \|\Pi' (\xi)\|\right) =\OO(1). 
\end{equation}
Using \eqref{eq:bound_G12} and \eqref{eq:bound_m12}, we readily obtain
\begin{equation}\label{eq:diff_G_m}
  \|G(z)-G(\theta_i^a)\|=\OO(n^{-2}) \ \ \text{with high probability}, \quad \text{and}\quad  \|\Pi(z)-\Pi(\theta_i^a)\|=\OO(n^{-2}). 
\end{equation}
These estimates imply that $\mathcal{P}_i^a(z)-\mathcal{P}_i^a(\theta_i^a)=\OO(n^{-3/2})$ with high probability, which is negligible under the $\sqrt{n}$ scaling. Hence, $\mathcal{P}_i^a(\theta_i^a)$ is also asymptotically Gaussian with mean 0 and variance \eqref{eq:aymp_var}. Together with \eqref{eq:outlier_expansion1}, this proves the first assertion of Theorem \ref{distribution of outliers} for \smash{$\wt\lambda_{\alpha(i)}$}. The second assertion for \smash{$\wt\lambda_{\beta(\mu)}$} follows in the same way, by establishing the analogous recursive moment relation as in Lemma \ref{recursive moment estimate}; we omit the details. 
\end{proof}

The rest of this section is devoted to the proof of Lemma \ref{recursive moment estimate}. We focus on deriving the recursive relation \eqref{expansion result} for $\ell\ge 2$. The case $\ell=1$ follows from the same, and in fact simpler, argument, since all leading contributions carrying the factor $(\ell-1)$ vanish. 

Following \cite{bao2020statisticalinferenceprincipalcomponents}, we prove Lemma \ref{recursive moment estimate} using the cumulant expansion formula in Lemma \ref{lem:cumulant_expansion}. In the following, unless otherwise specified, all spectral parameters of $\mathcal{P}_i^a$, $G$, and $\Pi$ are set to $z=\theta_i^a+\ii n^{-2}$. We rewrite $\cal P_i^a$ as
\begin{align}
  \mathcal{P}_i^a&=\sqrt{n}\cdot z\vv_i^*\p{G-\Pi}\vv_i =\sqrt{n} \pa{z\vv_i^*G \vv_i+\qa{1+m_{2c}(z)\sigma_i^a}^{-1}}  \nonumber\\
&=\sqrt{n}\vv_i^*\pa{z G+\qa{1+m_{2c}(z)\sigma_i^a}^{-1}\p{HG-zG}}\vv_i  =\sqrt{n}\vv_i^*\pa{\frac{HG}{1+m_{2c}(z)\sigma_i^a}+\frac{zm_{2c}(z)\sigma_i^a}{1+m_{2c}(z)\sigma_i^a}G}\vv_i  \nonumber\\
&=\sqrt{n} \frac{\sqrt{z\sigma_i^a}}{1+m_{2c}(z)\sigma_i^a}\sum_{j,\mu}\vv_i(j)x_{j\mu}(\BB G)_{\mu\vv_i}+\sqrt{n}\frac{zm_{2c}(z)\sigma_i^a}{1+m_{2c}(z)\sigma_i^a}G_{\vv_i\vv_i} ,\label{eq:expansion_P}
\end{align}
Applying the cumulant expansion formula in Lemma \ref{lem:cumulant_expansion} to the first term on the right-hand side of \eqref{eq:expansion_P}, we obtain
\begin{align}
&\expect[(\mathcal{P}_i^a)^\ell]=\sqrt{n}\frac{\sqrt{z\sigma_i^a}}{1+m_{2c}(z)\sigma_i^a}\sum_{j,\mu}\vv_i(j)\expect[x_{j\mu}(\BB G)_{\mu\vv_i}(\mathcal{P}_i^a)^{\ell-1}]+\sqrt{n}\frac{zm_{2c}(z)\sigma_i^a}{1+m_{2c}(z)\sigma_i^a}\expect[G_{\vv_i\vv_i}(\mathcal{P}_i^a)^{\ell-1}] \nonumber\\
&=\sqrt{n}\frac{\sqrt{z\sigma_i^a}}{1+m_{2c}(z)\sigma_i^a}\sum_{j,\mu}\vv_i(j)\left(\sum_{\alpha=1}^3\frac{\kappa_{\alpha+1}(x_{j\mu})}{\alpha!}\expect[\partial_{j\mu}^\alpha\pb{(\BB G)_{\mu\vv_i}(\mathcal{P}_i^a)^{\ell-1}}]+\mathcal{R}_i^{j\mu}\right) +\sqrt{n}\frac{zm_{2c}(z)\sigma_i^a}{1+m_{2c}(z)\sigma_i^a}\expect[G_{\vv_i\vv_i}(\mathcal{P}_i^a)^{\ell-1}]\nonumber\\
&=\sqrt{n}\frac{\sqrt{z\sigma_i^a}}{1+m_{2c}(z)\sigma_i^a}\pbb{\sum_{\substack{\alpha_1,\alpha_2\geq 0\\ 1\leq\alpha_1+\alpha_2\leq 3}}\expect[h_i(\alpha_1,\alpha_2)]+\mathcal{R}_i}+\sqrt{n}\frac{zm_{2c}(z)\sigma_i^a}{1+m_{2c}(z)\sigma_i^a}\expect[G_{\vv_i\vv_i}(\mathcal{P}_i^a)^{\ell-1}],\label{expansion}
\end{align}
where we use the abbreviations
$$h_i(\alpha_1,\alpha_2):=\sum_{j,\mu}\vv_i(j)\frac{\kappa_{\alpha_1+\alpha_2+1}(x_{j\mu})}{\alpha_1!\alpha_2!}\partial_{j\mu}^{\alpha_1}(\BB G)_{\mu\vv_i}\cdot \partial_{j\mu}^{\alpha_2}(\mathcal{P}_i^a)^{\ell-1},\quad \mathcal{R}_i:= \sum_{j,\mu}\vv_i(j)\mathcal{R}_i^{j\mu},$$
and the remainder term $\mathcal{R}_i^{j\mu}$ satisfies, for any constant $0<\e<1/2$,
\begin{equation}\label{remainder0}
\begin{aligned}
\mathcal{R}_i^{j\mu}\lesssim &~ \expect\qa{\abs{x_{j\mu}}^5}\E\pB{ \sup_{\abs{\xi_{j\mu}}\leq n^{-1/2+\e}}\absb{\partial_{j\mu}^4\pb{(\BB \wt G)_{\mu\vv_i}(\wt{\mathcal{P}}_i^a)^{\ell-1}}} }\\
&~+ \expect\qa{\abs{x_{j\mu}}^5\mathds{1}(|x_{i\mu}|>n^{-1/2+\e})}\E\pB{ \sup_{\xi_{j\mu}\in \R}\absb{\partial_{j\mu}^4\pb{(\BB \wt G)_{\mu\vv_i}(\wt{\mathcal{P}}_i^a)^{\ell-1}}} }.
\end{aligned}
\end{equation}
Here $\xi_{j\mu}$ is a real {\it deterministic} parameter, $\wt G$ is obtained from $G$ by replacing $x_{j\mu}$ with $\xi_{j\mu}$, and $\wt{\mathcal{P}}_i^a$ is obtained from $\mathcal{P}_i^a$ by replacing $G$ with $\wt G$. We next estimate the terms $h_i(\alpha_1,\alpha_2)$ and $\mathcal{R}_i$ one by one. 

\subsubsection{Estimate of $h_i(1,0)$}
Using \eqref{Gderivative}, we obtain
\begin{align}
h_i(1,0)&=\frac{1}{n}\sum_{j,\mu}\vv_i(j)\qa{\partial_{j\mu}(\BB G)_{\mu\vv_i}}(\mathcal{P}_i^a)^{\ell-1} \nonumber\\
&=-\frac{\sqrt{z}}{n}\sum_{j,\mu}\vv_i(j)\pb{(\BB G\AA)_{\mu j}(\BB G)_{\mu \vv_i}+(\BB G\BB)_{\mu\mu}(\AA G)_{j\vv_i}}(\mathcal{P}_i^a)^{\ell-1}\nonumber\\
&=-\frac{\sqrt{z}}{n}\pb{(G\BB^2 G\AA)_{\vv_i\vv_i}+\Tr(\BB G\BB)\cdot (\AA G)_{\vv_i\vv_i}}(\mathcal{P}_i^a)^{\ell-1}\nonumber\\
&=-\frac{\sqrt{z\sigma_i^a}}{n}\pa{ (G\BB^2 G)_{\vv_i\vv_i}+ \Tr(B\mathcal{G}_2)\cdot  G_{\vv_i\vv_i}}(\mathcal{P}_i^a)^{\ell-1}\nonumber\\
&=- \frac{\sqrt{z\sigma_i^a}}{n}(G\BB^2 G)_{\vv_i\vv_i}(\mathcal{P}_i^a)^{\ell-1}-\sqrt{z\sigma_i^a} \cdot m_2 G_{\vv_i\vv_i}(\mathcal{P}_i^a)^{\ell-1}.\label{h10}
\end{align}
By \eqref{eq:bound_G12}, the first term is of size $\OO_{\prec}(n^{-1})$. The second term cancels the second term on the right-hand side of (\ref{expansion}), leading to
\begin{align}
&~\sqrt{n}\frac{\sqrt{z\sigma_i^a}}{1+m_{2c}(z)\sigma_i^a}\expect[h_i(1,0)] + \sqrt{n}\frac{zm_{2c}(z)\sigma_i^a}{1+m_{2c}(z)\sigma_i^a}\expect[G_{\vv_i\vv_i}(\mathcal{P}_i^a)^{\ell-1}] \nonumber\\
=&~\sqrt{n}\frac{z\sigma_i^a}{1+m_{2c}(z)\sigma_i^a}\expect\qa{\pa{m_{2c}(z)-m_2(z)}G_{\vv_i\vv_i}(\mathcal{P}_i^a)^{\ell-1}}+ \OO_{\prec}(n^{-1/2}).\label{cancel_h10}
\end{align}
It remains to control the difference $m_2(z)-m_{2c}(z)$ in the first term. Take $\eta_0=n^{-2/3}$ and set $z_0=z+\ii\eta_0$. By the averaged local law \eqref{eq:averaged_outside} at $z_0$, we have
$$m_2(z_0)-m_{2c}(z_0) \prec {n}^{-1}+(n\eta_0)^{-2}\lesssim n^{-2/3}.$$ 
On the other hand, by the Lipschitz continuity of $m_2$ and $m_{2c}$ from \eqref{eq:bound_m0}, we have
$m_2(z_0)-m_{2}(z)\lesssim \eta_0$ with high probability and $m_{2c}(z_0)-m_{2c}(z)\lesssim \eta_0$. Together with the preceding bound, this yields
\begin{equation}\label{eq:good_aver_out}
m_2(z)-m_{2c}(z) \prec n^{-2/3} .
\end{equation}
Substituting this estimate into \eqref{cancel_h10}, we obtain an error of order $\OO_\prec(n^{-1/6})$:
\begin{align}
&~\sqrt{n}\frac{\sqrt{z\sigma_i^a}}{1+m_{2c}(z)\sigma_i^a}\expect[h_i(1,0)] + \sqrt{n}\frac{zm_{2c}(z)\sigma_i^a}{1+m_{2c}(z)\sigma_i^a}\expect[G_{\vv_i\vv_i}(\mathcal{P}_i^a)^{\ell-1}] \prec n^{-1/6}.\label{final_h10}
\end{align}

\subsubsection{Estimate of $h_i(0,1)$}

The derivative of $\mathcal{P}_i^a$ is given by
\begin{equation}\label{Pderivative}
\partial_{j\mu}\mathcal{P}_i^a=-\sqrt{n} z^{3/2}\pb{(G\AA)_{\vv_i j}(\BB G)_{\mu\vv_i}+(G\BB)_{\vv_i\mu}(\AA G)_{j\vv_i}} =-2\sqrt{n}z^{3/2}(G\AA)_{\vv_i j}(\BB G)_{\mu\vv_i}.
\end{equation} 
Using this identity, we compute $h_i(0,1)$ as
\begin{align}
h_i(0,1)&= \frac{1}{n}\sum_{j,\mu}\vv_i(j)(\BB G)_{\mu\vv_i}\partial_{j\mu}(\mathcal{P}_i^a)^{\ell-1} =-2(\ell-1) \frac{z^{3/2} }{\sqrt{n}}\sum_{j,\mu}\vv_i(j)(\BB G)_{\mu\vv_i}(G\AA)_{\vv_i j}(\BB G)_{\mu\vv_i}(\mathcal{P}_i^a)^{\ell-2}\nonumber\\
&=-2(\ell-1) \frac{z^{3/2}}{\sqrt{n}}(G\BB^2 G)_{\vv_i \vv_i}(G\AA)_{\vv_i \vv_i}(\mathcal{P}_i^a)^{\ell-2}\nonumber\\
&=-2(\ell-1) \frac{z^{3/2}\sqrt{\sigma_i^a}}{\sqrt{n}}(G\BB^2 G)_{\vv_i \vv_i}G_{\vv_i \vv_i}(\mathcal{P}_i^a)^{\ell-2}.\label{h01}
\end{align}
The entry $(G\BB^2 G)_{\vv_i \vv_i}G_{\vv_i \vv_i}$ has a nonzero limit, which contributes a nontrivial variance term. The deterministic limit of $G_{\vv_i \vv_i}$ follows from the local law \eqref{aniso_outstrong}:
\begin{equation}
\label{eq:Gvivi}
G_{\vv_i \vv_i}=\vv_i^*  \Pi(z)\vv_i +\OO_\prec(n^{-1/2}),\quad \text{with}\quad \vv_i^*  \Pi(z)\vv_i=-\frac{1}{z[1+m_{2c}(z)\sigma_i^a]}.
\end{equation}

The deterministic limit of $(G\BB^2 G)_{\vv_i \vv_i}$ is more involved. To compute it, we introduce a class of resolvents depending on a parameter $\omega\in\C$:
\begin{equation}\label{G_omega}
\begin{aligned}
G_\omega&=G_\omega(X,z):=\left(H(X,z)-\omega\BB^2-z\right)^{-1}\\
&=\qbb{z^{1/2}\begin{pmatrix}
  A^{1/2}&0 \\
  0&B^{1/2}
\end{pmatrix}\begin{pmatrix}
  0&X \\
  X^*&0
\end{pmatrix}\begin{pmatrix}
  A^{1/2}&0 \\
  0&B^{1/2}
\end{pmatrix}-\omega\begin{pmatrix}
  0&0 \\
  0&B
\end{pmatrix}-zI}^{-1}.
\end{aligned}
\end{equation} 
Note that $G_\omega$ is defined so that \( \left. \partial_\omega\right|_{\omega=0} G_\omega =G\BB^2G\). Moreover, since we work outside the limiting spectrum of $\cal Q_1$ and $\cal Q_2$, $G_\omega$ is a well-defined analytic (matrix-valued) function of $\omega\in\C$ for $\abs{\omega}$ sufficiently small. Hence, by Cauchy's integral formula with respect to $\omega$, to identify the deterministic limit of $\left. \partial_\omega \right|_{\omega=0}G_\omega $, it suffices to establish a local law for $G_\omega$ with $\abs{\omega}\le \e$, where $\e>0$ is a sufficiently small constant. 
We rewrite $G_\omega$ as \(G_\omega(z)= G(z)[1-\omega \BB^2 G(z)]^{-1}\) and recall from \eqref{eq:bound_G12} that $\norm{G(z)}\lesssim 1$ with high probability. Hence, we can choose $\e>0$ small enough so that $\e\norm{\BB^2 G(z)}\leq 1/2$ with high probability for all $\abs{\omega}\leq \e$. This gives the bound
\begin{equation}\label{eq:bound_Gomega}
\sup_{\omega\in \C:|\omega|\le \e}\|G_\omega(z)\| =\OO(1)\quad \text{with high probability}.
\end{equation}

Recall that $\Pi$ is characterized by the equation $M(\Pi)$, with $M$ defined in \eqref{s-ceq}. We claim that the deterministic limit of $G_\omega$ is given by the solution of the following equation:
\begin{equation}\nonumber
\begin{aligned}
M_\omega(\Pi):=I+z\Pi+\Pi\mathcal{S}(\Pi)+\omega\BB^2\Pi=0.
\end{aligned}
\end{equation}
Explicitly, the solution can be written as $\Pi_\omega=\textup{diag}\{\Pi_{1,\omega},\Pi_{2,\omega}\}$, where $\Pi_{1,\omega}$ and $\Pi_{2,\omega}$ are respectively the $p\times p$ and $n\times n$ matrices defined by
$$\Pi_{1,\omega}=-\pb{z+zm_{2c,\omega}(z)A}^{-1},\quad\Pi_{2,\omega}=-\pb{z+zm_{1c,\omega}(z)B+\omega B}^{-1}.$$
Here, the quantities $(m_{1c,\omega}(z),m_{2c,\omega}(z))$ are defined as the unique solution of the following system of self-consistent equations satisfying $|m_{\al c,\omega}(z)-m_{\al c}(z)|\le C|\omega|$, $\al\in\{1,2\}$, for a constant $C>0$ independent of $\omega$:
\begin{equation}\label{m_12c,omega}
\begin{aligned}
m_{1c,\omega}(z)&=d_n\int\frac{x}{-(z+zm_{2c,\omega}(z)x)}\pi_A(\dd x),\\ 
m_{2c,\omega}(z)&=\int\frac{x}{-(z+zm_{1c,\omega}(z)x+\omega x)}\pi_B(\dd x).
\end{aligned}
\end{equation}
Then, we can prove the following anisotropic local law uniformly in $|\omega|\le \e$:
\begin{equation}\label{eq:locallaw_Gomega}
  G_\omega-\Pi_\omega=\OO_\prec(n^{-1/2}).
\end{equation} 
This local law can be proved in the same way as Theorem \ref{local_law}. In fact, the argument is simpler thanks to the operator norm bound \eqref{eq:bound_Gomega}, which replaces the uses of \eqref{ward2} and Lemma \ref{multipleG}. More precisely, 
\begin{align*}
\sum_{a}\left|(\Gamma_1G_\omega\Gamma_2)_{\uu a}\right|^{2}\prec 1,\ \ \abs{(\Ga_1G_\omega\Ga_2)_{\xx\yy}}\prec 1,\ \ \abs{(\Ga_1G_\omega\Ga_2G_\omega\Ga_3)_{\xx\yy}}\prec 1,\ \ \abs{(\Ga_1G_\omega\Ga_2G_\omega\Ga_3G_\omega\Ga_4)_{\xx\yy}}\prec 1.
\end{align*}
Moreover, the assumption $G_\omega-\Pi_\omega=\OO_\prec(\phi)$ in Theorem \ref{maintheorem} holds trivially with $\phi=1$. With these estimates, the proof of Theorem \ref{maintheorem} in Section \ref{sec:stoch_step} gives
\begin{equation}\nonumber
  M_\omega(G_\omega)=\OO_{\prec}(n^{-1/2}),\qquad |\underline{\Gamma M_\omega(G_\omega)}|\prec n^{-1}.\end{equation}
Finally, repeating the argument in Section \ref{sec:det_step} yields the local law \eqref{eq:locallaw_Gomega}. We omit the details. 

Combining the local law \eqref{eq:locallaw_Gomega} with the Cauchy integral formula, we obtain
\begin{equation}
\left.\partial_\omega\right|_{\omega=0}(G_\omega-\Pi_\omega)_{\vv_i\vv_i}=\frac{1}{2\pi\textup{i}}\oint_{\cal C}\frac{(G_\zeta-\Pi_\zeta)_{\vv_i\vv_i}}{\zeta^2}\dd\zeta \prec n^{-1/2},\label{eq:diff_G_Pi_derv}
\end{equation}
where the contour $\cal C$ is chosen to be the circle centered at $0$ with radius $\e/2$. It remains to evaluate $\left.\partial_\omega\right|_{\omega=0}(\Pi_\omega)_{\vv_i\vv_i}$, which depends on the $\omega$-derivatives of $m_{1c,\omega}(z)$ and $m_{2c,\omega}(z)$ at $\omega=0$. Differentiating the equations in \eqref{m_12c,omega} with respect to $\omega$ gives
\begin{equation}\label{dm_12c,omega}
\begin{aligned}
 \partial_\omega m_{1c,\omega}(z)&=d_n\int\frac{x^2z \cdot \partial_\omega m_{2c,\omega}(z)}{(z+zm_{2c,\omega}(z)x)^2}\pi_A(\dd x),\\
\partial_\omega m_{2c,\omega}(z)&=\int\frac{x^2z \cdot \partial_\omega m_{1c,\omega}(z)+x^2}{(z+zm_{1c,\omega}(z)x+\omega x)^2}\pi_B(\dd x).
\end{aligned}
\end{equation}
Evaluating at $\omega=0$ yields
\begin{equation}\label{dm_12c,omega_0}
\begin{aligned}
m_{1a}(z)&=d_n\int\frac{x^2zm_{2a}(z)}{(z+zm_{2c}(z)x)^2}\pi_A(\dd x)=z\gamma_1(z)m_{2a}(z),\\
m_{2a}(z)&=\int\frac{x^2zm_{1a}(z)+x^2}{(z+zm_{1c}(z)x)^2}\pi_B(\dd x)=z\gamma_2(z)m_{1a}(z)+\gamma_2(z),
\end{aligned}
\end{equation}
where we abbreviate $m_{1a}(z):=\partial_\omega \lvert_{\omega=0}m_{1c,\omega}(z)$ and $m_{2a}(z):=\partial_\omega \lvert_{\omega=0} m_{2c,\omega}(z)$, and recall that the functions $\gamma_1$ and $\gamma_2$ are defined in \eqref{def:alpha(z) beta(z) a(z) b(z) s4}.
Solving the linear system \eqref{dm_12c,omega_0} yields
$$m_{1a}(z)=\frac{z\gamma_1(z)\gamma_2(z)}{1-z^2\gamma_1(z)\gamma_2(z)},\quad m_{2a}(z)=\frac{\gamma_2(z)}{1-z^2\gamma_1(z)\gamma_2(z)},$$
where $m_{2a}$ agrees with the definition in (\ref{def:alpha(z) beta(z) a(z) b(z) s4}). We can now compute $\left.\partial_\omega\right|_{\omega=0}(\Pi_\omega)_{\vv_i\vv_i}$ as
\begin{align*}
\left.\partial_\omega\right|_{\omega=0}(\Pi_\omega)_{\vv_i\vv_i}&=\pB{\Pi_{1,\omega=0}\pb{zm_{2a}(z)A}\Pi_{1,\omega=0}}_{\vv_i^a\vv_i^a}=\frac{m_{2a}(z)\sigma_i^a}{z(1+m_{2c}(z)\sigma_i^a)^2}.
\end{align*}
Together with \eqref{eq:diff_G_Pi_derv}, this gives
\[(G\BB^2 G)_{\vv_i\vv_i} =\frac{m_{2a}(z)\sigma_i^a}{z(1+m_{2c}(z)\sigma_i^a)^2} +\OO_\prec(n^{-1/2}).\]
Combining this with the deterministic limit of $G_{\vv_i\vv_i}$ in \eqref{eq:Gvivi} and substituting into \eqref{h01}, we get
\begin{align}
\E h_i(0,1)&= 2(\ell-1) \frac{\sqrt{\sigma_i^a}}{\sqrt{n z}}\frac{m_{2a}(z)\sigma_i^a}{(1+m_{2c}(z)\sigma_i^a)^3} \E(\mathcal{P}_i^a)^{\ell-2} +\OO_\prec(n^{-1}).\label{final_h01}
\end{align}

\subsubsection{Estimate of $h_i(2,0)$}

Using \eqref{Gderivative}, we compute $h_i(2,0)$ as
\begin{align}
&h_i(2,0)= \sum_{j,\mu}\frac{\kappa_3(x_{j\mu})}{2}\vv_i(j)\qa{\partial_{j\mu}^2(\BB G)_{\mu\vv_i}}(\mathcal{P}_i^a)^{\ell-1}\label{h20}\\
&= z(\mathcal{P}_i^a)^{\ell-1}\frac{\kappa_3}{2}\sum_{j,\mu}\vv_i(j)\Bigl\{(\BB G\AA)_{\mu j}\qa{(\BB G\AA)_{\mu j}(\BB G)_{\mu\vv_i}+(\BB G\BB)_{\mu\mu}(\AA G)_{j\vv_i}}\nonumber\\
&\quad +\qa{(\BB G\AA)_{\mu j}(\BB G\AA)_{\mu j}+(\BB G\BB)_{\mu\mu}(\AA G\AA)_{jj}}(\BB G)_{\mu\vv_i}+ (\BB G\BB)_{\mu\mu}\qa{(\AA G\AA)_{jj}(\BB G)_{\mu\vv_i}+(\AA G\BB)_{j\mu}(\AA G)_{j\vv_i}}\nonumber\\
&\quad +\qa{(\BB G\AA)_{\mu j}(\BB G\BB)_{\mu\mu}+(\BB G\BB)_{\mu\mu}(\AA G\BB)_{j\mu}}(\AA G)_{j\vv_i}\Bigr\}\nonumber\\
&=z(\mathcal{P}_i^a)^{\ell-1}\kappa_3\sum_{j,\mu}\vv_i(j)\pB{(\BB G\AA)_{\mu j}^2(\BB G)_{\mu\vv_i}+2(\BB G\AA)_{\mu j}(\BB G\BB)_{\mu\mu}(\AA G)_{j\vv_i}+(\BB G\BB)_{\mu\mu}(\AA G\AA)_{jj}(\BB G)_{\mu\vv_i}}.\nonumber
\end{align} 
We estimate these terms using the following bounds, which follow directly from the operator norm bound \eqref{eq:bound_G12}. Under the setup of Lemma \ref{wardlemma}, for any deterministic unit vectors $\uu, \xx, \yy\in\mathbf S$, we have
\begin{align}\label{eq:bound_G_by_1}
\sum_{a}\left|(\Gamma_1G \Gamma_2)_{\uu a}\right|^{2}\prec 1,\ \ \abs{(\Ga_1G \Ga_2)_{\xx\yy}}\prec 1.
\end{align}
Combining \eqref{eq:bound_G_by_1} with the Cauchy-Schwarz inequality, we bound the first two terms in the last line of \eqref{h20} as follows:
\begin{align}
&  \kappa_3\sum_{j,\mu}\vv_i(j)(\BB G\AA)_{\mu j}^2(\BB G)_{\mu\vv_i}\prec n^{-3/2} \sum_j \abs{\vv_i(j)} \sum_\mu \absb{(\BB G \AA)_{\mu j}}^2 \prec n^{-3/2} \sum_j \abs{\vv_i(j)} \lesssim n^{-1},\label{eq:bound_h20_1}\\
&\kappa_3\sum_{j,\mu}\vv_i(j)(\BB G\AA)_{\mu j}(\BB G\BB)_{\mu\mu}(\AA G)_{j\vv_i} \prec n^{-3/2}\sum_{j}\absa{\vv_i(j)(\AA G)_{j\vv_i}}\cdot  \sqrt{n}\pbb{\sum_\mu |(\BB G\AA)_{\mu j}|^2}^{1/2} \nonumber\\
&\qquad \qquad\qquad \qquad\qquad \qquad\qquad \qquad \prec n^{-1} \pbb{\sum_{j}|(\AA G)_{j\vv_i}|^2}^{1/2} \prec n^{-1}.\label{eq:bound_h20_2}
\end{align}
To handle the last term on the right-hand side of \eqref{h20}, we first use the local law \eqref{aniso_outstrong} to write
\begin{equation}
(\BB G\BB)_{\mu\mu}=(\BB \Pi\BB)_{\mu\mu}+\OO_\prec(n^{-1/2}).\label{eq:bound_GBG1}
\end{equation} 
Moreover, choose a deterministic vector $\mathbf{b}\in\RR^{p+n}$ by setting $\mathbf{b}(i)=0$ for $1\le i \le p$ and $\mathbf{b}({\mu})=\overline{(\BB \Pi\BB)_{\mu\mu}}$ for $p+1\le \mu\le p+n$. This vector has $l^2$-norm of order $\OO(\sqrt{n})$. Note that $\vv_i$ has nonzero entries only in the $j$-indices, whereas $\BB \mathbf b$ has nonzero entries only in the $\mu$-indices. Hence, applying the local law \eqref{aniso_outstrong} again gives
\begin{equation}
  (\BB G)_{\mathbf b\vv_i} \prec n^{-1/2}\|\mathbf b\| \lesssim 1 .\label{eq:bound_GBG2}
\end{equation} 
Using \eqref{eq:bound_GBG1} and \eqref{eq:bound_GBG2}, together with \eqref{eq:bound_G_by_1} and the Cauchy-Schwarz inequality, we obtain the bound
\begin{align}
 \sum_\mu (\BB G\BB)_{\mu\mu}(\BB G)_{\mu\vv_i} &=  (\BB G)_{\mathbf b\vv_i} + \OO_\prec\pbb{n^{-1/2}\sum_\mu \absa{(\BB G)_{\mu\vv_i}}} \prec 1 + \pbb{\sum_\mu |(\BB G)_{\mu\vv_i}|^2}^{1/2} \prec 1.\label{eq:bound_GBG_BG}
\end{align}
Combining this bound with \eqref{eq:bound_G_by_1}, we control the last term on the right-hand side of \eqref{h20} as
\begin{align}
&\kappa_3\sum_{j,\mu}\vv_i(j)(\BB G\BB)_{\mu\mu}(\AA G\AA)_{jj}(\BB G)_{\mu\vv_i}\prec n^{-3/2}\sum_{j}|\vv_i(j)| \lesssim  n^{-1}.\label{eq:bound_h20_3}
\end{align}
Substituting \eqref{eq:bound_h20_1}, \eqref{eq:bound_h20_2}, and \eqref{eq:bound_h20_3} into \eqref{h20}, we obtain
\begin{equation}\label{final_h20}
\E h_i(2,0) \prec n^{-1}.
\end{equation}

\subsubsection{Estimate of $h_i(1,1)$}

Using \eqref{Gderivative} and \eqref{Pderivative}, we compute $h_i(1,1)$ as
\begin{align}\label{h11}
h_i(1,1)=&~\sum_{j,\mu}\kappa_3(x_{j\mu})\vv_i(j)\partial_{j\mu}(\BB G)_{\mu\vv_i}\cdot \partial_{j\mu}(\mathcal{P}_i^a)^{\ell-1}\\
=&~2(\ell-1)z^{2} (\mathcal{P}_i^a)^{\ell-2}\cdot \sqrt{n} \kappa_3\sum_{j,\mu}\vv_i(j)(G\AA)_{\vv_i j}(\BB G)_{\mu\vv_i} \pb{(\BB G\AA)_{\mu j}(\BB G)_{\mu \vv_i}+(\BB G\BB)_{\mu\mu}(\AA G)_{j\vv_i}}. \nonumber
\end{align}
Using \eqref{eq:bound_G_by_1} and \eqref{eq:bound_GBG_BG}, we bound the two terms on the right-hand side of \eqref{h11} as follows:
\begin{align*}
&\sqrt{n}\kappa_3 \sum_{j,\mu}\vv_i(j)(G\AA)_{\vv_i j}(\BB G)_{\mu\vv_i}^2(\BB G\AA)_{\mu j}  \prec n^{-1} \sum_j\absb{\vv_i(j)(G\AA)_{\vv_i j}}  \sum_\mu|(\BB G)_{\mu\vv_i}|^2\\
&\qquad \qquad \prec n^{-1} \sum_j\absb{\vv_i(j)(G\AA)_{\vv_i j}} \lesssim n^{-1} \pbb{\sum_j|(G\AA)_{\vv_i j}|^2}^{1/2} \prec n^{-1},\\
&\sqrt{n}\kappa_3 \sum_{j,\mu}\vv_i(j)(G\AA)_{\vv_i j}(\AA G)_{j\vv_i} \sum_\mu (\BB G\BB)_{\mu\mu} (\BB G)_{\mu\vv_i}\prec n^{-1} \sum_{j,\mu}|(G\AA)_{\vv_i j}|^2 \prec n^{-1} .
\end{align*}
Therefore,
\begin{equation}\label{final_h11}
\E h_i(1,1) \prec n^{-1}.
\end{equation}

\subsubsection{Estimate of $h_i(0,2)$}

Using \eqref{Gderivative} and \eqref{Pderivative}, we compute $h_i(0,2)$ as
\begin{align}
h_i(0,2)&= \sum_{j,\mu}\frac{\kappa_3(x_{j\mu})}{2}\vv_i(j)(\BB G)_{\mu\vv_i}\partial_{j\mu}^2(\mathcal{P}_i^a)^{\ell-1} \nonumber\\
&=2(\ell-1)(\ell-2)z^3(\mathcal{P}_i^a)^{\ell-3}\cdot n\kappa_3\sum_{j,\mu}\vv_i(j)(\BB G)_{\mu\vv_i}(G\AA)_{\vv_i j}^2(\BB G)_{\mu\vv_i}^2\label{h02}\\
&+(\ell-1) z^{2}(\mathcal{P}_i^a)^{\ell-2}\cdot \sqrt{n}\kappa_3\sum_{j,\mu}\vv_i(j)(\BB G)_{\mu\vv_i} (G\AA)_{\vv_i j}\qb{(\BB G\AA)_{\mu j}(\BB G)_{\mu\vv_i}+(\BB G\BB)_{\mu\mu}(\AA G)_{j\vv_i}}\nonumber\\
&+(\ell-1) z^{2}(\mathcal{P}_i^a)^{\ell-2}\cdot \sqrt{n}\kappa_3\sum_{j,\mu}\vv_i(j)(\BB G)_{\mu\vv_i}\qb{(G\AA)_{\vv_i j}(\BB G\AA)_{\mu j}+(G\BB)_{\vv_i\mu}(\AA G\AA)_{jj}}(\BB G)_{\mu\vv_i} .\nonumber
\end{align}
Using \eqref{eq:bound_G_by_1} and \eqref{eq:bound_GBG_BG}, we bound the terms on the right-hand side of \eqref{h02} as follows:
\begin{align*}
&n\kappa_3 \sum_{j,\mu}\vv_i(j) (G\AA)_{\vv_i j}^2(\BB G)_{\mu\vv_i}^3 \prec \sqrt{n}\kappa_3 \sum_{j}|(G\AA)_{\vv_i j}^2|\prec \sqrt{n}\kappa_3 \lesssim n^{-1},\\
&\sqrt{n}\kappa_3 \sum_{j,\mu}\vv_i(j) (G\AA)_{\vv_i j}(\BB G\AA)_{\mu j}(\BB G)_{\mu\vv_i}^2 \prec \sqrt{n}\kappa_3 \sum_{j}|\vv_i(j)| |(G\AA)_{\vv_i j}| \prec \sqrt{n}\kappa_3 \lesssim n^{-1},\\
&\sqrt{n}\kappa_3\sum_{j}\vv_i(j)(G\AA)_{\vv_i j}(\AA G)_{j\vv_i} \sum_\mu (\BB G\BB)_{\mu\mu}(\BB G)_{\mu\vv_i} \prec \sqrt{n}\kappa_3\sum_{j}|(G\AA)_{\vv_i j}|^2 \prec \sqrt{n}\kappa_3 \lesssim n^{-1},\\
&\sqrt{n}\kappa_3 \sum_{j,\mu}\vv_i(j)(\BB G)_{\mu\vv_i}^2(G\BB)_{\vv_i\mu}(\AA G\AA)_{jj} \prec \kappa_3 \sum_{j}|\vv_i(j)| \prec \sqrt{n}\kappa_3 \lesssim n^{-1}, 
\end{align*}
where, in deriving the first and fourth bounds, we also used
\begin{align}\label{eq:bound_GBG_muvi}
  (\BB G)_{\mu\vv_i}\prec n^{-1/2}
\end{align} 
which follows from the local law \eqref{aniso_outstrong}. Therefore,
\begin{equation}\label{final_h02}
\E h_i(0,2) \prec n^{-1}.\end{equation}

\subsubsection{Estimate of $h_i(3,0)$}

We next bound $h_i(3,0)$, which is defined as
\begin{align*}
 h_i(3,0)&=\frac{1}{6} \sum_{j,\mu} {\kappa_4(x_{j\mu})} \vv_i(j)\partial_{j\mu}^3(\BB G)_{\mu\vv_i}\cdot (\mathcal{P}_i^a)^{\ell-1}.
\end{align*}
Using \eqref{Gderivative}, we can write $\partial_{j\mu}^3(\BB G)_{\mu\vv_i}$ as a linear combination of eight terms, each of which is a product of three resolvent factors, with at least one factor of the form $(\BB G)_{\mu\vv_i}$ or $(G\AA)_{\vv_i j}$. Therefore, by \eqref{eq:bound_G_by_1} and the Cauchy-Schwarz inequality, we can bound $\E h_i(3,0)$ by
\begin{align}\nonumber
\E h_i(3,0) &\prec \kappa_4\E\sum_{j,\mu}|\vv_i(j)| \pa{|(\BB G)_{\mu\vv_i}|+|(G\AA)_{\vv_i j}|} \\
&\prec n^{-2} \pbb{\sum_{j,\mu}|\vv_i(j)|^2}^{1/2}\pbb{\E\sum_{j,\mu}\pa{|(\BB G)_{\mu\vv_i}|+|(G\AA)_{\vv_i j}|}^2}^{1/2} \prec n^{-2}\cdot \sqrt{n}\cdot \sqrt{n} = n^{-1}.\label{final_h30} 
\end{align}

\subsubsection{Estimate of $h_i(2,1)$}
With \eqref{Pderivative}, we compute $h_i(2,1)$ as
\begin{align}
h_i(2,1)&=\sum_{j,\mu}\frac{\kappa_4(x_{j\mu})}{2}\vv_i(j)\partial_{j\mu}^2(\BB G)_{\mu\vv_i}\partial_{j\mu}(\mathcal{P}_i^a)^{\ell-1}\nonumber\\
&=-(\ell-1)z^{3/2}(\mathcal{P}_i^a)^{\ell-2}\cdot  \sqrt{n}\kappa_4\sum_{j,\mu}\vv_i(j)(G\AA)_{\vv_i j}(\BB G)_{\mu\vv_i}\partial_{j\mu}^2(\BB G)_{\mu\vv_i}.\label{h21}
\end{align}
Using the bound $\partial_{j\mu}^2(\BB G)_{\mu\vv_i}=\OO_\prec(1)$, the Cauchy-Schwarz inequality, and \eqref{eq:bound_G_by_1}, we get
\begin{align}
\E h_i(2,1)&\prec n^{-3/2}\sum_{j,\mu}|\vv_i(j)| |(G\AA)_{\vv_i j}||(\BB G)_{\mu\vv_i}| \prec n^{-3/2}\sum_{j}|\vv_i(j)| |(G\AA)_{\vv_i j}| \cdot \sqrt{n}\pbb{\sum_\mu |(\BB G)_{\mu\vv_i}|^2}^{1/2} \nonumber\\
&\prec n^{-3/2}\pbb{\sum_{j}|(G\AA)_{\vv_i j}|^2}^{1/2}\prec n^{-3/2}. \label{final_h21}
\end{align}

\subsubsection{Estimate of $h_i(1,2)$}

Using \eqref{Gderivative} and \eqref{Pderivative}, we compute $h_i(1,2)$ as
\begin{align}
&h_i(1,2)= \sum_{j,\mu}\frac{\kappa_4(x_{j\mu})}{2}\vv_i(j)\partial_{j\mu}(\BB G)_{\mu\vv_i}\cdot \partial_{j\mu}^2(\mathcal{P}_i^a)^{\ell-1}\label{h12}\\
&=-2(\ell-1)(\ell-2) z^{7/2}(\mathcal{P}_i^a)^{\ell-3} \cdot n\kappa_4 \sum_{j,\mu}\vv_i(j)\pb{(\BB G\AA)_{\mu j}(\BB G)_{\mu\vv_i}+(\BB G\BB)_{\mu\mu}(\AA G)_{j\vv_i}}(G\AA)_{\vv_i j}^2(\BB G)_{\mu\vv_i}^2\nonumber\\
&\quad -(\ell-1) z^{5/2} (\mathcal{P}_i^a)^{\ell-2}\cdot \sqrt{n}\kappa_4\sum_{j,\mu}\vv_i(j)\pb{(\BB G\AA)_{\mu j}(\BB G)_{\mu\vv_i}+(\BB G\BB)_{\mu\mu}(\AA G)_{j\vv_i}} \nonumber\\
&\quad \times \Bigl[(G\AA)_{\vv_i j}\pb{(\BB G\AA)_{\mu j}(\BB G)_{\mu\vv_i}+(\BB G\BB)_{\mu\mu}(\AA G)_{j\vv_i}}+\pb{(G\AA)_{\vv_i j}(\BB G\AA)_{\mu j}+(G\BB)_{\vv_i\mu}(\AA G\AA)_{jj}}(\BB G)_{\mu\vv_i}\Bigr].\nonumber
\end{align}
Combining \eqref{eq:bound_G_by_1} with the Cauchy-Schwarz inequality, we bound the terms on the right-hand side of \eqref{h12} as follows:
\begin{align*}
&n\kappa_4 \sum_{j,\mu}\vv_i(j)\pb{(\BB G\AA)_{\mu j}(\BB G)_{\mu\vv_i}+(\BB G\BB)_{\mu\mu}(\AA G)_{j\vv_i}}(G\AA)_{\vv_i j}^2(\BB G)_{\mu\vv_i}^2  \prec n^{-1} \sum_j|(G\AA)_{\vv_i j}^2| \sum_\mu(\BB G)_{\mu\vv_i}^2 \prec n^{-1},\\
&\sqrt{n}\kappa_4 \sum_{j,\mu}\vv_i(j)(\BB G\AA)_{\mu j}(\BB G)_{\mu\vv_i}(G\AA)_{\vv_i j}\pb{(\BB G\AA)_{\mu j}(\BB G)_{\mu\vv_i}+(\BB G\BB)_{\mu\mu}(\AA G)_{j\vv_i}} \\
&\quad \prec n^{-3/2}{\sum_j\absb{\vv_i(j)(G\AA)_{\vv_i j}}}\cdot {\sum_\mu\absb{(\BB G\AA)_{\mu j}(\BB G)_{\mu\vv_i}}} \prec n^{-3/2}{\sum_j\absb{\vv_i(j)(G\AA)_{\vv_i j}}} \prec n^{-3/2},\\
&\sqrt{n}\kappa_4 \sum_{j,\mu}\vv_i(j)\pb{(\BB G\AA)_{\mu j}(\BB G)_{\mu\vv_i}+(\BB G\BB)_{\mu\mu}(\AA G)_{j\vv_i}}(G\BB)_{\vv_i\mu}(\AA G\AA)_{jj}(\BB G)_{\mu\vv_i} \\
&\quad \prec n^{-3/2} \sum_j\absb{\vv_i(j)}  \sum_\mu |(\BB G)_{\mu\vv_i}|^2 \prec n^{-3/2} \sum_j\absb{\vv_i(j)} \lesssim n^{-1},\\
&\sqrt{n}\kappa_4 \sum_{j,\mu}\vv_i(j)(\BB G\BB)_{\mu\mu}(\AA G)_{j\vv_i}(G\AA)_{\vv_i j}(\BB G\AA)_{\mu j}(\BB G)_{\mu\vv_i} \\
&\quad \prec n^{-3/2} \sum_j\absb{(G\AA)_{\vv_i j}}^2 \sum_\mu\absb{(\BB G\AA)_{\mu j}(\BB G)_{\mu\vv_i}} \prec n^{-3/2} \sum_j\absb{(G\AA)_{\vv_i j}}^2 \prec n^{-3/2}.
\end{align*}
The only remaining term is
\begin{equation}\nonumber
-(\ell-1)z^{5/2} (\mathcal{P}_i^a)^{\ell-2}\cdot \sqrt{n}\kappa_4\sum_{j,\mu}\vv_i(j)(\BB G\BB)_{\mu\mu}^2(\AA G)^2_{j\vv_i}(G\AA)_{\vv_i j}.
\end{equation}
We apply the local law \eqref{aniso_outstrong} to replace $G$ by $\Pi$ in the above expression. This gives
\begin{align}
&-(\ell-1)z^{5/2} (\mathcal{P}_i^a)^{\ell-2}\cdot \sqrt{n}\kappa_4\sum_{j}\vv_i(j)\qa{(\Pi\AA)_{\vv_i j}+\OO_\prec(n^{-1/2})}^3\sum_\mu \qa{(\BB \Pi\BB)_{\mu\mu}^2+\OO_\prec(n^{-1/2})} \nonumber\\
=& -(\ell-1)z^{5/2} (\mathcal{P}_i^a)^{\ell-2}\cdot \sqrt{n}\kappa_4\qbb{\sum_{j}\vv_i(j)(\Pi\AA)_{\vv_i j}^3+\OO_\prec(1)}\qbb{\sum_\mu (\BB \Pi\BB)_{\mu\mu}^2 +\OO_\prec(n^{1/2})}\nonumber\\
=& (\ell-1) (\mathcal{P}_i^a)^{\ell-2}\cdot \sqrt{n}\kappa_4\sum_{j}\frac{(\sigma_i^a)^{3/2}\vv_i(j)^4}{\sqrt{z}[1+m_{2c}(z)\sigma_i^a]^3} \sum_\mu (\BB \Pi\BB)_{\mu\mu}^2 +\OO_\prec(n^{-1}).
\label{h12contributionterm}
\end{align} 
Altogether, we obtain
\begin{align}\label{final_h12}
&\E h_i(1,2)= (\ell-1)n^{3/2}\kappa_4\sum_{j}\frac{(\sigma_i^a)^{3/2}\vv_i(j)^4}{\sqrt{z}[1+m_{2c}(z)\sigma_i^a]^3}\cdot \frac{1}{n} \sum_\mu (\BB \Pi\BB)_{\mu\mu}^2  \cdot \E[(\mathcal{P}_i^a)^{\ell-2}]+\OO_\prec(n^{-1}).
\end{align}

\subsubsection{Estimate of $h_i(0,3)$}

A direct computation using \eqref{Gderivative} and \eqref{Pderivative} gives
\begin{align*}
h_i(0,3)=&-\frac{4}{3}(\ell-1)(\ell-2)(\ell-3) z^{9/2} (\mathcal{P}_i^a)^{\ell-4}\cdot n^{3/2} \kappa_4 \sum_{j,\mu}\vv_i(j)(\BB G)_{\mu\vv_i}^4(G\AA)_{\vv_i j}^3\\
&-\frac{2}{3}(\ell-1)(\ell-2) z^{7/2} (\mathcal{P}_i^a)^{\ell-3}\cdot n \kappa_4\sum_{j,\mu}\vv_i(j)(\BB G)_{\mu\vv_i}^2(G\AA)_{\vv_i j} \cal F_1\\
&-\frac{1}{3}(\ell-1) z^{5/2}(\mathcal{P}_i^a)^{\ell-2} \cdot \sqrt{n}\kappa_4\sum_{j,\mu}\vv_i(j)(\BB G)_{\mu\vv_i} \cal F_2.
\end{align*}
Here $\cal F_1$ is a linear combination of products of three resolvent factors and is bounded by $\OO_\prec(1)$, and the term $\cal F_2$ is a linear combination of products of four resolvent factors, with at least one factor of the form $(\BB G)_{\mu\vv_i}$ or $(G\AA)_{\vv_i j}$. Therefore, by \eqref{eq:bound_G_by_1} and the Cauchy-Schwarz inequality, we bound the second and third terms on the right-hand side of the above expression as follows:
\begin{align*}
  &n \kappa_4\sum_{j,\mu}\vv_i(j)(\BB G)_{\mu\vv_i}^2(G\AA)_{\vv_i j} \cal F_1 \prec  n^{-1}\sum_j  \absa{\vv_i(j)  (G\AA)_{\vv_i j} } \sum_\mu |(\BB G)_{\mu\vv_i}|^2 \prec n^{-1}\sum_j  \absa{\vv_i(j)  (G\AA)_{\vv_i j} } \prec n^{-1},\\ 
  &\sqrt{n}\kappa_4\sum_{j,\mu}\vv_i(j)(\BB G)_{\mu\vv_i} \cal F_2 \prec n^{-3/2}\sum_{j,\mu}|\vv_i(j)| \left(\left|(\BB G)_{\mu\vv_i}\right|^2+\left|(\BB G)_{\mu\vv_i}\right|\left|(G\AA)_{\vv_i j}\right|\right) \\
  &\quad \prec n^{-3/2}\pbb{\sum_{j}|\vv_i(j)| +\sqrt{n}\sum_{j}|\vv_i(j)|\left|(G\AA)_{\vv_i j}\right| } \prec n^{-1}.
\end{align*}
For the first term, \eqref{eq:bound_G_by_1} and \eqref{eq:bound_GBG_muvi} give
\begin{align*}
  n^{3/2} \kappa_4 \sum_{j,\mu}\vv_i(j)(\BB G)_{\mu\vv_i}^4(G\AA)_{\vv_i j}^3 \prec n^{-1/2}\sum_j |(G\AA)_{\vv_i j}|^2 \cdot \sum_\mu |(\BB G)_{\mu\vv_i}|^4 \prec n^{-3/2}\sum_j |(G\AA)_{\vv_i j}|^2\prec n^{-3/2}.
\end{align*}
Altogether, we obtain
\begin{align}\label{final_h03}
&\E h_i(0,3) \prec n^{-1}.
\end{align}

\subsubsection{Estimate of the remainder}

Finally, we control the remainder term $\mathcal{R}_i:= \sum_{j,\mu}\vv_i(j)\mathcal{R}_i^{j\mu}$, where $\mathcal{R}_i^{j\mu}$ satisfies \eqref{remainder0}.
As for $G$ and ${\mathcal{P}}_i^a$, \smash{$\wt G$ and $\wt{\mathcal{P}}_i^a$} satisfy the following deterministic bounds at $z$:
\[\norm{\wt G(z)}\lesssim (\Im z^{1/2})^{-1}\sim n^{2},\quad |\wt{\mathcal{P}}_i^a(z)| \lesssim n^{5/2}.\] 
These bounds imply
\[\sup_{\xi_{j\mu}\in \R}\absb{\partial_{j\mu}^4\pb{(\BB \wt G)_{\mu\vv_i}(\wt{\mathcal{P}}_i^a)^{\ell-1}}} \lesssim n^{10+\frac{5}{2}(\ell-1)}.\] 
On the other hand, by the high-moment assumption \eqref{eq_highmoment} on $x_{j\mu}$, we have $\expect\qa{\abs{x_{j\mu}}^5\mathds{1}(|x_{j\mu}|>n^{-1/2+\e})}\le n^{-D}$ for any large constant $D>0$. Hence, the second term in \eqref{remainder0} is bounded by $n^{-1}$, provided $D$ is chosen large enough depending on $\ell$. For the first term in \eqref{remainder0}, using $\expect\qa{\abs{x_{j\mu}}^5}\lesssim n^{-5/2}$, we get
\begin{align}
  \cal R_i &\lesssim n^{-5/2}\E\sum_{j,\mu}|\vv_i(j)|\sup_{\abs{\xi_{j\mu}}\leq n^{-1/2+\e}}\absb{\partial_{j\mu}^4\pb{(\BB \wt G)_{\mu\vv_i}(\wt{\mathcal{P}}_i^a)^{\ell-1}}}+\OO(n^{-1})\lesssim \sum_{t=0}^4 \E\cal R_i(t)+\OO(n^{-1}), \label{remainder}
\end{align}
where, for $t\in\{0,1,2,3,4\}$, 
\[\cal R_i(t):= n^{-5/2}\sum_{j,\mu}|\vv_i(j)|\sup_{\abs{\xi_{j\mu}}\leq n^{-1/2+\e}}\absb{\partial_{j\mu}^{4-t} (\BB \wt G)_{\mu\vv_i}}\absb{\partial_{j\mu}^{t}(\wt{\mathcal{P}}_i^a)^{\ell-1}}. \]
To control these terms, we use the estimate
\begin{align}\label{eq:replace_GbyG}
\sup_{\abs{\xi_{j\mu}}\leq n^{-1/2+\e}}\|\wt G - G\| \lesssim  n^{-1/2+\e} \quad \text{with high probability},  
\end{align}
which follows from the identity $\wt G=G(z)[1-(x_{j\mu}-\xi_{j\mu})\DD G(z)]^{-1}$, with $\DD$ defined in \eqref{defn_DD}, together with the operator norm bound \eqref{eq:bound_G12} for $G$. By \eqref{eq:replace_GbyG}, we also obtain that, for any $0\le k\le 4$,
\begin{align}\label{eq:bound_wtP-P}
\sup_{\abs{\xi_{j\mu}}\leq n^{-1/2+\e}}\absa{\partial_{j\mu}^k\wt{\mathcal{P}}_i^a-\partial_{j\mu}^k{\mathcal{P}}_i^a}\prec n^{\e},\quad \sup_{\abs{\xi_{j\mu}}\leq n^{-1/2+\e}}\absa{\wt{\mathcal{P}}_i^a}\prec n^{\e},\quad \sup_{\abs{\xi_{j\mu}}\leq n^{-1/2+\e}}\absa{\partial_{j\mu} \wt{\mathcal{P}}_i^a}\prec n^{\e},
\end{align}
where for the last bound we also used \eqref{Pderivative} and \eqref{eq:bound_GBG_muvi} to bound $\partial_{j\mu} {\mathcal{P}}_i^a$ by $\OO_\prec(1)$.

With \eqref{eq:replace_GbyG} and \eqref{eq:bound_G12}, for $t\in\{0,1\}$ we have $\partial_{j\mu}^{4-t} (\BB \wt G)_{\mu\vv_i}\prec 1$ and $\partial_{j\mu}^{t}(\wt{\mathcal{P}}_i^a)^{\ell-1}\prec n^{(\ell-1)\e}$ uniformly in $\xi_{j\mu}$. Thus, for $t\in\{0,1\}$,
\begin{align}\label{Ri01}
\cal R_i(t)\prec  n^{-5/2 +(\ell-1)\e} \sum_{j,\mu}|\vv_i(j)| \lesssim n^{-1 +(\ell-1)\e}. 
\end{align}
For $t\in \{2,3\}$, \smash{$\partial_{j\mu}^{4-t} (\BB \wt G)_{\mu\vv_i}$} can be written as a linear combination of products of $(5-t)$ resolvent factors, with at least one factor of the form \smash{$(\BB \wt G)_{\mu\vv_i}$} or \smash{$(\AA \wt G)_{j\vv_i}$}. Hence, by \eqref{eq:replace_GbyG},
\[ \sup_{\abs{\xi_{j\mu}}\leq n^{-1/2+\e}}\absb{\partial_{j\mu}^{4-t} (\BB \wt G)_{\mu\vv_i}}\prec \absb{(\BB G)_{\mu\vv_i}}+\absb{(\AA G)_{j\vv_i}}+\OO_\prec(n^{-1/2+\e}). \]
Moreover, by \eqref{eq:bound_wtP-P}, we obtain
\[\partial_{j\mu}^{t}(\wt{\mathcal{P}}_i^a)^{\ell-1} \prec n^{(\ell-1)\e} + n^{1/2+(\ell-2)\e} \lesssim n^{1/2+(\ell-2)\e}.\]
Here the factor $n^{(\ell-1)\e}$ comes from products involving only $\wt{\mathcal{P}}_i^a$ and its first-order derivative, while the factor $n^{1/2+(\ell-2)\e}$ comes from products involving a higher-order derivative \smash{$\partial_{j\mu}^{r}\wt{\mathcal{P}}_i^a$} with $r\ge 2$, which is bounded by $n^{1/2}$ since $\partial_{j\mu}^{r}{\mathcal{P}}_i^a\prec n^{1/2}$. Thus, for $t\in\{2,3\}$,
\begin{align}
\cal R_i(t) &\prec  n^{-2 +(\ell-2)\e} \E\sum_{j,\mu}|\vv_i(j)|\qa{\abs{(\BB G)_{\mu\vv_i}}+\abs{(\AA G)_{j\vv_i}}+\OO_\prec(n^{-1/2+\e})} \label{Ri23}\\
&\prec n^{-2 +(\ell-2)\e} \pbb{\sum_{j}|\vv_i(j)| \sum_\mu \abs{(\BB G)_{\mu\vv_i}} +n \sum_{j}|\vv_i(j)|\abs{(\AA G)_{j\vv_i}}}+n^{-1 +(\ell-1)\e} \prec n^{-1 +(\ell-1)\e}, \nonumber
\end{align}
where in the last step we used the Cauchy-Schwarz inequality and \eqref{eq:bound_G_by_1} again.

Finally, we consider $\cal R_i(4)$. We expand $\partial_{j\mu}^{4}(\wt{\mathcal{P}}_i^a)^{\ell-1}$ as a sum of products of $\wt{\mathcal{P}}_i^a$ and its derivatives. If a product involves at most one higher-order derivative \smash{$\partial_{j\mu}^{r}\wt{\mathcal{P}}_i^a$} with $r\ge 2$, then the corresponding term is bounded by $\OO_\prec(n^{-1 +(\ell-1)\e})$, as in the cases $t\in \{2,3\}$. Thus, it remains only to consider the term $(\partial_{j\mu}^{2}\wt{\mathcal{P}}_i^a)^2 (\wt{\mathcal{P}}_i^a)^{\ell-3}$:
\begin{align}
&  \cal R_i(4)\prec n^{-1 +(\ell-1)\e}+ n^{-5/2}\sum_{j,\mu}|\vv_i(j)|\sup_{\abs{\xi_{j\mu}}\leq n^{-1/2+\e}}\absb{(\BB \wt G)_{\mu\vv_i}}\absb{(\partial_{j\mu}^{2}\wt{\mathcal{P}}_i^a)^2 (\wt{\mathcal{P}}_i^a)^{\ell-3}} \nonumber\\
&  \prec n^{-1 +(\ell-1)\e}+ n^{-5/2+(\ell-3)\e}\sum_{j,\mu}|\vv_i(j)|\sup_{\abs{\xi_{j\mu}}\leq n^{-1/2+\e}}\qa{\absb{(\BB G)_{\mu\vv_i}}+\OO_\prec(n^{-1/2+\e})}\absb{(\partial_{j\mu}^{2}\wt{\mathcal{P}}_i^a)^2} \nonumber\\
&  \prec n^{-1 +(\ell-1)\e}+ n^{-3/2+(\ell-3)\e}\sum_{j,\mu}|\vv_i(j)|\qa{\absb{(\BB G)_{\mu\vv_i}}+\OO_\prec(n^{-1/2+\e})}\qa{\absb{(\BB G)_{\mu\vv_i}}+\absb{(\AA G)_{j\vv_i}}+\OO_\prec(n^{-1/2+\e})}^2 \nonumber\\
&  \prec n^{-1 +(\ell-1)\e}+ n^{-3/2+(\ell-3)\e} \sum_{j,\mu}|\vv_i(j)|\absb{(\BB G)_{\mu\vv_i}}^2+n^{-3/2+(\ell-3)\e}\sum_{j,\mu}|\vv_i(j)|\absb{(\BB G)_{\mu\vv_i}}\absb{(\AA G)_{j\vv_i}}   \nonumber\\ 
&\quad  + n^{-2+(\ell-2)\e}\sum_{j,\mu}|\vv_i(j)|\qa{\absb{(\BB G)_{\mu\vv_i}}^2+\absb{(\BB G)_{\mu\vv_i}}\absb{(\AA G)_{j\vv_i}}+\OO_\prec(n^{-1+2\e})} \nonumber\\
&\prec  n^{-1 +(\ell-1)\e}+ n^{-1+(\ell-3)\e} + n^{-3/2+ \ell \e} \lesssim n^{-1 +(\ell-1)\e}.\label{Ri4}
\end{align}
Here, the second step uses \eqref{eq:replace_GbyG} and \eqref{eq:bound_wtP-P}. The third step uses \eqref{eq:replace_GbyG} and the fact that $\partial_{j\mu}^{2}\wt{\mathcal{P}}_i^a$ can be written as $\sqrt{n}$ times a product of three resolvent factors, at least one of which is of the form \smash{$(\BB \wt G)_{\mu\vv_i}$} or \smash{$(\AA \wt G)_{j\vv_i}$}. The fifth step uses the Cauchy-Schwarz inequality and \eqref{eq:bound_G_by_1}. 

Substituting \eqref{Ri01}, \eqref{Ri23}, and \eqref{Ri4} into \eqref{remainder}, we obtain
\begin{equation}\label{final_Ri}
  \cal R_i \prec n^{-1+(\ell-1)\e}.
\end{equation}

Finally, we combine (\ref{expansion}) with \eqref{final_h10}, \eqref{final_h01}, \eqref{final_h20}, \eqref{final_h11}, \eqref{final_h02}, \eqref{final_h30}, \eqref{final_h21}, \eqref{final_h12}, \eqref{final_h03}, and \eqref{final_Ri}. Choosing $\e$ small enough so that $(\ell-1)\e\le 1/4$, we obtain
\begin{align}\label{recursive_moment_z}
  \expect[(\mathcal{P}_i^a)^\ell]=&~(\ell-1)\frac{(\sigma_i^a)^2}{(1+m_{2c}(z)\sigma_i^a)^4} \qbb{2m_{2a}(z) + n^{2}\kappa_4\sum_{j} \vv_i(j)^4 \cdot \frac{1}{n} \sum_\mu (\BB \Pi\BB)_{\mu\mu}^2 }\E[(\mathcal{P}_i^a)^{\ell-2}]\\
  &~+\OO_\prec(n^{-1/6}). \nonumber
\end{align}
Recall that $k_4=n^2\kappa_4$ and  $\sum_{j} \vv_i(j)^4=s_4(\vv_i^a)$. Moreover, by \eqref{eq:diff_G_m}, we have
\begin{equation}
\label{eq:Pivivi}
\frac{1}{1+m_{2c}(z)\sigma_i^a}=\frac{1}{1+m_{2c}(\theta_i^a)\sigma_i^a} +\OO(n^{-2}) = \frac{1+d_i^a}{d_i^a}+ \OO(n^{-2}).
\end{equation}
Similarly, we have $m_{2a}(z)=m_{2a}(\theta_i^a)+\OO(n^{-2})$ and 
\[\frac{1}{n} \sum_\mu (\BB \Pi(z)\BB)_{\mu\mu}^2=\frac{1}{n} \sum_\mu \left(z^{-1}B(1+m_{1c}(z)B)^{-1}\right)_{\mu\mu}^2=b(z)=b(\theta_i^a) +\OO(n^{-2}).\]
Thus, \eqref{recursive_moment_z} can be rewritten in the form \eqref{expansion result}, which proves Lemma \ref{recursive moment estimate}.

\bibliographystyle{alpha}
\bibliography{sample}

\end{document}